\documentclass[11pt,a4paper,reqno]{amsart}
\usepackage{amssymb,amsfonts,amsmath,bm}
\usepackage{psfrag,graphicx,color}
\usepackage{mathrsfs}
\usepackage{graphics}
\usepackage{subfigure}
\usepackage{enumerate,enumitem}
\usepackage{pgfplots}
\usepackage{cite}
\usepackage{diagbox}
\usepackage{hyperref}

\allowdisplaybreaks

\usepackage{booktabs}
\usepackage{multirow}

\usepackage[margin=1.05in]{geometry}

\DeclareGraphicsExtensions{.eps}
\usepackage{amssymb,amsmath,graphicx,amsfonts,euscript,mathrsfs}
\usepackage{hyperref}

\usepackage{titlesec} 
\titleformat{\section}{\vskip10pt\large\bfseries}{\thesection.}{0.5em}{\centering\vspace{5pt}}
\titleformat{\subsection}{\vskip10pt\normalsize\bfseries}{\thesubsection.}{0.5em}{}

\newtheorem{theorem}{Theorem}
\newtheorem{lemma}[theorem]{Lemma}
\newtheorem{proposition}{Proposition}

\newtheorem{remark}{Remark}
\theoremstyle{definition}

\newtheorem{assumption}[theorem]{Assumption}

\def\R{\mathbb{R}}
\def\E{\mathbb{E}}

\def\T{\mathbb{T}}
\def\I{\mathbb{I}}
\def\D{\mathbb{D}}
\def\C{\mathbb{C}}
\def\P{\mathbb{P}}
\def\B{\mathcal{B}}
\def\F{\mathcal{F}}
\def\L{\mathcal{L}}
\def\DD{\mathcal{D}}

\def\DH{\dot{H}}
\def\d{\mathrm{d}}

\def\n{{\rm n}}

\definecolor{mygreen}{RGB}{34,139,34}

\def\endproof{\qed}

\numberwithin{equation}{section}
\numberwithin{theorem}{section}
\numberwithin{remark}{section}
\numberwithin{proposition}{section}
\numberwithin{exercise}{section}

\begin{document}

\title[]{Optimal analysis of finite element methods \\
for the stochastic Stokes equations}

\author[]{\,\,Buyang Li}
\address{\hspace*{-12pt}Buyang Li and Shu Ma: 
Department of Applied Mathematics, The Hong Kong Polytechnic University,
Hong Kong. {\it E-mail address}: {\tt buyang.li@polyu.edu.hk} and {\tt maisie.ma@connect.polyu.hk}}

\author[]{\,\,Shu Ma}

\author[]{\,\,Weiwei Sun}
\address{\hspace*{-12pt}Weiwei Sun: Division of Science and Technology, BNU-HKBU United International College; 
Advanced Institute of Natural Science, Beijing Normal University at Zhuhai; Zhuhai, P.R. China. {\it E-mail address}: {\tt maweiw@uic.edu.cn }}

\subjclass[2010]{}



\maketitle

\vspace{-20pt}

\begin{abstract}\noindent 
Numerical analysis for the stochastic Stokes equations is still challenging even though it has been well done for the corresponding deterministic equations. In particular, the pre-existing error estimates of finite element methods for the stochastic Stokes equations { in the $L^\infty(0, T; L^2(\Omega; L^2))$ norm} all suffer from the order reduction with respect to the spatial discretizations. The best convergence result obtained for these fully discrete schemes is only half-order in time and first-order in space, which is not optimal in space in the traditional sense. The objective of this article is to establish strong convergence of $O(\tau^{1/2}+ h^2)$ in the $L^\infty(0, T; L^2(\Omega; L^2))$ norm for approximating the velocity, and strong convergence of $O(\tau^{1/2}+ h)$  in the $L^{\infty}(0, T;L^2(\Omega;L^2))$ norm for approximating the time integral of pressure, where $\tau$ and $h$ denote the temporal step size and spatial mesh size, respectively. The error estimates are of optimal order for the spatial discretization considered in this article (with MINI element), and consistent with the numerical experiments. The analysis is based on the fully discrete Stokes semigroup technique and the corresponding new estimates. 
\end{abstract}

\maketitle

\setlength\abovedisplayskip{4pt}
\setlength\belowdisplayskip{4pt}


\section{\bf Introduction}\label{sec:intr}
We consider the time-dependent stochastic Stokes equations in a 
domain $D \subset \R^d$, $d\in\{2,3\}$, under the stress boundary condition, i.e., 
\begin{equation}
\label{spde}
\left \{
\begin{aligned} 
\d u &= [\nabla \cdot \T (u, p)+ f] \, \d t + B(u) \, \d W(t) && \mbox{in}\,\,\,  D\times (0,T] ,\\
\nabla\cdot u&=0&&\mbox{in}\,\,\,  D\times (0,T] ,\\
\T (u, p) \n&=0 && \mbox{on}\,\,\, \partial D\times (0,T] ,\\
u&=u^0 && \mbox{at}\,\,\, D\times \{0\} , 
\end{aligned}
\right .
\end{equation}
where $u$ and $p$ denote the velocity and the pressure of the fluid, respectively, $f$ is a given source field and {$\n$ denotes the outward unit normal vector on the boundary $\partial D$. }Moreover, 
the stress tensor $\T (u, p)$ is defined by 
\begin{align}
\T (u, p)=2\D(u)-p\I \qquad  \mbox{and}  \qquad \D(u)=\frac12 \big(\nabla u + (\nabla u)^T\big),
\end{align}
where $\I$ denotes
the identity tensor. 
The stochastic noise is determined by an $L^2(D)^d$-valued $Q$-Wiener process $\{ W(t); t  \ge 0 \}$ on a filtered probability space $(\Omega, \mathcal{F}, \mathbb{P}, \{\mathcal{F}_t\}_{t\ge 0})$
with respect to the normal filtration $\{\mathcal{F}_t\}_{t\ge 0}$, and a linear operator $B(u): L^2(D)^d\rightarrow L^2(D)^d$ which depends on the solution nonlinearly.

The numerical approximations of deterministic Navier--Stokes (NS) equations have been well-understood nowadays; see \cite{Heywood-Rannacher-1990, ingram-2013-new,layton-2014-numerical, marion-1998-navier, Nochetto-Pyo-2004, shen-1992-error-1, shen-1992-error-2}. {
For the stochastic NS equations driven by multiplicative non-solenoidal noises, Brze{\'z}niak, Carelli \& Prohl \cite{brzezniak2013finite} proposed practical time-stepping schemes based on the finite element methods (FEMs) and established the convergence for velocity approximation (as a function sequence) to weak martingale solutions in 3D and to strong solutions in 2D, using the compactness argument. 
To obtain convergence rates for space-time discretizations of the stochastic NS equations, a main tool is the localization of the nonlinear term over a probability space of large probability, leading to a convergence rate in probability,  as discussed in \cite{carelli2012rates, breit2021convergence, breitProhl2023error, bessaih2014splitting}.
\begin{enumerate}[label={\rm(\arabic*)},ref=\arabic*,topsep=2pt,itemsep=0pt,partopsep=1pt,parsep=1ex,leftmargin=20pt]

\item[$\bullet$] For the 2D stochastic NS equations with non-solenoidal noises under the periodic boundary condition, Carelli \& Prohl \cite{carelli2012rates} investigated implicit and semi-implicit time discretizations with FEMs, demonstrating a convergence in probability in the $L^\infty(0, T; L^2)$ norm with rate (almost) $1/4$ in time and linear convergence in space for the velocity. 

\item[$\bullet$] For the 2D stochastic NS equations with non-solenoidal noises under the periodic boundary condition, Bessaih, Brze{\'z}niak \& Millet \cite{bessaih2014splitting} studied the convergence of a time-splitting method based on the Lie-Trotter formula. They proved that the speed of the convergence in probability is almost $1/2$ for the velocity approximations, which is shown by means of an $L^2(\Omega,\P)$ convergence localized on a set of arbitrarily large probability.

\item[$\bullet$] For the 2D stochastic NS equations with non-solenoidal noises under the periodic boundary condition, Breit \& Dodgson \cite{breit2021convergence} recently established convergence in probability for the fully discrete implicit FEMs based on a stochastic pressure decomposition technique. They obtained a convergence in probability with rate (almost) $1/2$ in time and linear convergence in space,  measured in the norm of $L^\infty(0, T; L^2) \cap L^2(0, T; H^1)$. This improves the earlier results in \cite{carelli2012rates}, where the convergence rate in time was only (almost) $1/4$.   

\item[$\bullet$] For the 2D stochastic NS equations with solenoidal noises under the Dirichlet boundary condition, Breit \& Prohl \cite{breitProhl2023error} established convergence rates for the fully discrete semi-implicit FEMs using an approach based on discrete stopping times. They showed the convergence of velocity approximations in the $L^\infty(0, T; L^2) \cap L^2(0, T; H^1)$ norm with respect to convergence in probability, achieving the rate (almost) 1/2 in time and linear convergence in space.

\end{enumerate}
In addition to previously discussed convergence in probability, the strong rates of convergence (i.e., rates in $L^2(\Omega)$) for the stochastic NS equations have also been explored, see\cite{bessaih2019strongL2, bessaih2022strong, bessaih2021space}.
\begin{enumerate}[label={\rm(\arabic*)},ref=\arabic*,topsep=2pt,itemsep=0pt,partopsep=1pt,parsep=1ex,leftmargin=20pt]
\item[$\bullet$] The first study on the strong convergence for the 2D stochastic NS equations was conducted by Bessaih \& Millet \cite{bessaih2019strongL2} under periodic boundary conditions.  
They focused on the splitting scheme from Bessaih et al. \cite{bessaih2014splitting} and the implicit Euler schemes used in Carelli \& Prohl \cite{carelli2012rates}.

\item[$\bullet$] Further exploration of strong convergence for the fully discrete schemes of the 2D stochastic NS equations was carried out in \cite{bessaih2021space}. They focused on the implicit Euler scheme coupled with FEMs for non-solenoidal noises under periodic boundary conditions.  
This research refines previous results in the stochastic NS equations, which had only established the convergence in probability of these fully discrete numerical approximations.
\end{enumerate}

The stochastic Stokes system \eqref{spde}  is a simplified version of the stochastic NS equations with non-solenoidal noises. Most of the numerical analyses discussed above can be applied to the 2D stochastic Stokes equation. However, the convergence of pressure was not provided for the stochastic NS equations. Studies on the convergence of velocity approximations in 3D and pressure approximations for the stochastic Stokes equations have emerged only recently.
\begin{enumerate}[label={\rm(\arabic*)},ref=\arabic*,topsep=2pt,itemsep=0pt,partopsep=1pt,parsep=1ex,leftmargin=20pt]
\item[$\bullet$] For the 2D and 3D  stochastic Stokes equations with non-solenoidal noises under the periodic boundary condition, Feng \& Qiu \cite{fengQiu2021analysis} developed the fully discrete semi-implicit mixed FEMs and established strong convergence with rates for both velocity approximation in the norm of $L^\infty(0, T; L^2(\Omega; L^2))$
and pressure approximation in a time-averaged norm.  The error estimates provided in \cite{fengQiu2021analysis} were derived based on a $\tau$-dependent stability of the pressure approximations, as discussed in \cite[Lemma 2]{fengQiu2021analysis}, leading to a sub-optimal error estimate of order $O(\tau^{\frac12} + h\tau^{-\frac12})$.  This $\tau$-dependent stability of pressure approximations can be avoided in the case of solenoidal noises (i.e., $B(u)$ maps $L^2(D)^d$ into its divergence-free subspace, as considered in \cite{carelli2012time}) or pointwise divergence-free FEMs, as discussed in \cite{carelli2012rates}. 
The convergence order was improved to $O(\tau^{\frac12} + h + h^2\tau^{-\frac12})$ in Feng \& Vo \cite{fengVo2022analysis} based on the Chorin-type projection methods.
However, all spatial error constants presented in \cite{fengQiu2021analysis, fengVo2022analysis}  include a bad growth factor $\tau^{-\frac12}$.



\item[$\bullet$] Recently, for 2D and 3D non-solenoidal noises under the periodic boundary condition,  Feng,  Prohl \& Vo \cite{fengProhl2021optimally} proposed new fully discrete mixed FEMs for the stochastic Stokes equations by utilizing the Helmholtz decomposition to the noises. By removing a gradient part from the non-solenoidal noise, the modified noise becomes divergence-free. 
This modification results in an improved stability estimate of the new pressure approximations, which is not dependent on the temporal stepsize $\tau$.  
Consequently, the inf-sup stable mixed FEMs for the stochastic Stokes equation proposed in \cite{fengProhl2021optimally} achieve strong convergence with rate $O(\tau^{\frac12} + h)$ for velocity in the $L^\infty(0, T; L^2(\Omega; L^2))$ norm and pressure in a time-averaged norm. This improvement addresses the sub-optimal estimates in \cite{fengQiu2021analysis}.

\end{enumerate}
The numerical analysis in \cite{fengQiu2021analysis, fengProhl2021optimally, fengVo2022analysis} is based on the certain conditions of noise}, which can be viewed as the following Lipschitz continuity and growth conditions in the case that $B$ is an $\L_2^0$-valued function:
\begin{align}\label{ass-BdW-strong}
\|B(v) - B(w)\|_{{\L_2^0(\mathbb{L}^2, \, \mathbb{L}^2)}} \le  C\|v - w\|_{L^2} \quad \mbox{and} \quad \|B(v)\|_{{\L_2^0(\mathbb{L}^2, \, \mathbb{L}^2)}} \le C \big(1 + \|v\|_{L^2} \big) 
\end{align}
for $v, w \in L^2(D)^d$, where $\L_2^0(\mathbb{L}^2,\mathbb{L}^2)$ is the space of Hilbert--Schmidt operators on $\mathbb{L}^2=L^2(D)^d$. 
In the presence of non-solenoidal noises, the half-order temporal convergence shown in the above-mentioned analyses is optimal and consistent with the numerical experiments. However, the first-order spatial convergence in the $L^\infty(0, T; L^2(\Omega; L^2))$ norm is not optimal and inconsistent with the numerical experiments. 

As far as we know, the numerical analysis of the stochastic heat equation has been studied extensively \cite{Cao-Hong-Liu-2017, Mukam-Tambue-2020, Wang-2017, Yan2005} and the second-order convergence in space has been proved. However, for the stochastic Stokes equations, the existing approach applies only to a simple case with a solenoidal noise and a pointwise divergence-free finite element space, for which the error analysis actually reduces to the analysis for an abstract parabolic equation. 
The second-order convergence in space in the norm of $L^\infty(0, T; L^2(\Omega; L^2))$ for \eqref{spde}, under the common setting involving general non-solenoidal noises and frequently used inf-sup stable mixed FEMs,  has not been proved. 
The objective of this article is to address this question under the following noise condition for some $\beta > \frac{d}{2}$: 
\begin{align}\label{ass-BdW-weak}
\left \{
\begin{aligned}
& \|B(v) - B(w)\|_{\L_2^0(\mathbb{L}^2, \, \mathbb{H}^{-1/2})} \le C\|v - w\|_{L^2},\\[4pt]
& \|B(v) - B(w)\|_{\L_2^0(\mathbb{L}^2, \, \mathbb{H}^1)}  \le C\|v - w\|_{H^\beta} \,\,\, \mbox{and} \,\,\, 
\|B(v)\|_{\L_2^0(\mathbb{L}^2, \, \mathbb{H}^1)}  \le C \big(1 + \|v\|_{H^\beta}\big) , 
\end{aligned} \right .
\end{align} 
which also covers many noises that were considered in the literature (e.g., \cite{fengProhl2021optimally, fengVo2022analysis}); see the examples in Remark \ref{remark-ass} and the numerical experiments in Section \ref{sec:num}. The error analyses for the stochastic Stokes equations in the previous articles, based on energy approach, do not yield second-order convergence in space in the $L^\infty(0,T;L^2(\Omega;L^2))$ norm under noise condition \eqref{ass-BdW-weak}, for the same difficulty caused by the low regularity of pressure; see the discussions in \cite{fengQiu2021analysis, fengProhl2021optimally, fengVo2022analysis}. Therefore, a new approach of error analysis needs to be developed to address this difficulty.

In this article, we establish optimal convergence of fully discrete mixed methods with standard inf-sup stable finite element pairs for the stochastic Stokes equations driven by a {non-solenoidal} multiplicative noise under condition \eqref{ass-BdW-weak} and the regularity of the mild solution (see Proposition~\ref{thm-u-stability}). In particular, the strong convergence of $O(\tau^{1/2}+ h^2)$ in the $L^\infty(0, T; L^2(\Omega; L^2))$ norm   is proved for approximating the velocity, and the strong convergence of $O(\tau^{1/2}+ h)$  in the $L^{\infty}(0, T;L^2(\Omega;L^2))$ norm  is proved for approximating the time integral of pressure (see Theorem~\ref{THM:sfem-ferrs}), where $\tau$ and $h$ denote the temporal stepsize and spatial mesh size, respectively. The error estimates are of optimal order for MINI element in the traditional sense and consistent with the numerical experiments. 

The analysis presented in this article is based on the fully discrete Stokes semigroup technique and the corresponding new estimates (refer to Lemma \ref{lem:fem-ferrs} and Remark \ref{remark41}) for general non-solenoidal functions. The regularity of the pressure solution to the stochastic Stokes problem is generally low due to the influence of the non-solenoidal noise. This low regularity of the pressure is a main obstacle in proving second-order convergence in space, as discussed in \cite{fengQiu2021analysis, fengProhl2021optimally}.
We overcome this difficulty by using the fully discrete Stokes semigroup technique to avoid using the $\tau$-dependent estimates of the pressure approximations.
Specifically, we have developed technical estimates in Lemma \ref{lem:fem-ferrs}, which do not require $v$ to be divergence-free (where $v$ represents the noise term in the error estimates, as indicated by  $T_3$ in the error equation \eqref{error-expr}).
This was achieved by proving and utilizing the $H^1$-stability of the orthogonal projection onto the discrete divergence-free finite element subspace (see Subsection \ref{subsec:orth-dicomp-V_h}) and based on the error estimates of the fully-discrete FEMs for the deterministic Stokes problem.
In the previous papers, the semigroup estimates provided in Lemma \ref{lem:fem-ferrs} were only shown and utilized for the abstract parabolic equation, which requires $v$ to be divergence-free and requires the finite element space to be pointwise divergence-free when the results are applied to the Stokes equations. 
This distinction is crucial for our error analysis of the stochastic Stokes equations, and makes it possible to prove a better convergence rate for non-solenoidal noises using non-divergence-free finite elements.

The rest of this article is organized as follows. 
In Section~\ref{sec:pre}, we collect the assumptions and describe a fully discrete method with a standard inf-sup stable FE pair for the stochastic Stokes and then,  present our main theorem. 
In Section~\ref{section:abstract}, we present the abstract formulation of the stochastic Stokes equations under the stress boundary condition and define the mild solution based on the abstract formulation. 
In Section~\ref{sec:est-dis-semigroup}, we present some technical estimates for the discrete semigroup associated to the Stokes operator. The results are used in the error analysis of the fully discrete FEMs for the stochastic problem in Section~\ref{sec:spde-fully-fem}. 
In Section~\ref{sec:num}, we present numerical experiments to support our theoretical analysis by illustrating the convergence orders of the velocity and pressure approximations.


\section{\bf Main results}\label{sec:pre}


In this section, we present some notations and assumptions to be used in this article, as well as the numerical scheme for the stochastic Stokes equations. Then we present the main theoretical result on the convergence of the numerical scheme. 

\subsection{Basic notations}
{We assume that the domain $D$ exhibits elliptic $H^2$ regularity when considering the deterministic Stokes equations with the stress boundary condition. This assumption implies that solutions $(v, q)$ of the linear Stokes equations
\begin{equation}
\left \{
\begin{aligned}  
v - \nabla \cdot \T(v, q) &= g \qquad &&{\rm in}\, \, D,\\
\nabla \cdot v &= 0 \qquad &&{\rm in}\, \, D,\\
\T(v, q)\n &= 0 \qquad &&{\rm on}\, \, \partial D,
\end{aligned}
\right .
\end{equation}
satisfy the following estimates:
\begin{align} \label{stokes-sta-H2}
\|v\|_{H^2} + \|q\|_{H^1} \le  C\|g\|_{L^2}.
\end{align}
This $H^2$ elliptic regularity estimate holds for Stokes equations in two-dimensional convex polygons under both Dirichlet boundary condition \cite{Kellogg-Osborn-1976} and Neumann/Stress boundary condition \cite{li2020explicit,Orlt-Sandig-1995}, and three-dimensional convex polyhedron under the Dirichlet boundary condition (see \cite[Eq. (1.8)]{dauge1989stationary}). If the domain is smooth then the $H^2$ elliptic regularity estimate holds for both Dirichlet \cite[Eq. (1.5)]{dauge1989stationary} and Neumann/Stress boundary conditions \cite[Theorem 1.1]{Shibata-Shimizu}. In this paper, we focus on domains on which the  $H^2$ elliptic regularity estimate in \eqref{stokes-sta-H2} holds. 
%
}

Let $H^s(D)$, $s\ge0$, denote the conventional Sobolev space of functions defined on $D$, with $L^2(D)=H^0(D)$, {spaces with blackboard letters (e.g., $\mathbb{H}^s(D) = H^s(D)^d$) represent the spaces of vector valued functions. The dual space 
of $H^s(D)$ is denoted by ${H^{-s}(D)}$.} 
Let $(\Omega, \mathcal{F}, \mathbb{P}, \{\mathcal{F}_t\}_{t\ge 0})$ denote a filtered probability space with the probability measure $\mathbb{P}$, the $\sigma$-algebra $\mathcal{F}$ and the continuous filtration $\{\mathcal{F}_t\}_{t\ge 0}$. The expectation of a random variable $v$ defined on $(\Omega, \mathcal{F}, \mathbb{P}, \{\mathcal{F}_t\}_{t\ge 0})$ is denoted by $\E v$. 


{For a Hilbert space $\mathcal{K}$}, let $W(t)$ be a $\mathcal{K}$-valued $Q$-Wiener process on $(\Omega, \mathcal{F}, \mathbb{P}, \{\mathcal{F}_t\}_{t\ge 0})$, with expression
\begin{align}\label{ass-con-W}
W(t) 
=\,&\sum_\ell \sqrt{\mu_\ell} \phi_\ell W_\ell(t) \qquad \forall \, t \in [0, T],
\end{align}
where $\{W_\ell(t)\}_{j\ge 1}$ is a family of independent real-valued Wiener processes and the trace operator 
$Q : \mathcal{K} \rightarrow \mathcal{K}$
is bounded, self-adjoint, positive semi-definite, with eigenvalues $\{\mu_\ell\}_{\ell \ge 1}$ and eigenfunctions $\{\phi_\ell\}_{\ell \ge 1}$.

{
Let $\L_2(\mathcal{K}, \mathcal{H})$  and $\L_2^0(\mathcal{K}, \mathcal{H})$ be the spaces of Hilbert-Schmidt operators from $\mathcal{K}$ to $\mathcal{H}$ and from $Q^{1/2}(\mathcal{K})$ to $\mathcal{H}$, respectively,
satisfying
$$
\|\Phi\|_{\L_2^0(\mathcal{K},\mathcal{H})}: = \Big( \sum_\ell \mu_\ell\|\Phi \phi_\ell\|_{_\mathcal{H}}^2 \Big)^{\frac12} = \|\Phi Q^{\frac12}\|_{\L_2(\mathcal{K},\mathcal{H})} .
$$
For a progressively measurable process $\Phi: [0, T] \to \L_2^0(\mathcal{K},\mathcal{H})$ 
with }
$\int^T_0 \|\Phi(s)\|_{{\L_2^0(\mathcal{K}, \mathcal{H})}}^2 \, \d s  < \infty$ $\mathbb{P}$-a.s.,
the stochastic integral $\int^t_0 \Phi(s) \, \d W(s)$ is well defined and It\^o's isometry holds, i.e., 
\begin{align}\label{Ito_isometry}
\E \bigg\|\int^t_0 \Phi(s) \, \d W(s)\bigg\|_{{\mathcal{H}}}^2 =
\E \int^t_0 \|\Phi(s)\|_{{\L_2^0(\mathcal{K},\mathcal{H})}}^2 \, \d s.
\end{align}

For the simplicity of notations, {we denote by $\L_2^0 = \L_2^0(\mathbb{L}^2, \, \mathbb{L}^2)$}  and  denote by $x \lesssim y$ 
(or $y \gtrsim x$) the
statement ``$x \le Cy$ (or $x \ge Cy$) for some positive constant $C$ which is independent of the stepsize $\tau$ and the mesh size $h$ in the numerical approximation".

\subsection{Assumptions on the noise and nonlinearity}

For the existence and uniqueness of mild solutions to problem \eqref{spde}, as well as the numerical approximation to the mild solutions, we work with the following assumptions on the noise and nonlinearity. 

\begin{assumption}{(Stochastic noise)}\label{ass-W}
We assume that the $Q$-Wiener process $W(t)$ has the following property: 
\begin{align}\label{ass-con-AQ}
\|(-\Delta)^{\frac 12}\|_{\L_0^2} \lesssim 1,
\end{align}
where $\Delta: H^2_N(D)\rightarrow L^2(D)$ denotes the Neumann Laplacian operator with the domain 
$$
H^2_N(D)=\{v\in H^2(D):\partial_nv=0\,\,\mbox{on}\,\,\partial D\} ,
$$ 
and $(-\Delta)^{\frac12}:H^1(D)\rightarrow L^2(D)$ denotes the fractional power of $-\Delta$. 
\begin{remark}\upshape 
In the case that $Q$ and  $-\Delta$ have the same eigenfunctions, the condition \eqref{ass-con-AQ} is equivalent to 
\begin{align*}
\sum_\ell \mu_\ell\lambda_\ell \lesssim 1,
\end{align*}
where $\lambda_\ell$ is an eigenvalue of  $-\Delta$.
\end{remark}
\end{assumption}

\begin{assumption}{(Nonlinearity and source term)}\label{ass-B}
We assume that $B(v):L^2(D)^d\rightarrow L^2(D)^d$ is a bounded nonlinear operator for any $v\in L^2(D)^d$, 
{satisfying the following Lipschitz continuity and growth conditions for some $\beta\in(\frac{d}{2},2)$: 
\begin{align}\label{ass-con-nosie-1}
& \|B(v) - B(w)\|_{\L_2^0(\mathbb{L}^2, \, \mathbb{H}^{-1/2})} \lesssim\|v - w\|_{L^2},\\ 
\label{ass-con-nosie-2}
& \|B(v) - B(w)\|_{\L_2^0(\mathbb{L}^2, \, \mathbb{H}^1)} \lesssim\|v - w\|_{H^\beta} \,\,\, \mbox{and} \,\,\, 
\|B(v)\|_{\L_2^0(\mathbb{L}^2, \, \mathbb{H}^1)} \lesssim 1 + \|v\|_{H^\beta} .
\end{align}  
}
Moreover, we assume that the function $f: [0, T] \times L^2(D)^d \to L^2(D)^d$ satisfies 
\begin{align}\label{ass-con-f-sta}
\|f(t)\|_{L^2} &\lesssim  1 &&\forall\,\, 0\le t\le T,&\\
\label{ass-con-f-lip}
\|f(t_1) - f(t_2)\|_{L^2} &\lesssim  (t_1 - t_2)^{\frac12} &&\forall\,\, 0\le t_1 \le t_2 \le T.&
\end{align}    
\end{assumption}

\begin{remark}\label{remark-ass}
\upshape 
The conditions in \eqref{ass-con-nosie-1}-\eqref{ass-con-nosie-2} are satisfied if $B(v)$ satisfies the following estimates:
\begin{align}\label{ass-con-B1}
\|B(v) \phi_\ell \|_{H^1} &\lesssim {(1 + \|v\|_{H^\beta})} \|\phi_\ell\|_{H^1} &&\forall\, \, v \in H^\beta(D)^d, \\
\label{ass-con-B2}
\|(B(v_1) - B(v_2))\phi_\ell \|_{H^{-\frac12}}  &\lesssim \|v_1-v_2\|_{L^2} \|\phi_\ell\|_{H^1} 
&&\forall\, \, v_1,v_2 \in L^2(D)^d ,\\
\label{ass-con-B3}
\|(B(v_1) - B(v_2))\phi_\ell \|_{H^1}  &\lesssim \|v_1-v_2\|_{H^\beta} \|\phi_\ell\|_{H^1} 
&&\forall\, \, v_1,v_2 \in H^\beta(D)^d  . 
\end{align}
for some $\beta\in(\frac{d}{2},2)$. 
The proofs in this paper exclusively rely on the use of  \eqref{ass-con-nosie-1}--\eqref{ass-con-nosie-2}.
A 2D example of a suitable operator $B(u)$ and noise $W(t)$,  satisfying \eqref{ass-con-B1}--\eqref{ass-con-B3},  is given by
\begin{align}\label{test-BB}
\begin{aligned}
B(u) &= \left(\begin{array}{cc}\sqrt{u^2_1 +1} &  \sqrt{u^2_1 +1} \\[4pt]  \sqrt{u^2_2 + 1} & \sqrt{u^2_2 + 1} \end{array}\right) \quad \mbox{for} \quad u = (u_1, u_2), \\
W(t, \mathbf{x})
&=  \sum_{\ell_1 = 1}^{\infty}  \sum_{\ell_2 = 1}^{\infty}  \sqrt{\mu_{\ell_1 \ell_2}}
\left(\begin{array}{cc} 
\phi_{\ell_1 \ell_2}(\mathbf{x})  \\[4pt] 
\phi_{\ell_1 \ell_2}(\mathbf{x}) \end{array}\right) 
w_{\ell_1 \ell_2}(t) \qquad \forall \,\, t \in [0, T], 
\end{aligned}
\end{align} 
with 
\begin{align*}
\mu_{\ell_1 \ell_2} & = 
\left \{
\begin{aligned}
&0 \quad &&\mbox{for} \,\, \,(\ell_1, \ell_2) = (0, 0),\\
& (\ell_1^2 + \ell_2^2)^{-(2 + \varepsilon)}\quad &&\mbox{for} \,\, \, (\ell_1,\ell_2)\in \mathbb{Z}^2 /\{(0, 0)\},  \,\,\, \varepsilon = 0.1,
\end{aligned}
\right .\\[5pt]
\phi_{\ell_1 \ell_2}(\mathbf{x}) & = \sin(\ell_1 \pi x_1) \sin(\ell_2 \pi x_2),\quad \ell_1,\ell_2 = 0, 1,2,\dots, 
\end{align*}
which forms an orthonormal basis of $L^2(D)$. The noise term $B(u)\d W$ determined by \eqref{test-BB} and series $W(t)$ can be written as 
\begin{align*}
B(u)\d W(t)
&= \sum_{\ell_1 = 1}^{\infty}  \sum_{\ell_2 = 1}^{\infty}  \sqrt{\mu_{\ell_1 \ell_2}} \left(\begin{array}{cc}\sqrt{u^2_1 +1} \\ \sqrt{u^2_2 +1}  \end{array}\right) \phi_{\ell_1 \ell_2}(\mathbf{x}) \, w_{\ell_1 \ell_2}(t),
\end{align*}
which is non-solenoidal and was used in \cite{fengProhl2021optimally, fengVo2022analysis} to measure the effectiveness of numerical methods for the stochastic Stokes/NS equation. This example of noise satisfies both Assumption \ref{ass-W} and conditions \eqref{ass-con-nosie-1}--\eqref{ass-con-nosie-2} in Assumption \ref{ass-B}. 
\end{remark}

\begin{remark}\upshape 
In the case that $(B(v)\phi_\ell)(x) = b_\ell(v(x))\phi_\ell(x)$ for some functions $b_\ell:\R\rightarrow \R$, $\ell=1,2,\dots$, 
conditions \eqref{ass-con-B1}--\eqref{ass-con-B3} are satisfied if the functions $b_\ell$ are uniformly Lipschitz continuous with respect to $\ell$, i.e., 
\begin{align*}
\begin{aligned}
|b_\ell(\sigma)| &\lesssim 1+|\sigma| &&\forall\,\sigma\in\R, \\
|b_\ell(\sigma_1)-b_\ell(\sigma_2)| &\lesssim |\sigma_1-\sigma_2|  &&\forall\,\sigma_1,\sigma_2\in\R .
\end{aligned}
\end{align*}
\end{remark}

\begin{assumption}{(Initial value)}\label{ass-u0} 
We assume that 
the initial value $u_0 : \Omega \to L^2(D)^d$ is an $\F_0/\B(L^2(D)^d)$-measurable function with $u_0 \in L^2(\Omega,  \DD(A))$.
\end{assumption}







\subsection{The numerical method and its convergence}
\label{section:regularity}

Let {$V_h\times Q_h \subset H^1(D)^d \times L^2(D)$} be a pair of finite element spaces subject to a quasi-uniform triangulation of $D$ with the mesh size $h > 0$, satisfying the following properties (see \cite[Chapter II]{girault1986finite}):
\begin{enumerate}
\item
There exists a projection operator
$\Pi_h: H^1(D)^d \to V_h$, called the Fortin projection, satisfying
\begin{align}
\label{Fortin-Pih-def}
& \quad \big(\nabla \cdot (v -\Pi_h v), q_h\big) = 0  &&\forall \, \,v \in H^1(D)^d\,\,\, \mbox{and}\,\,\, q_h \in Q_h,  \quad\\
\label{Fortin-Pih-conv}
&  \quad \|v - \Pi_h v\|_{H^m}\lesssim h^{1-m}\|v\|_{H^1} &&\forall \, \,v \in H^1(D)^d\,\,\, \mbox{and}\,\,\,m = 0, 1,  \quad\\
\label{Fortin-Pih-H1-sta}
& \quad \|\Pi_h v\|_{H^1}\lesssim \|v \|_{H^1} &&\forall \,\, v \in H^1(D)^d,     \quad
\end{align}

\item
The following approximation properties hold: 
\begin{align} 
\inf_{q_h\in Q_h}\|q- q_h\|_{L^2}
&\lesssim h^m \|q\|_{H^m}\quad \forall\, q \in H^m(D) \,\,\, \mbox{and}\,\,\, m =1, 2 , \\
\label{app-pro-q-H10}
\inf_{q_h\in Q_h\cap H^1_0(D)} \|q- q_h\|_{L^2}
&\lesssim h^m \|q\|_{H^m}\quad \forall\, q \in H^m(D)\cap H^1_0(D) \,\,\, \mbox{and}\,\,\, m =1, 2 .
\end{align}
Therefore, the $L^2$ projection $P_{Q_h}:L^2(D)\rightarrow Q_h$ satisfies the following estimates:  
\begin{align} 
\label{app-pro-q}
&\|q- P_{Q_h}\|_{L^2}
\lesssim h^m \|q\|_{H^m}\quad \forall\, q \in H^m(D) \,\,\, \mbox{and}\,\,\, m =1, 2 ,\\
&\|q - P_{Q_h} q\|_{  H^{-1}} 
= \sup_{\substack{\eta \in H^1(D)\\ \hspace{-20pt}\eta \neq 0}  } \frac{(q - P_{Q_h} q, \eta)}{\|\eta\|_{H^1}} 
= \sup_{\eta \, \in H^1(D)} \frac{(q - P_{Q_h} q, \eta - P_{Q_h} \eta)}{\|\eta\|_{H^1}} \notag\\
&\hspace{189pt} \lesssim h^2\|q\|_{H^1}.
\end{align}

\item 
The following inverse inequality holds 
\begin{align}\label{app-inverse}
\|v_h\|_{H^1} \lesssim h^{-1} \|v_h\|_{L^2} \quad\forall\, v_h \in V_h .
\end{align}

\item
The inf-sup condition holds:  
\begin{align}\label{inf-sup}
\|q_h\|_{L^2}  \lesssim \sup_{\substack{v_h\in V_h\\v_h\neq 0}} \frac{ (q_h,\nabla\cdot v_h)}{\|v_h\|_{H^1}}
\quad\forall\, q_h\in Q_h . 
\end{align}

\end{enumerate}

Several finite element spaces $V_h\times Q_h$ are known to satisfy the properties above, such as the standard inf-sup finite element spaces including the mini element space in \cite{arnold1984stable} and the Taylor--Hood finite element space in \cite{taylor1973numerical}. We assume that the triangulation may contain curved triangles/tetrahedra which fits the boundary exactly in order to avoid making the problem more complicated with additional errors in approximating the boundary. 

The natural function spaces associated to incompressible flow is the divergence-free subspaces of $L^2(D)^d$ and $H^1(D)^d$, defined by 
\begin{align}
\begin{aligned}
X =\{ v\in L^2(D)^d: \nabla\cdot v=0\}
\quad  \mbox{and}  \quad  V=X\cap H^1(D)^d.
\end{aligned}
\end{align}
Let 
$
X_h:= \big\{v_h\in V_h: (\nabla\cdot v_h, q_h) =0,\,\, \forall\,\, q_h\in Q_h \big\} 
$
be a discrete divergence--free subspace of $V_h$, 
and denote by $P_{X_h}: L^2(D)^d \to X_h$ the $L^2$-orthogonal projection onto $X_h$. 
On a uniform partition $t_n=n\tau$, $n=0,1,\dots,N$, of the time interval $[0, T]$ with the stepsize $\tau = T/N$, we consider the following fully discrete 
semi-implicit Euler method for problem \eqref{spde}: 
For the given initial value $u^0_h=P_{X_h}u^0$, find a pair of processes $(u_h^n,p_h^n)\in V_h\times Q_h$, $n=1,\dots,N$, such that the weak formulation 
\begin{align}\label{fully-sFEM-weak}
\left\{\begin{aligned}
(u_h^n,v_h) 
+ 2\tau\big(\D(u_h^n),\D( v_h)\big)
=\,&(u_h^{n-1},v_h) + \tau(p_h^n,\nabla\cdot v_h) + \tau (f(t_n) ,v_h) &&\forall\, v_h\in V_h \\
&+ (B(u_h^{n-1})\Delta W_n ,v_h) ,\\
(\nabla\cdot u_h^n,q_h) =\,&0 &&\forall\, q_h\in Q_h 
\end{aligned}
\right.
\end{align}
holds $\P$-a.s. for all test functions $(v_h, q_h) \in V_h \times Q_h$, where $\Delta W_n: = W(t_{n}) - W(t_{n-1}) $ is a random variable with $N(0, \tau Q)$ distribution.

By choosing $v_h \in X_h$ in \eqref{fully-sFEM-weak}, the  fully discrete method in \eqref{fully-sFEM-weak} can be equivalently written as finding 
{a $X_h$-valued process $u_h^n$},
$n=1,\dots,N$, such that $\P$-a.s.
\begin{align}\label{fully-sFEM-weak-Xh}
\left\{\begin{aligned}
(u_h^n,v_h) 
+ 2\tau\big(\D(u_h^n),\D( v_h)\big)
= \, & (u_h^{n-1},v_h)  + \tau (f(t_n) ,v_h) + (B(u_h^{n-1})\Delta W_n ,v_h)\quad\forall\, v_h\in X_h ,\\
u_h^0 =\, &  P_{X_h}u^0.
\end{aligned}
\right.
\end{align}
If we denote by $A_h:X_h\rightarrow X_h$ the discrete Stokes operator  defined by 
\begin{align}
\label{def-Ah}
(A_hv_h,w_h)= 2\big(\D(v_h),\D (w_h)\big)\qquad\forall\,\, v_h,\, w_h\in X_h.
\end{align}
Then the fully discrete method in \eqref{fully-sFEM-weak-Xh} is equivalent to finding {a $X_h$-valued process $u_h^n$}, $n=1,\dots,N$ such that $\P$-a.s. 
\begin{equation}\label{fully-sFEM}
\left \{
\begin{aligned} 
u_h^n
&=  \bar E_{h, \tau} u_h^{n-1} +  \tau \bar E_{h, \tau} P_{X_h}f(t_n) + \bar E_{h, \tau}P_{X_h}[ B(u_h^{n-1})\Delta W_n], \\
u_h^0&=P_{X_h}u^0,
\end{aligned} \right.
\end{equation} 
where $\bar E_{h, \tau}$ denotes the discrete semigroup in the full discretization defined by 
\begin{align}
\label{def-E_h-tau}
\bar E_{h, \tau} v_h = (I + \tau A_h)^{-1}v_h .
\end{align}

The main result of this article is the following theorem, which provides the convergence of the numerical solution to the mild solution of the stochastic Stokes equations.
\begin{theorem}\label{THM:sfem-ferrs}
{
Let $\{ W(t); t  \ge 0 \}$ be a $Q$-Wiener process  on the filtered probability space $(\Omega, \mathcal{F}, \mathbb{P}, \{\mathcal{F}_t\}_{t\ge 0})$}.
Let Assumptions \ref{ass-W}--\ref{ass-u0} be fulfilled and assume that the finite element space $V_h\times Q_h$ has Properties (1)--(4). Then the numerical solution $(u_h^n,p_h^n)$, $n=1,\dots,N$, determined by \eqref{fully-sFEM-weak} has the following error bounds: 
\begin{align}\label{sfem-ferrs:u}
\max_{1\le n \le N}\Big(\E \|u(t_n)-u_h^n \|_{L^2}^2\Big)^{\frac12}
&\lesssim 
\tau^{\frac12} + h^2 , \\[5pt] 
\label{sfem-ferrs:p}
{\max_{1\le m \le N}} \Big(\E \Big\|\int_0^{t_m} p(s)  \,\d s- \tau\sum_{n = 1}^m  p_h^n \Big\|_{L^2}^2 \Big)^{\frac12}
&\lesssim \tau^{\frac 12} + h \, . 
\end{align}
\end{theorem}

\begin{remark}
{\upshape
Since the numerical scheme in \eqref{fully-sFEM-weak} is linearly implicit, the existence and uniqueness of numerical solutions are standard. The convergence rates presented in the Theorem \ref{THM:sfem-ferrs} 
is optimal in space for the inf-sup stable MINI element space in \cite{arnold1984stable}. 
The half-order convergence in time is the same as the previous results and consistent with the numerical experiments.}
\end{remark}

\begin{remark}
{\upshape
The numerical scheme and analysis presented in this article can be extended to noises of the type $B(s, u)$, provided that certain H\"older continuity conditions of $B(s,u)$ with respect to $s$ are assumed, as stated in Assumption~\ref{ass-B}.
}
\end{remark}

The proof of Theorem \ref{THM:sfem-ferrs} will be presented in the next three sections based on the techniques of continuous and discrete analytic semigroups.



\section{The abstract formulation under the stress boundary condition}
\label{section:abstract}

In this section, we present the abstract formulation and functional setting of the stochastic Stokes equations under the stress boundary condition, and define the mild solution of the stochastic Stokes equations to be approximated by the numerical solutions. 

The natural function spaces associated to incompressible flow is the divergence-free subspaces of $L^2(D)^d$ and $H^1(D)^d$, defined by 
\begin{align}
\begin{aligned}
X =\{ v\in L^2(D)^d: \nabla\cdot v=0\}
\quad  \mbox{and}  \quad  V=X\cap H^1(D)^d.
\end{aligned}
\end{align}
{The divergence-free subspace  $X$ of $L^2(D)^d$ is endowed with the $L^2(D)^d$ norm.}
It is known that the following orthogonal decomposition holds: 
$$L^2(D)^d=X\oplus U\quad\mbox{with}\quad U=\{\nabla q:q\in H^1_0(D)\} . $$
In particular, any $v\in L^2(D)^d$ can be decomposed as $v=P_Xv + \nabla \eta$, where 
$$
P_Xv = v - \nabla \eta
$$
denotes the $L^2$-orthogonal projection from $L^2(D)^d$ onto $X$, with $\eta$ being the solution of the equation  
\begin{align}
\left\{\begin{aligned}
\Delta \eta  &= \nabla\cdot v && \mbox{in}\,\,\,  D,\\
\eta &= 0 && \mbox{on}\,\,\,  \partial D .
\end{aligned}
\right.
\end{align}
Since the solution $\eta$ of the Poisson equation satisfies 
$$
\|\eta\|_{H^{s+1}}\le C\|v\|_{H^{s}} \quad\mbox{for}\,\,\, s\in[0,2],
$$
it follows that the $L^2$ projection $P_Xv = v - \nabla \eta$ has the following properties:
\begin{align}\label{P_X-bound-s}
\|P_X v\|_{H^s} \lesssim \|v\|_{H^s} \quad \mbox{for}\,\,\, s\in[0,2] .
\end{align}
Since the $L^2$ projection operator $P_X:L^2(D)^d\rightarrow X$ is self-adjoint, it follows that 
(by a duality argument)  
\begin{align}\label{P_X-bound}
\|P_X v\|_{H^{-s}} \lesssim \|v\|_{H^{-s}}  \quad \mbox{for}\,\,\, s\in[0,2] .
\end{align}

Let $E(t):X\rightarrow X$ be the semigroup of bounded linear operators defined by $E(t)v^0:=v(t)$ as the solution of the linear Stokes equations 
\begin{align*}
\left\{\begin{aligned}
\frac{\partial v}{\partial t} - \nabla\cdot\T(v,q) &= 0 &&\mbox{in}\,\,\, D\times\R_+\\
\nabla\cdot v&=0 &&\mbox{in}\,\,\, D\times\R_+ \\
T(v,q) n&=0&&\mbox{on}\,\,\,\partial D\times\R_+ \\
v(0)&=v^0 &&\mbox{in}\,\,\, D .
\end{aligned}\right. 
\end{align*}
Let $-A$ be the generator of the semigroup $E(t)$ with domain 
$$
D(A)= \Big\{v^0\in X: \lim_{t\rightarrow 0} \frac{E(t)v^0-v^0}{t} \,\,\,\mbox{exists in $X$} \Big\} ,
$$
which is a dense subspace of $X$. 
Then $v^0\in D(A)$ {if and only if} $\partial_tv\in C([0,T];X)$, which {is equivalent to}
\begin{align} 
\mbox{$\nabla\cdot\T(v,q) \in C([0,T];X) $\, and\, $\nabla\cdot v=0 $}. \notag
\end{align}
The above condition holds {if and only if }
\begin{align} 
&\mbox{$-\Delta v + \nabla q \in C([0,T];X) $\, and\, $\nabla\cdot v=0 $}, \notag\\
&\mbox{$2\D(v)\n-q\n=0$ on $\partial D$}, \notag\\
&\mbox{$q$ is a harmonic function with boundary condition $q=2\D(v)\n\cdot \n$},\label{v0inD(A)1}
\end{align}
which {is equivalent to}
\begin{align} 
&v\in C([0,T];H^2(D)^d),\,\,\,\mbox{$\nabla\cdot v=0$\, and\, $\D(v)\n \times \n =0$\, on\, $\partial D$} , \label{v0inD(A)2}
\end{align}
where the last equivalence relation follows from the $H^2$ regularity of the stationary Stokes equations. 
Since $v^0\in D(A)$ {if and only if } $Av\in C([0,T];X)$,
it follows from \eqref{v0inD(A)2} that  
$$
D(A) = \DH^2(D)^d 
:=  \big\{\mbox{$v\in H^2(D)^d:  \nabla\cdot v=0$\, and\, $\D(v)\n \times \n =0$ on $\partial D$} \big\} .
$$
Moreover, from \eqref{v0inD(A)1} we see that the operator $A:D(A)\rightarrow X$ can be written as  
\begin{align}\label{A-operator}
Av= - \nabla \cdot \T (v, q_v)  
= - \Delta v + \nabla q_v ,
\end{align}
where $q_v$ is determined by $v$ through the following equation:   
\begin{align}\label{def-Domain-p}
\left\{\begin{aligned}
\Delta q_v &= 0 && \mbox{in}\,\,\,  D,\\
q_v &= 2\D(v)\n \cdot \n && \mbox{on}\,\,\,  \partial D .
\end{aligned}
\right.
\end{align}

For $v\in D(A)$, testing \eqref{A-operator} with $w$ and using integration by parts, we obtain by utilizing the boundary condition $q_v = 2\D(v)\n \cdot \n$ and $\D(v)\n \times \n =0$ (which imply $2\D(v)\n -q_v\n=0$ on $\partial D$) 
\begin{align}\label{Apz-operator}
\big(Av, w\big) = \big(2\D(v), \D(w)\big) \quad \forall\,\, w \in V. 
\end{align}
Therefore, the Stokes operator $A:D(A)\rightarrow X$ has an extension $A: V \to V'$ defined by \eqref{Apz-operator}. 

The boundary condition $\T (u, p) \n =0 $ in \eqref{spde} implies that $\D(u)\n\times\n=0$ and $p=2\D(u)\n\cdot \n$, and therefore 
$$
- P_X \T (u, p) =
- \Delta u + P_X\nabla p
= - \Delta u + \nabla (p-\eta) ,
$$
where $\eta$ is the solution of 
\begin{align}
\left\{\begin{aligned}
\Delta \eta  &= \Delta p && \mbox{in}\,\,\,  D\\
\eta &= 0 && \mbox{on}\,\,\,  \partial D .
\end{aligned}
\right.
\end{align}
This implies that 
\begin{align}\label{PX-Tup}
- P_X \T (u, p)  
= - \Delta u + \nabla q_u = Au ,
\end{align}
where $q_u=p-\eta$ is the harmonic function satisfying $q_u=2\D(u)\n\cdot \n$ on $\partial D$. 

As a result, applying $P_X$ to \eqref{spde} yields the following abstract formulation of the stochastic Stokes problem in \eqref{spde}:
\begin{equation}
\label{spde_abstact}
\left \{
\begin{aligned} 
\d u &=  [-Au + P_X f] \, \d t + P_X [B(u) \, \d W]  && \mbox{in}\,\,\,  D\times (0,T] ,\\
u&=u^0 && \mbox{at}\,\,\, D\times \{0\} .
\end{aligned}
\right .
\end{equation}
A predictable process $u\in C([0,T];L^2(\Omega; X))$ is called a mild solution of problem \eqref{spde_abstact} if 
\begin{align}\label{SPDE-mild}
u(t) 
= E(t)u^0 + \int_{0}^{t} E(t-s) P_X f(s)\d s
+ \int_{0}^{t}E(t-s) P_X B(u(s))\, \d W(s) \quad \P\mbox{-a.s.} 
\end{align}

The proof of the following proposition is presented in Appendix A. 

\begin{proposition}{{\rm (Well-posedness and regularity)}} \label{thm-u-stability}
Under Assumptions \ref{ass-W}--\ref{ass-u0}, problem \eqref{spde_abstact} has a unique mild solution in the sense of \eqref{SPDE-mild}. 
Moreover, the mild solution satisfies the following regularity estimates:
\begin{align}\label{SPDE-u-sta}
\sup_{t \in [0,T]}  \E \|u(t)\|_{H^2}^2 
&\lesssim \big(1+ \E\|u^0\|_{H^2}^2\big) , \\
\E \|u(t) - u(s)\|_{H^1}^2 &\lesssim \big(1+ \E\|u^0\|_{H^2}^2\big)(t - s)
&&\forall\, 0\le s\le t\le T
\label{SPDE-u-holder} \\[5pt]
\E \|u(t) - u(s)\|_{H^\beta}^2 &\lesssim  \big(1+ \E\|u^0\|_{H^2}^2\big)(t - s)^{2 - \beta} 
&&\forall\, 0\le s\le t\le T,\,\,\, \forall\, \mbox{$\beta\in(\frac{d}{2},2)$} .
\label{SPDE-u-holder-Hbeta} 
\end{align}
\end{proposition}

The error estimates for the numerical approximations will be proved based on the regularity results in Proposition \ref{thm-u-stability}. 

Since the mild solution has the regularity $u\in C([0,T]; L^2(\Omega,H^1(D)^d))$, the inf-sup condition \cite[Theorem 4.1]{girault1979finite} implies that there exits $\int_0^t p(s)\d s \in C([0,T]; L^2(\Omega,L^2(D)))$ satisfying following relation: 
\begin{align}\label{spde-weak}
\Big(\int_0^t p(s)\,\d s,\nabla\cdot v\Big) 
=&\, (u(t)-u^0,v)
+   2\int_0^t\big(\D(u(s)), \D(v)\big)\,\d s  \notag \\[5pt] 
&\, - \int_0^t (f(s),v)\,\d s - \Big(\int_0^t B(u(s))\,\d W(s) , v\Big) \quad\forall \, v \in H^{1}(D)^d . 
\end{align}

\section{\bf Estimates for the discrete semigroups} \label{sec:est-dis-semigroup}

In this section, we establish some technical estimates of the discrete analytic semigroup associated to the Stokes operator. The results will be used to prove the optimal-order convergence of the numerical solution in Section \ref{sec:spde-fully-fem}. 

The main results of this section are the following three types of error estimates about approximating the  semigroup $E(t)P_X$ by the discrete semigroup $\bar E_{h, \tau}^nP_{X_h}$, $n=1,\dots,N$, defined in \eqref{def-E_h-tau}. These results are the key to the error analysis for the stochastic Stokes problem.

\begin{lemma}\label{lem:fem-ferrs}For any  $v \in H^1(D)^d$,  there holds 
\begin{align}\label{lem:fem-ferrs-L2-int}
\sum_{j = 1}^{n} \int_{t_{j-1}}^{t_j}\|[ E(s)P_X - \bar E_{h, \tau}^j P_{X_h}] \, v\|_{L^2}^2 \,\d s  \le C(\tau +  h^4) \|v\|_{H^1}^2.
\end{align} 
In addition, for all $v \in L^2(D)^d$,  there holds
\begin{align}\label{lem:fem-ferrs-L2}
\Big\|\sum_{j = 1}^{n} \int_{t_{j-1}}^{t_j} [E(s)P_X - \bar E_{h, \tau}^j P_{X_h}] \, v \,\d s \Big \|_{L^2}^2 \le C(\tau + h^4) \|v\|_{L^2}^2,
\end{align}
and
\begin{align}\label{lem:fem-ferrs-H1}
\Big\|\sum_{j = 1}^{n}  \int_{t_{j-1}}^{t_j} \nabla[E(s)P_X - \bar E_{h, \tau}^j P_{X_h}] \, v\,  \d s \Big \|_{L^2}^2 \le C(\tau + h^2) \|v\|_{L^2}^2.
\end{align} 
\end{lemma}

\begin{remark}\upshape\label{remark41}
Note that the error estimates of $E(t)P_X - \bar E_{h, \tau}^n P_{X_h}$ in \cite{thomee2007galerkin} holds for an abstract parabolic equation, including the Stokes equations, but requires the condition $v \in X$. Since we only require $B(u): L^2(D)^d\rightarrow L^2(D)^d$ rather than $B(u): L^2(D)^d\rightarrow X$, we cannot apply the results in Lemma \ref{lem:fem-ferrs-L2-int} to the stochastic Stokes equations under the condition $v\in X$. By dropping this condition in Lemma \ref{lem:fem-ferrs-L2-int}, we manage to prove second-order convergence in space for the numerical solution of the stochastic Stokes equations.  
\end{remark}

In order to prove Lemma \ref{lem:fem-ferrs}, we need to extend the $H^1$-stability of the $L^2$ projection $P_{X_h}$ from $v\in X\cap H^1(D)^d$ to $v\in H^1(D)^d$. This is obtained by characeterizing the orthogonal complement of $X_h$ in $V_h$, as discussed in subsection \ref{subsec:orth-dicomp-V_h}. The proof of Lemma \ref{lem:fem-ferrs} is presented at the end of this section after introducing the orthogonal decomposition and some properties of the discrete semigroup.

\subsection{Orthogonal complement of $X_h$ in $V_h$}\label{subsec:orth-dicomp-V_h}

Let $X_h^{\perp}$ be the orthogonal complement of $X_h$ in $V_h$, i.e., $V_h = X_h \oplus X_h^{\perp}$, namely, any $v_h \in V_h$ has an orthogonal decomposition 
\begin{align}\label{orth-decom-Vh}
v_h = w_h + z_h \quad \mbox{with} \quad w_h \in X_h\,\,\,\mbox{and}\,\,\, z_h \in X_h^{\perp}\,\,\, \mbox{satisfying}\,\,\,  (w_h, z_h) = 0.
\end{align}
This decomposition  {(see \cite{breit2021convergence} for a slightly different presentation)}  is stable in the $H^1$ norm, as shown in the following lemma. 
\begin{lemma}\label{Lemma:H1PXh}
The orthogonal decomposition in \eqref{orth-decom-Vh} is stable in the $H^1$ norm, i.e., 
\begin{align}\label{orth-decom-H1-sta}
\|w_h\|_{H^1} + \|z_h\|_{H^1}\lesssim \|v_h\|_{H^1} \quad \forall \, \,v_h \in V_h. 
\end{align}
\end{lemma}
\begin{proof}
Let $P_{Q_h}: L^2(D) \to Q_h$ be the $L^2$-orthogonal projection onto $Q_h$.
For any given $v_h \in V_h$, the inf-sup condition \eqref{inf-sup} implies that there exists a unique solution $(w_h, q_h) \in V_h \times Q_h$ of the following equations:
\begin{align*}
(w_h, a_h) - (q_h, \nabla \cdot a_h) &= (v_h, a_h) &&\forall\, a_h \in V_h,&\\
(\nabla \cdot w_h, \eta_h) &= 0 &&\forall\, \eta_h \in Q_h.&
\end{align*}
Then $w_h \in X_h$  and
$$(w_h, a_h) = (v_h, a_h) \qquad \forall \,\,a_h \in X_h.$$
Let $z_h = v_h - w_h$, then $w_h$ and $z_h$ are the functions in the orthogonal decomposition \eqref{orth-decom-Vh}.
Before estimating $(w_h, q_h)$ directly, we first introduce $(w, q) \in  H^1(D)^d \times H^1(D)$ to be the solution of the continuous problem
\begin{align*}
(w, a) - (q, \nabla \cdot a) &= (v_h, a)  &&\forall\, a \in H^1(D)^d,\\
(\nabla \cdot w,  \eta) &= 0 &&\forall\, \eta \in L^2(D).
\end{align*}
Via integration by parts in the equation of $w$, one can obtain that $w = v_h - \nabla q$, with $q \in H^1(D)$
being the weak solution of
\begin{equation}
\left \{
\begin{aligned} 
\Delta q &= \nabla \cdot v_h \qquad &&{\rm in}\, \, D,  \\
q &= 0 \qquad &&{\rm on}\, \, \partial D.
\end{aligned}
\right .
\end{equation} 
The equation above has the standard $H^2$ elliptic regularity, i.e.,
\begin{align}\label{orth-H1-sta-q}
\|q\|_{H^2}\lesssim \|\nabla \cdot v_h\|_{L^2}
\lesssim \|v_h\|_{H^1},
\end{align}
which implies that 
\begin{align}\label{orth-H1-sta-w}
\|w\|_{H^1}= \|v_h - \nabla q\|_{H^1}
\lesssim \|v_h\|_{H^1}
\end{align}

Denote $\theta_h = w_h - \Pi_h w$ and $\phi_h = q_h - P_{Q_h} q$, and
consider the difference between the equations of $w_h$ and $w$, i.e.,
\begin{align*}
(\theta_h, a_h) - ( \phi_h, \nabla \cdot a_h) &=(w - \Pi_h w, a_h) + (P_{Q_h} q - q, \nabla \cdot a_h) &&\forall \, a_h \in V_h,&\\
(\nabla \cdot \theta_h, \eta_h) &= 0 &&\forall \, \eta_h \in Q_h.&
\end{align*}
Substituting $a_h = \theta_h$ and $\eta_h = \phi_h$ into the equations above and using the inverse inequality \eqref{app-inverse} we obtain
\begin{align*}
\|\theta_h\|_{L^2}^2 =\, & (w - \Pi_h w, \theta_h) + (P_{Q_h} q - q, \nabla \cdot \theta_h)\\
\le\, & \|w - \Pi_h w\|_{L^2}\|\theta_h\|_{L^2}
+ \|P_{Q_h} q - q\|_{L^2} \|\nabla \cdot \theta_h\|_{L^2}\\
\lesssim\, & h\|w\|_{H^1} \|\theta_h\|_{L^2} + h^2\|q\|_{H^2}\|\theta_h\|_{H^1}\\
\lesssim\, & h\big(\|w\|_{H^1} + \|q\|_{H^2}\big)\|\theta_h\|_{L^2},
\end{align*}
which together with \eqref{orth-H1-sta-q} and \eqref{orth-H1-sta-w} implies
\begin{align*}
\|\theta_h\|_{H^1} \lesssim h^{-1} \|\theta_h\|_{L^2}
\lesssim \|w\|_{H^1} + \|q\|_{H^2} \le \|v_h\|_{H^1}.
\end{align*}
Therefore, the $H^1$-stability \eqref{Fortin-Pih-H1-sta} of the Fortin projection operator gives
\begin{align*}
\|w_h\|_{H^1} = \|\theta_h + \Pi_h w\|_{H^1} \lesssim \|v_h\|_{H^1},
\end{align*}
which yields
\begin{align*}
\|z_h\|_{H^1} = \|v_h - w_h\|_{H^1} \lesssim  \|v_h\|_{H^1}.
\end{align*}
This proves the desired $H^1$-stability result in \eqref{orth-decom-H1-sta}.
\end{proof}
\begin{remark}\upshape 
The $H^1$ stability in Lemma \ref{Lemma:H1PXh} implies the following properties: 
\begin{align}\label{H1-stability-PXh}
\|P_{X_h}v\|_{H^1} &\le C\|v\|_{H^1}  &&\mbox{for}\,\,\,v\in H^1(D)^d,\\[5pt] 
\label{H2-app-PXh2}
\|v-P_{X_h}v\|_{H^1} \lesssim \inf_{v_h \in X_h} \|v-v_h\|_{H^1} 
&\le Ch\|v\|_{H^2} &&\mbox{for}\,\,\,v \in \DD(A) ,\\
\label{H2-app-PXh1}
\|v-P_{X_h}v\|_{L^2} \lesssim \inf_{v_h \in X_h} \|v-v_h\|_{L^2} 
&\le Ch^2 \|v\|_{H^2} &&\mbox{for}\,\,\, v \in \DD(A) . 
\end{align}
\end{remark}

For any $v\in \DD(A)$, we denote by $q_v$ the solution of \eqref{def-Domain-p}, and denote by 
$R_{X_h}: V \rightarrow X_h$ the Stokes--Ritz projection defined by 
\begin{align}\label{def-Stokes-Ritz}
\big(v-R_{X_h}v , w_h\big) + \big(A_h(v-R_{X_h}v), w_h\big) -(q_v, \nabla\cdot w_h) = 0\qquad\forall\, w_h\in X_h,
\end{align}
which is equivalent to finding $(R_{X_h}v,q_{v,h})\in V_h\times Q_h$ satisfying  
\begin{align}
\begin{aligned}
\big(v-R_{X_h}v , w_h\big) + 2\big(\D(v-R_{X_h}v), \D(w_h) \big)-(q_v - q_{v,h}, \nabla\cdot w_h)&=0 \quad
&&\forall\, w_h\in V_h ,\\
\big(\nabla\cdot (v-R_{X_h}v), q_h\big)&=0\quad
&&\forall\, q_h\in Q_h . 
\end{aligned}
\end{align}
The Stokes--Ritz projection has the following error bound 
(cf. \cite[Lemma 2.44 and
Lemma 2.45]{ern2004theory} and \cite[Proposition 4.18]{ern2004theory}):
\begin{align}\label{Error-Stokes-Ritz}
\|v-R_{X_h}v\|_{L^2} + h \|v-R_{X_h}v\|_{H^1} 
\lesssim h^2 \big(\|v\|_{H^2} + \|q_v\|_{H^1} \big) \le Ch^2 \|v\|_{H^2}\quad\forall\,\, v\in \DD(A).
\end{align}

\subsection{Fractional powers of $I+A$}

Since the Stokes operator $A$ defined in \eqref{A-operator} is self-adjoint and positive semi-definite, it follows that the operator $z + A$ is invertible on $X$ for  $z\in \C\backslash (-\infty, 0]$, and $-A$ generates a bounded analytic semigroup $E(t)=e^{-tA}$ on $X$; see \cite[Example 3.7.5]{ABHN2011}. 
The fractional powers of the positive definite operator $I+A$ (with compact inverse) can be defined by means of the spectral decomposition (see \cite[Appendix B.2]{kruse2014strong}), and the following norm equivalence holds. 

\begin{lemma}\label{lem-IA-sta}
The following equivalence relations hold for $s \in [0,2]${\rm:} 
\begin{align}
\label{IA-eqi-norms} 
\qquad \|v\|_{H^s} &\lesssim \|(I+A)^{\frac{s}{2}}v\|_{L^2} \lesssim \|v\|_{H^s}  &&\forall\, v \in \DD(A^{\frac{s}{2}})  ,  \\
\label{IA-eqi-norms-dual}
\qquad \|v\|_{H^{-s}}& \lesssim \|(I+A)^{-\frac{s}{2}}v\|_{L^2} \lesssim \|v\|_{H^{-s}}  &&\forall\, v \in \DD(A^{\frac{s}{2}}) ' ,
\end{align}
where $\DD(A^{\frac{s}{2}})=(X, \DD(A))_{[s/2]}$ for $s\in[0,2]$. 
\end{lemma}
\begin{remark}
\upshape
$(X, \DD(A))_{[\theta]}$, with $\theta\in[0,1]$, denotes the complex interpolation spaces between $X$ and $\DD(A)$. Since $\DD(A^{\frac{1}{2}})=X\cap H^1(D)^d$, via complex interpolation it can be shown that $\DD(A^{\frac{s}{2}})=X\cap H^s(D)^d$ for $s \in [0, 1]$. 
\end{remark}

The proof of Lemma \ref{lem-IA-sta} is presented in Appendix B. 

The analyticity of the semigroup $E(t)=e^{-tA}$ implies the following results. 

\begin{lemma}\label{lem-A:sg-sta} 
For any $v \in X$, the following  results hold for $t\in(0,T]${\rm:} 
\begin{align}
\label{A:sg-sta1}
\|(I+A)^{\gamma}E(t)v\|_{L^2} &\lesssim  t^{-\gamma}\|v\|_{L^2} &&\forall \, \gamma \ge 0,&\\
\label{A:sg-sta2}
\|(I+A)^{-\mu} (I- E(t))v\|_{L^2} &\lesssim t^{\mu}\|v\|_{L^2}  &&\forall \,\mu \in [0, 1] .&
\end{align}  
The following  results hold for $0\le t_1 \le t_2 \le T$:
\begin{align}
\label{A:sg-sta3}
\int_{t_1}^{t_2}\|(I+A)^{\frac{\rho}{2}}E(t_2 - s)v\|_{L^2}^2 \,\d s &\lesssim (t_2 - t_1)^{1 - \rho} \|v\|_{L^2}^2   &&\forall \, \rho  \in [0, 1],&\\
\label{A:sg-sta4}
\Big\|(I+A)^{\rho}\int_{t_1}^{t_2}E(t_2 - s)v \,\d s \Big\|_{L^2} &\lesssim(t_2 - t_1)^{1 - \rho} \|v\|_{L^2}  &&\forall \,\rho  \in [0, 1].&
\end{align}  
\end{lemma}


The proof of Lemma \ref{lem-A:sg-sta} is also presented in Appendix B.

\subsection{The discrete semigroup from spatial discretization}
Let $0 = \lambda_{h,1} \le \lambda_{h,2} \le \cdots \le \lambda_{h,M_h}$ be the eigenvalues of the discrete Stokes operator $A_h: X_h \to X_h$ with corresponding orthonormal eigenvectors $\{\varphi_{h,j}\}_{j=1}^{M_h} \subset X_h$ and $\dim(X_h) = M_h$. 
The operator $-A_h$ generates a  bounded analytic semigroup on $X_h$ defined by
\begin{align}\label{semigroup-Eh-t}
E_h(t) = e^{-tA_h},
\end{align}
which can be expressed as $E_h(t)v_h = \sum_{j = 1}^{M_h} e^{-t \lambda_{h,j}}(v_h, \varphi_{h,j})\varphi_{h,j}$ for all $v_h \in X_h$.
Hence, the following two properties can be derived based on the proof of \cite[Lemma 3.9]{thomee2007galerkin} and \cite[Lemma 3.2 (iii)]{kruse2012optimal}, respectively:
\begin{align}\label{Ah:sg-sta-E}
\qquad\qquad\|(I+A_h)^{\frac{\gamma}{2}}E_h(t)v_h\|_{L^2} &\lesssim t^{-\frac{\gamma}{2}}\|v_h\|_{L^2}  && \mbox{for} \,\,\, v_h \in X_h, \, \, \gamma \in [0, 1],&\\
\label{Ah:sg-sta-int}
\qquad\qquad\int_0^t\|(I+A_h)^{\frac12}E_h(s) v_h\|_{L^2}^2 \d s &\lesssim \|v_h\|_{L^2}^2, && \mbox{for} \,\,\, v_h \in X_h.&
\end{align}
Let 
\begin{align}
\label{semi-dis-approx:Phi}
\Phi_h(t): = E(t)P_{X} - E_h(t)P_{X_h}\qquad \mbox{for} \quad t \in [0, T],
\end{align}
which is the error in approximating the continuous semigroup. 
The main results of this subsection are the estimates of $\Phi_h(t)$ presented in the following lemma. 

\begin{lemma}\label{lem:FEM-errs}
For $v \in  H^1(D)^d$, the following estimates hold{\rm:}
\begin{align}\label{lem:FEM-errs-TH}
\big\|\Phi_h(t) v \big\|_{L^2}
\le 
Ct^{-\frac 12}h^2\|v\|_{H^1}, 
\end{align}
and
\begin{align}
\label{lem:FEM-errs-int2}
\Big(\int_0^t \big\|\Phi_h(s) v\big\|_{L^2}^2 \, \d s \Big)^{\frac12}
\le 
Ch^2 \|v\|_{H^1},
\end{align}
For $v \in L^2(D)^d$, the following estimates hold{\rm:}
\begin{align}\label{lem:FEM-errs-TH-l2}
\big\|\Phi_h(t) v \big\|_{L^2}
\le 
Ct^{-1}h^2\|v\|_{L^2}, 
\end{align}
and
\begin{align}
\label{lem:FEM-errs-int}
\bigg\|\int_0^t \Phi_h(s) v\, \d s\bigg\|_{L^2}  + 
h\Big\| \int_0^t \nabla\Phi_h(s) v \, \d s \Big\|_{L^2}
\le 
Ch^2 \|v\|_{L^2}.
\end{align}
If $v \in \DD(A^{\frac{1+\rho}{2}})$ then 
\begin{align}\label{lem:FEM-errs-In-u0}
\big\|\Phi_h(t) v \big\|_{L^2}+ h\big\|\nabla \Phi_h(t) v \big\|_{L^2}
\le Ct^{-\frac{1-\rho}{2}}h^2\|v\|_{H^{1+\rho}} \quad\mbox{for}\,\,\,\rho\in[0,1]. 
\end{align}
\end{lemma}
\begin{remark}\upshape 
{Lemma \ref{lem:FEM-errs} is proved later based on an orthogonal decomposition of $v = P_X v + \nabla q$ for $v\in L^2(D)^d$ at the end of this subsection, and it does not require $\nabla \cdot v = 0$ in the cases $v \in  L^2(D)^d$ and $v \in  H^1(D)^d$. This is important for the error analysis of the stochastic Stokes equations.} 
%
\end{remark}

The proof of Lemma~\ref{lem:FEM-errs} is based on error estimates for the semi-discrete FEM for the  deterministic linear Stokes problem  
\begin{equation}
\label{pde-linear-stokes}
\left \{
\begin{aligned} 
\partial_t u -\nabla \cdot \T(u, p) &= 0  \qquad &&{\rm in}\, \, D \times [0,T],\\
\nabla\cdot u&=0  \qquad &&{\rm in}\, \, D\times [0,T],\\
\T(u, p)\n &= 0 \qquad &&{\rm on}\, \, \partial D\times [0,T],\\
u(0) &= u^0 \qquad &&{\rm in}\, \, D.\\
\end{aligned}
\right .
\end{equation}
Applying $P_X$ to the first equation of \eqref{pde-linear-stokes}  yields the following
abstract formulation:
\begin{align}
\label{pde-u0-abs}
\partial_t u + A u = 0  \qquad \mbox{with}\qquad  u(0) = u^0 \in X,
\end{align}
of which the solution can be represented $u(t) = E(t)P_X u^0$ in terms of the semigroup generated by the Stokes operator.
The FEM for \eqref{pde-linear-stokes} reads: for given $u_h(0) = P_{X_h}u^0$, find $(u_h(t), p_h(t)) \in V_h \times Q_h$ such that 
\begin{equation}
\label{l-stokes-fem-weak}
\left \{
\begin{aligned} 
\big(\partial_t u_h(t), v_h\big) + 2\big(\D( u_h(t)), \D( v_h) \big)-(p_h(t),  \nabla \cdot v_h)&= 0 \qquad&& \forall\, v_h\in V_h ,\\
(\nabla\cdot u_h(t), q_h)&=0\qquad&& \forall\, q_h\in Q_h .
\end{aligned}
\right .
\end{equation}
This can be written into the following abstract form by choosing $v_h\in X_h$:
\begin{equation}\label{l-stokes-fem}
\frac{\d}{\d t} u_h(t) + A_h u_h(t) = 0 \quad \mbox{with} \quad u_h(0)=P_{X_h}u^0 , 
\end{equation}
of which the solution can be represented as $u_h(t) = E_h(t)P_{X_h}u^0$.



We shall estimate the error between the solutions of \eqref{pde-linear-stokes} and \eqref{l-stokes-fem-weak} by estimating the difference between the continuous and discrete resolvent operators, i.e., $(z+A)^{-1}$ and $(z+A_h)^{-1}$. This is presented in the following lemma. Its proof is presented in Appendix B. 
\begin{lemma}\label{lem:FEM-errs-w}
For any $f \in L^2(D)^d$, let $w$ and $w_h$ be the solution and finite element solution of the Stokes equations 
\begin{equation}
\label{pde-w}
\left \{
\begin{aligned} 
zw -\nabla \cdot \T(w, p) &= f  \qquad &&{\rm in}\, \, D,\\
\nabla\cdot w&=0  \qquad &&{\rm in}\, \, D,\\
\T(w, p)\n &= 0 \qquad &&{\rm on}\, \, \partial D,
\end{aligned}
\right .
\end{equation}
and
\begin{equation}
\label{pde-wh}
\left \{
\begin{aligned} 
(zw_h, v_h) + 2\big(\D( w_h), \D( v_h) \big)-(p_h,  \nabla \cdot v_h)&= (f, v_h) &&  \forall \, v_h \in V_h ,\\
(\nabla\cdot w_h, q_h)&=0&& \forall \, q_h \in Q_h ,\\
\end{aligned}
\right .
\end{equation}
where $z \in \Sigma_\phi:=\{1+z'\in\C: |{\rm arg}(z')|<\phi\}$ for some $\phi\in(0,\pi)$.
Then the following results hold{\rm:}
\begin{align}
\label{pde-w-err}
&\|w-w_h\|_{L^2} + h\|\nabla (w-w_h)\|_{L^2}\le 
Ch^2 (\|w\|_{H^2} + \|p\|_{H^1})  , \\
\label{reg-w-err}
&\|w\|_{H^2} + \|p\|_{H^1}\le 
C\|f\|_{L^2} , 
\end{align}
where the constant $C$ is independent of $z\in \Sigma_\phi$ {\rm(}but may depend on $\phi)$. 
\end{lemma}

\begin{remark}\upshape 
The first equation in \eqref{pde-w} can be written as 
$$
\nabla p = f - (z-\Delta) w . 
$$
By applying $P_X$ to the first equation in \eqref{pde-w} and using relation $- P_X \T (w, p)  
= - \Delta w + \nabla q_w  $, as shown in \eqref{PX-Tup}, we obtain
$$
\nabla q_w =P_X f - (z-\Delta) w . 
$$
Combining the two equations above yields
$$
\nabla p = f-P_X f + \nabla q_w . 
$$
In the case $f\in X$, we obtain $p=q_w$ and therefore \eqref{pde-w-err} reduces to 
\begin{align}
\label{pde-w-err-2}
&\|w-w_h\|_{L^2} + h\|\nabla w-\nabla w_h\|_{L^2}\le 
Ch^2 \|w\|_{H^2} ,
\end{align}
where we have used the inequality $\|q_w\|_{H^1}\lesssim \|w\|_{H^2}$. 
\end{remark}

Before proving  Lemma \ref{lem:FEM-errs} for the general case $v\in H^1(D)^d$, we consider the simpler case $g\in X\cap H^1(D)^d$ in the following lemma.  
\begin{lemma}\label{lem:FEM-errs-u0}
It holds that 
\begin{equation}
\label{pde-u0-err}
\int_0^t \|\Phi_h(s)g\|_{L^2}^2\, \d s\le 
Ch^4\|g\|_{H^1}^2 \quad\forall\, g \in X \cap H^1(D)^d. 
\end{equation}
\end{lemma}
\begin{proof}
Let $(u,p)$ and $(u_h,p_h)$ be the solution of \eqref{pde-linear-stokes} 
and \eqref{l-stokes-fem-weak} with $u^0 = g \in X\cap H^1(D)^d$, respectively. 
Then the definition of $\Phi_h(s)=E(s)P_{X} - E_h(s)P_{X_h}$ in \eqref{semi-dis-approx:Phi} implies that 
\begin{equation}
\label{def-Phih=u-u_h}
\Phi_h(s)g = u(s) - u_h(s) . 
\end{equation}
Subtracting \eqref{l-stokes-fem-weak} from \eqref{pde-linear-stokes} yields
\begin{equation}
\left \{
\begin{aligned} 
\big(\partial_t (u - u_h), v_h\big) + 2\big(\D(u - u_h), \D( v_h)\big) - (p - p_h,  \nabla \cdot v_h)&= 0 \quad&&  \forall \, v_h \in V_h ,\\
(\nabla\cdot (u-u_h), \eta_h)&=0 \quad&& \forall \, \eta_h \in Q_h,
\end{aligned}
\right . \notag
\end{equation}
and therefore, choosing $v_h \in X_h$ and using the Ritz projection $R_{X_h}$ defined in \eqref{def-Stokes-Ritz}, we have 
\begin{align}
\label{u0-err-eqn-abs}
&\big(\partial_t (u - u_h), v_h\big) + \big((I+A_h)(R_{X_h}u - u_h),  v_h\big) 
\notag \\ 
&= \big(u - u_h ,  v_h\big) + (q - P_{Q_h} q,  \nabla \cdot v_h)  - (q_u - P_{Q_h}q_u, \nabla \cdot v_h) . 
\end{align} 
{where $q_u$ is defined in \eqref{def-Domain-p} (with $v$ replaced by $u$ there).}
By denoting $\widetilde e_h = P_{X_h}u - u_h \in X_h$ and choosing $v_h = (I + A_h)^{-1} \widetilde e_h \in X_h $ in \eqref{u0-err-eqn-abs}, we obtain 
\begin{align} \label{err-u0-result0}
\frac12 \frac{\d}{\d t}&\, \|(I + A_h)^{-\frac12} \widetilde e_h\|_{L^2}^2 + \|\widetilde e_h\|_{L^2}^2 \\
=&\, 
\|(I + A_h)^{-\frac12} \widetilde e_h\|_{L^2}^2
+\big(P_{X_h}u - R_{X_h}u ,\widetilde e_h\big)  
+ R_{q}(\widetilde e_h) + R_{q_u}(\widetilde e_h) \notag\\
\le&\, \| (I + A_h)^{-1}\widetilde e_h\|_{L^2}^2  
+ C\|R_{X_h}u - P_{X_h}u\|_{L^2}^2
+ \frac14\|\widetilde e_h\|_{L^2}^2 + |R_{q}(\widetilde e_h) | + |R_{q_u}(\widetilde e_h)| \notag\\
\le&\,  \| (I + A_h)^{-1}\widetilde e_h\|_{L^2}^2  
+ Ch^4\|u\|_{H^2}^2  + \frac14 \|\widetilde e_h\|_{L^2}^2  + | R_{q}(\widetilde e_h) | + |R_{q_u}(\widetilde e_h)|  , \notag
\end{align}
where
$$
R_{q}(\widetilde e_h) 
= \big( q - P_{Q_h} q, \nabla \cdot[(I + A_h)^{-1}\widetilde e_h] \big) \quad\mbox{and} 
\quad R_{q_u} (\widetilde e_h) =-\big(q_u - P_{Q_h} q_u, \nabla \cdot[(I + A_h)^{-1}\widetilde e_h ]\big).
$$

The remainders $R_{q}(\widetilde e_h) $ and $R_{q_u}(\widetilde e_h) $ can be estimated by considering $w = (I+A)^{-1}P_X\widetilde e_h$ and $w_h = (I+A_h)^{-1}P_{X_h} \widetilde e_h$, which are the solutions of \eqref{pde-w} and \eqref{pde-wh}, respectively, with $z=1$ and $f=\widetilde e_h$.
Then applying Lemma \ref{lem:FEM-errs-w} yields 
\begin{align*}
\big\|\big[(I + A)^{-1} P_X - (I + A_h)^{-1}P_{X_h}\big] \widetilde e_h \big\|_{H^1} \lesssim h\|\widetilde e_h\|_{L^2} ,
\end{align*}
which implies that 
\begin{align}\label{err-u0-Rq}
|R_{q}(\widetilde e_h) | = &\,  
\big| \big(q - P_{Q_h} q, \nabla \cdot\big[(I + A)^{-1}P_X \widetilde e_h\big] \big) \big| \notag \\
&\, +\big| \big(q - P_{Q_h} q, \nabla \cdot \big[(I + A)^{-1} P_X- (I + A_h)^{-1}P_{X_h}\big] \widetilde e_h \big) \big|  \notag \\
\lesssim \,& \|q - P_{Q_h} q\|_{H^{-1}} \|\nabla \cdot\big[(I + A)^{-1}P_X \widetilde e_h \big]\|_{H^1} \notag \\
&\,  + \|q - P_{Q_h} q\|_{L^2} \big\|\big[(I + A)^{-1} P_X - (I + A_h)^{-1}P_{X_h}\big] \widetilde e_h\big\|_{H^1} 
\notag \\
\lesssim \, & h^2 \|q\|_{H^1} \|(I + A)^{-1}P_X \widetilde e_h\|_{H^2} + h \|q\|_{H^1} h \|\widetilde e_h\|_{L^2} \notag\\
\lesssim\, & h^2 \|q\|_{H^1}\|\widetilde e_h\|_{L^2} \notag\\
\le\,& Ch^4 \|q\|_{H^1}^2 + \frac{1}{8}\|\widetilde e_h\|_{L^2}^2 ,
\end{align}
where we have used the inequality  
$\|q - P_{Q_h} q\|_{H^{-1}} \lesssim h^2\|q\|_{H^1}$ {and Young's inequality}; see Property (2) of Section \ref{section:regularity}. 
Similarly, since $\|q_u\|_{H^1}^2 \lesssim \|u\|_{H^2}$, it follows that 
\begin{align}\label{err-u0-R-eta_u}
| R_{q_u}(\widetilde e_h) | 
\le Ch^4 \|q_u\|_{H^1}^2 + \frac{1}{8}\|\widetilde e_h\|_{L^2}^2
\le Ch^4 \|u\|_{H^2}^2 + \frac{1}{8}\|\widetilde e_h\|_{L^2}^2.
\end{align}
Substituting \eqref{err-u0-Rq}--\eqref{err-u0-R-eta_u} into \eqref{err-u0-result0},  and 
using Gronwall's inequality with $\widetilde e_h(0) = 0$, we obtain  
\begin{align}\label{err-u0-result1}
\|(I + A_h)^{-\frac12}\widetilde e_h(t)\|_{L^2}^2 + \int_0^t\|\widetilde e_h(s)\|_{L^2}^2\,\d s 
\lesssim h^4 \int_0^t \big( \|u\|_{H^2}^2 +  \|q\|_{H^1}^2\big) \,\d s
&\lesssim h^4 \|u^0\|_{H^1}^2 \notag\\
& = h^4 \|g\|_{H^1}^2 ,
\end{align}
where the last inequality is the basic energy inequality for the Stokes equations, which can be obtained by testing \eqref{pde-u0-abs} with $Au$, respectively. 
Since $u(s)-u_h(s)=u(s)-P_{X_h}u(s)+\widetilde e_h$, by using the inequality above and the triangle inequality we obtain 
\begin{align*} 
\int_0^t\|u(s)-u_h(s)\|_{L^2}^2\,\d s 
\lesssim h^4 \|g\|_{H^1}^2 .
\end{align*}
This proves the result of Lemma \ref{lem:FEM-errs-u0} in view of \eqref{def-Phih=u-u_h}.  
\end{proof}
\bigskip

Now we prove Lemma \ref{lem:FEM-errs} by utilizing the result of Lemma \ref{lem:FEM-errs-u0}. 
\bigskip

{\noindent \bf Proof of Lemma \ref{lem:FEM-errs}.}
For $v \in L^2(D)^d$ we denote 
$$
\mbox{$w = (z+A)^{-1}P_X v$ \quad and \quad $w_h = (z+A_h)^{-1}P_{X_h} v$} ,
$$
which are the solutions of \eqref{pde-w} and \eqref{pde-wh} with $f = v$.
%

Since  $-A$ generates a bounded analytic semigroup on $X$, there exists an angle $\phi\in(\frac{\pi}{2},\pi)$ such that the operator $(z+A)^{-1}$ is analytic with respect to $z$ in the sector 
$\Sigma_{\phi}=\{z\in\C\backslash \{0\} : |{\rm arg}(z)|<\phi\}$. Moreover, the semigroup $E(t)=e^{-tA}$ can be expressed in terms of the resolvent operator $(z+A)^{-1}$ through a contour integral on the complex plane, i.e., 
\begin{align}\label{E(t)P_Xv}
E(t)P_Xv = \frac{1}{2\pi i}\int_{\Gamma}(z+A)^{-1}P_Xv e^{zt}\,\d z , \quad\mbox{with}\,\,\, \Gamma =\{1+z : |{\rm arg}(z)|=\phi\} \subset \Sigma_{\phi} . 
\end{align}
Similarly,
\begin{align}\label{E_h(t)P_{X_h}v}
E_h(t)P_{X_h}v = \frac{1}{2\pi i}\int_{\Gamma}(z+A_h)^{-1}P_{X_h}v e^{zt}\,\d z  . 
\end{align}
Therefore, 
\begin{align}\label{ww-transform}
\| \Phi_h(t)v \|_{L^2} &= \| E(t)P_X v - E_h(t)P_{X_h} v \|_{L^2} \notag \\
&= \bigg\| \frac{1}{2\pi i} \int_{\Gamma}\big[(z+A)^{-1}P_X v - (z+A_h)^{-1}P_{X_h} v\big]e^{zt}\,\d z \bigg\|_{L^2} \notag\\
&\lesssim \int_{\Gamma} \|w-w_h\|_{L^2}  e^{{\rm Re} (z) t}\, |\d z| \notag\\
&\lesssim \int_{\Gamma} h^2 \|v\|_{L^2}  e^{{\rm Re} (z) t}\, |\d z| 
\quad\mbox{(This follows from Lemma \ref{lem:FEM-errs-w} with $f=v$)}\notag\\
&\lesssim t^{-1}h^2 \|v\|_{L^2} \quad\mbox{for}\,\,\, t\in(0,T] .  
\end{align}  
This proves \eqref{lem:FEM-errs-TH-l2}.

If $v \in \DD(A^{\frac{1+\rho}{2}})$ then we use inequality \eqref{pde-w-err-2}, which implies that 
\begin{align}\label{ww-errs-M-l2}
\|w-w_h\|_{L^2} +  h \|\nabla (w-w_h)\|_{L^2} 
&\lesssim h^2 \|w\|_{H^2}  \notag\\ 
&= h^2 \|(z+A)^{-1} v\|_{H^2}  \notag\\ 
&\lesssim h^2\|(I+A)(z+A)^{-1} v\|_{L^2} \notag\\
&= h^2\big\|(I+A)^\frac{1-\rho}{2} (z+A)^{-1} (I+A)^\frac{1+\rho}{2}  v\big\|_{L^2}  \notag\\
&\lesssim h^2|z|^{-\frac{1+\rho}{2} }\big\|(I+A)^\frac{1+\rho}{2} v\big\|_{L^2}  \quad \forall \,\,\rho \in [0, 1], 
\end{align}
where we have used the interpolation inequality to get 
\begin{align}\label{ww-errs-A-z-theta}
\|(I+A)^{\theta}(z+A)^{-1}v\|_{L^2} \le C_{\theta} |z|^{-(1-\theta)}\|v\|_{L^2} \quad \forall \,\, \theta \in [0, 1],  \, \, v\in X.
\end{align}
Substituting \eqref{ww-errs-M-l2} into \eqref{ww-transform} and using \eqref{IA-eqi-norms} gives
\begin{align}
\|\Phi_h(t)v \|_{L^2} + h\|\nabla \Phi_h(t)v\|_{L^2} 
&\lesssim h^2 \big\|(I+A)^\frac{1+\rho}{2} v\big\|_{L^2} \int_{\Gamma}|z|^{-\frac{1+\rho}{2} }e^{{\rm Re}(z) t}\, |\d z|. \notag\\
&\lesssim t^{-\frac{1- \rho}{2}}h^2\|v\|_{H^{1+\rho}}\qquad \forall \,\,  \rho \in [0, 1],  \, \, v\in \DD(A^{\frac{1+\rho}{2}}) . \label{ww-errs-M-DF1}
\end{align}
This proves \eqref{lem:FEM-errs-In-u0}.

In order to prove \eqref{lem:FEM-errs-TH}--\eqref{lem:FEM-errs-int2}, we consider the orthogonal decomposition $v = P_X v + \nabla q$ for a function $v\in H^1(D)^d$, where $q\in H^1(D)$ is the weak solution of 
\begin{equation}
\left \{
\begin{aligned} 
\Delta q  & = \nabla \cdot v \qquad &&{\rm in}\, \, D   \\
q & = 0 \qquad &&{\rm on}\, \, \partial D,
\end{aligned}
\right . \notag
\end{equation} 
which has the classical $H^1_0(D)\cap H^2(D)$ elliptic regularity, i.e.,
\begin{align}\label{ww-errs-TH-sta}
\|q\|_{H^2} \lesssim \|\nabla \cdot v\|_{L^2} \lesssim \|v\|_{H^1} . 
\end{align}
Since $\Phi_h(t) \nabla q = E(t) P_X \nabla q - E_h(t)P_{X_h} \nabla q = - E_h(t)P_{X_h} \nabla q \in X_h$, 
by using the self-adjointness of $E_h(t)$ we have
\begin{align*} 
(\Phi_h(t) \nabla q , a_h)\notag
&=-(P_{X_h} \nabla q, E_h(t)a_h)  \\
&=( q , \nabla \cdot [E_h(t) a_h]) \,\,\,\quad\mbox{(since $E_h(t) a_h\in X_h$ and $q=0$ on $\partial D$)} \notag\\
&=( q - q_h , \nabla \cdot [E_h(t) a_h]) \quad\mbox{(for any $q_h\in H^1_0(D)\cap Q_h$)} \notag\\
&=- (E_h(t)  P_{X_h} \nabla (q - q_h), a_h ) \notag\\ 
&=- ((I+A_h)^{\frac12}E_h(t) (I+A_h)^{-\frac12} P_{X_h} \nabla (q - q_h), a_h ) ,
\end{align*}
which implies that
\begin{align} 
\label{ww-errs-des-q}
\Phi_h(t) \nabla q = - (I+A_h)^{\frac12}E_h(t) (I+A_h)^{-\frac12} P_{X_h} \nabla (q - q_h) 
\quad\forall\, q_h\in Q_h\cap H^1_0(D) . 
\end{align}
On the one hand, \eqref{ww-errs-des-q} can be combined with \eqref{Ah:sg-sta-E} to yield 
\begin{align*} 
\|\Phi_h(t) \nabla q\|_{L^2} 
\lesssim 
\inf_{q_h\in Q_h\cap H^1_0(D)} t^{-\frac12} \| (I+A_h)^{-\frac12} P_{X_h} \nabla (q - q_h) \|_{L^2} 
&\lesssim \inf_{q_h\in Q_h\cap H^1_0(D)} t^{-\frac12} \| q - q_h \|_{L^2} \\
&\lesssim t^{-\frac12} h^2 \|q\|_{H^2} ,
\end{align*}
where \eqref{app-pro-q-H10} is used. 
On the other hand, inequality \eqref{ww-errs-des-q} can be combined with \eqref{Ah:sg-sta-int} to yield 
\begin{align*} 
\bigg( \int_0^T \|\Phi_h(t) \nabla q\|_{L^2}^2 \d t \bigg)^{\frac12} 
\lesssim \inf_{q_h\in Q_h\cap H^1_0(D)}\| (I+A_h)^{-\frac12} P_{X_h} \nabla (q - q_h) \|_{L^2} 
&\lesssim \inf_{q_h\in Q_h\cap H^1_0(D)} \| q - q_h \|_{L^2} \\
&\lesssim   h^2 \|q\|_{H^2} . 
\end{align*}
Substituting \eqref{ww-errs-TH-sta} into the two inequalities above, we obtain 
\begin{align}\label{ww-errs-TH-q}  
\|\Phi_h(t) \nabla q\|_{L^2} \lesssim t^{-\frac12}h^2\|v\|_{H^1} \quad \mbox{and} \quad
\bigg(  \int_0^t\|\Phi_h(s) \nabla q\|_{L^2}^2\, \d s\bigg)^{\frac12}  \lesssim h^2\|v\|_{H^1} .
\end{align}

By using inequality \eqref{ww-errs-M-DF1} with $\rho=0$ (with $v$ replaced by $P_Xv$ therein) and Lemma \ref{lem:FEM-errs-u0} with $g=P_Xv$, we have  
\begin{align}\label{ww-errs-DF1}
\|\Phi_h(t)P_X v \|_{L^2} 
\lesssim t^{-\frac{1}{2}}h^2\|v\|_{H^1} \quad\mbox{and}\quad
\bigg( \int_0^t\|\Phi_h(s)P_Xv \|_{L^2}^2 \, \d s \bigg)^{\frac12}
\lesssim h^2\|v\|_{H^1}.
\end{align}
Then, combining \eqref{ww-errs-TH-q} and \eqref{ww-errs-DF1} in the decomposition $v=P_Xv+\nabla q$, we obtain 
\begin{align}
\|\Phi_h(t)v \|_{L^2} 
\lesssim t^{-\frac{1}{2}}h^2\|v\|_{H^1} \quad\mbox{and}\quad
\bigg( \int_0^t\|\Phi_h(s)v \|_{L^2}^2 \, \d s \bigg)^{\frac12}
\lesssim h^2\|v\|_{H^1}.
\end{align}
This proves \eqref{lem:FEM-errs-TH} and \eqref{lem:FEM-errs-int2}. 

It remains to prove \eqref{lem:FEM-errs-int}.
To this end, we use the expressions in \eqref{E(t)P_Xv} and \eqref{E_h(t)P_{X_h}v}, which imply that 
\begin{align*}
\int_0^t \Phi_h(s)v \,\d s
&= \int_0^t \frac{1}{2\pi i}\int_{\Gamma} [ (z+A)^{-1}P_Xv - (z+A_h)^{-1}P_{X_h}v ] e^{zs}\, \d z \d s \notag \\
&= \frac{1}{2\pi i}\int_{\Gamma} z^{-1}[ (z+A)^{-1}P_Xv - (z+A_h)^{-1}P_{X_h}v ] e^{zs}\, \d z  \notag \\
&= \frac{1}{2\pi i} \int_{ \Gamma}z^{-1}[w - w_h]e^{zt}\,\d z.
\end{align*}
By applying Lemma \ref{lem:FEM-errs-w} with $f = v$, we have
\begin{align}\label{ww-errs-int-H1}
&\Big\|\int_0^t \Phi_h(s) v \,\d s \Big\|_{L^2} 
+h\Big\|\nabla\int_0^t \Phi_h(s)v \,\d s \Big\|_{L^2} \notag \\
&\lesssim  \int_{ \Gamma} |z|^{-1} \big(\|w-w_h\|_{L^2} + h\|\nabla w-\nabla w_h\|_{L^2}\big) e^{{\rm Re}(z) t}\, |\d z| \notag\\
&\lesssim h^2\|v\|_{L^2} \int_{ \Gamma} |z|^{-1} e^{{\rm Re}(z) t}\, |\d z| 
\lesssim h^2\|v\|_{L^2} .
\end{align}
This proves \eqref{lem:FEM-errs-int}. 
\hfill\endproof

\subsection{The discrete semigroup in the full discretization}
Let $\lambda_j^*\ge 0$ and $\phi_j^*$, $j=1,2,\dots$, be the eigenvalues and eigenfunctions of the operator $A:D(A)\rightarrow X$. Similarly, let $\lambda_{h,j}^*\ge 0$ and $\phi_{h,j}^*$, $j=1,\dots,M_h$, be the eigenvalues and eigenfunctions of the operator $A_h:X_h\rightarrow X_h$. 

We denote by $R(\tau A_h):X_h\rightarrow X_h$ the linear operator defined by 
\begin{align}\label{def-rational-R}
R(\tau A_h)v_h = \sum_{j = 1}^{M_h} R(\tau \lambda_{h,j}^*)(v_h, \phi_{h,j}^*)\phi_{h,j}^* \quad \mbox{with\, $R(z) = \dfrac{1}{1+z}$\, for\, $z \in \R$,\, $z \neq -1$.}
\end{align} 
The discrete semigroup $\bar E_{h, \tau}$ defined in \eqref{def-E_h-tau} can be written as
$$
\bar E_{h, \tau} v_h = R(\tau A_h)v_h\quad \mbox{for}\quad v_h \in X_h.
$$
As a {time discrete version} of \eqref{Ah:sg-sta-E} and \eqref{Ah:sg-sta-int}, the following estimates hold for any $v_h \in X_h$,
\begin{align}\label{E_ht:sg-sta}
\|(I+A_h)^{\frac{\gamma}{2}}\bar E_{h, \tau}^n v_h\|_{L^2} \lesssim &\, t_n^{-\frac{\gamma}{2}}\|v_h\|_{L^2} 
&& \mbox{for} \,\, 1\le n\le N\,\,\,\mbox{and}\,\,\, \gamma \in [0, 1],\\
\label{E_ht:sg-sta-sum-h1}
\tau\sum_{j = 1}^n\|(I+A_h)^{\frac12}\bar E_{h, \tau}^j v_h\|_{L^2}^2  \lesssim &\, C\|v_h\|_{L^2}^2 
&& \mbox{for} \,\, 1\le n\le N .
\end{align}
\begin{remark}\upshape 
Inequality \eqref{E_ht:sg-sta} is equivalent to  
\begin{align}\label{sg-sta-lambda}
|(1+\lambda_{h,j}^*)^{\frac{\gamma}{2}} R(\tau \lambda_{h,j}^*)^n| \lesssim t_n^{-\frac{\gamma}{2}} 
\quad\mbox{for}\,\,\, \lambda_{h,j}^*\ge 0 \,\,\,\mbox{and}\,\,\, t_n=n\tau \in[0,T] .
\end{align}
The proof of this inequality can be found in \cite[Lemma 7.3]{thomee2007galerkin}). 
Since the function $w_h^j=\bar E_{h, \tau}^j v_h$, $n=1,2,\dots$, are solutions of the equation
$$
\frac{w_h^j-w_h^{j-1}}{\tau} + A_h w_h^j = 0 . 
$$ 
Testing the equation with $w_h^j$ and summing up the results for $j=1,\dots,n$, yield the basic energy inequality 
$$
\max_{1\le j\le n} \frac12 \|w_h^j\|_{L^2}^2 
+ \tau\sum_{j= 1}^n\|A_h^{\frac12} w_h^j\|_{L^2}^2 
\le \frac12 \|w_h^0\|_{L^2}^2 = \frac12 \|v_h\|_{L^2}^2 ,
$$
which implies \eqref{E_ht:sg-sta-sum-h1}. 
\qed
\end{remark}

The error of $\bar E_{h, \tau}^n P_{X_h}$ in approximating $E_h(t_n)P_{X_h} $ is discussed in the following lemma. 

\begin{lemma}\label{lem:D-Fully-errs}
Let $\bar \Phi_{h, \tau}^n v := E_h(t_n)P_{X_h} v -\bar E_{h, \tau}^n P_{X_h} v$ for $v\in L^2(D)^d$.
Then the following estimates hold: 
\begin{align}\label{lem:D-Fully-err-h01}
& \big\|\bar \Phi_{h, \tau}^n v \big\|_{L^2}
\lesssim  \tau^{\frac12}\|v\|_{H^1} &&\forall\, v\in H^1(D)^d, \\ 
\label{lem:D-Fully-err-l2}
& \big\|\bar \Phi_{h, \tau}^n v \big\|_{L^2} 
\lesssim  t_n^{-\frac12}\tau^{\frac12} \|v\|_{L^2}  &&\forall\, v\in L^2(D)^d , \\
\label{lem:D-Fully-err-sum}
& \Big(\tau \sum_{j = 1}^n \big\|\bar \Phi_{h, \tau}^j v\big\|_{L^2}^2 \Big)^{\frac12}
+ \Big\|\tau \sum_{j = 1}^n \nabla \bar \Phi_{h, \tau}^j v\Big\|_{L^2}
\lesssim  \tau^{\frac12} \|v\|_{L^2} &&\forall\, v\in L^2(D)^d . 
\end{align}
\end{lemma}

\begin{proof}
It is known that the function $F_n(z) = e^{-nz} - R(z)^n$ has the following upper bound:
\begin{align}\label{D-Ferr-l2-lambda}
|F_n(\tau \lambda_{h,j})| \lesssim (\tau\lambda_{h,j}^*)^{\frac12}  \quad \mbox{for the eigenvalues $\lambda_{h,j}\ge 0$.} 
\end{align}
A proof of this result can be found in \cite[Theorem 7.1]{thomee2007galerkin} (with $q = \frac12$ therein). 
Inequality immediately implies the following error estimate: 
\begin{align}\label{D-Ferr-l2-norm} 
\big\| (E_h(t_n)-\bar E_{h, \tau}^n) v_h \big\|_{L^2} = 
\|F_n(\tau A_h) v_h\|_{L^2} \lesssim \tau^{\frac12}\|v_h\|_{H^1} \quad \forall\, v_h \in X_h .       
\end{align}
In view of the $H^1$-stability of $P_{X_h}$ in \eqref{H1-stability-PXh}, we have 
\begin{align}
\big\| (E_h(t_n)-\bar E_{h, \tau}^n) P_{X_h} v\big\|_{L^2}  
\lesssim \tau^{\frac12}\|P_{X_h} v\big\|_{H^1}  
\lesssim \tau^{\frac12}\|v\|_{H^1} \quad \forall \, v \in H^1(D)^d .
\end{align}
This proves \eqref{lem:D-Fully-err-h01}.

Similarly, the following inequality can be shown (see \cite[inequality (7.22)]{thomee2007galerkin} with $q = \frac12$)
$$
|F_n(\tau \lambda_{h,j})| \lesssim \tau^{\frac12}t_n^{-\frac12} \quad \mbox{for the eigenvalues $\lambda_{h,j}\ge 0$},
$$
which immediately implies that 
\begin{align}
\big\| (E_h(t_n)-\bar E_{h, \tau}^n) P_{X_h} v \big\|_{L^2}  =
\|F_n(\tau A_h) P_{X_h} v \|_{L^2} 
\lesssim \tau^{\frac12}t_n^{-\frac12}\|P_{X_h} v\big\|_{L^2}   \quad \forall \, v \in L^2(D)^d .
\end{align}
This proves \eqref{lem:D-Fully-err-l2}.

The first and second terms in \eqref{lem:D-Fully-err-sum} were estimated in \cite[inequality (4.21) with $\rho=0$]{kruse2014optimal} and \cite[inequality (4.19) with $\rho=1$, replacing $A_h^{-\frac12}P_hx$ by $P_hx$ therein]{kruse2014optimal}, respectively.
\end{proof}

\bigskip
{\noindent \bf Proof of Lemma~\ref{lem:fem-ferrs}.}
By using the expressions 
$$
\Phi_h(t): = E(t)P_{X} - E_h(t)P_{X_h} \quad\mbox{and}\quad 
\bar \Phi_{h, \tau}^n v := E_h(t_n)P_{X_h} v -\bar E_{h, \tau}^n P_{X_h} v , 
$$
and the triangle inequality, we have 
\begin{align}\label{E(s)P_X-E_hjP_{X_h}}
\sum_{j = 1}^{n} \int_{t_{j-1}}^{t_j}\|[ E(s)P_X - \bar E_{h, \tau}^j P_{X_h}] \, v\|_{L^2}^2 \,\d s 
\lesssim\, & \sum_{j = 1}^{n} \int_{t_{j-1}}^{t_j} \| [E_h(s) - E_h(t_j)] P_{X_h}v\|_{L^2}^2 \,  \,  \d s \notag \\
&+ \int_{0}^{t_n} \|\Phi_h(s) v \|_{L^2}^2 \,  \d s 
+  \tau \sum_{j = 1}^{n} \|\bar \Phi_{h, \tau}^j v \|_{L^2}^2. 
\end{align} 
{The first term on the right-hand side of \eqref{E(s)P_X-E_hjP_{X_h}} 
is bounded by using \eqref{Ah:sg-sta-int}, i.e.,}
\begin{align}\label{Ah:sg-sta-sum-l2}
\sum_{j = 1}^n\int_{t_{j-1}}^{t_j}\|[E_h(s) - E_h(t_j)]P_{X_h}v\|_{L^2}^2 \d s 
&= \sum_{j = 1}^n\int_{t_{j-1}}^{t_j}\|(1-e^{-(t_j-s)A_h})e^{-sA_h}P_{X_h}v\|_{L^2}^2 \d s \notag \\
&\lesssim \tau  \int_{0}^{t_n} \| A_h^{\frac12}e^{-sA_h}P_{X_h}v\|_{L^2}^2 \d s \notag \\ 
&\lesssim \tau\|v \|_{L^2}^2 \quad \forall \,\, v\in L^2(D)^d .
\end{align}
{The last two terms 
on the right-hand side of \eqref{E(s)P_X-E_hjP_{X_h}} 
have been estimated in
\eqref{lem:FEM-errs-int2} and \eqref{lem:D-Fully-err-sum}}
respectively,
which imply \eqref{lem:fem-ferrs-L2-int}.

Similarly, \eqref{lem:fem-ferrs-L2} can be proved by using \eqref{lem:FEM-errs-int}, \eqref{lem:D-Fully-err-sum}  and \eqref{Ah:sg-sta-sum-l2}. {This requires using the following Cauchy-Schwarz inequality: 
\begin{align*}
\Big\|\sum_{j = 1}^{n} \int_{t_{j-1}}^{t_j}g(s) \,\d s \Big \|_{L^2}^2 
\lesssim  \Big(\sum_{j = 1}^{n} \int_{t_{j-1}}^{t_j} \|g(s)\|_{L^2} \,\d s \Big)^2 
&\lesssim \Big(\sum_{j = 1}^{n}  \Big[ \tau^{\frac12}  \big(\int_{t_{j-1}}^{t_j} \|g(s)\|_{L^2}^2 \,\d s \big)^\frac12 \Big] \Big)^2 \\
&\lesssim \sum_{j = 1}^{n}  \int_{t_{j-1}}^{t_j} \|g(s)\|_{L^2}^2 \,\d s . 
\end{align*}
Then \eqref{lem:fem-ferrs-L2} can be proved by using the triangle inequality with 
$
\Phi_h(t) = E(t)P_{X} - E_h(t)P_{X_h}$ and 
$\bar \Phi_{h, \tau}^n  = E_h(t_n)P_{X_h}  -\bar E_{h, \tau}^n P_{X_h} 
$, i.e., 
\begin{align*}
&\Big\|\sum_{j = 1}^{n} \int_{t_{j-1}}^{t_j} [E(s)P_X - \bar E_{h, \tau}^j P_{X_h}] \, v \,\d s \Big \|_{L^2}^2  \\
\lesssim\, &  \Big\| \sum_{j = 1}^{n} \int_{t_{j-1}}^{t_j}  [E_h(s) - E_h(t_j)] P_{X_h}v \,  \,  \d s  \Big \|_{L^2}^2  
+ \Big\| \int_{0}^{t_n} \Phi_h(s) v  \,  \d s \Big \|_{L^2}^2
+  \Big\| \tau \sum_{j = 1}^{n} \bar \Phi_{h, \tau}^j v \Big \|_{L^2}^2\\
\lesssim\, &   \sum_{j = 1}^{n} \int_{t_{j-1}}^{t_j} \| [E_h(s) - E_h(t_j)] P_{X_h}v  \|_{L^2}^2  \,  \,  \d s   
+ \Big\| \int_{0}^{t_n} \Phi_h(s) v  \,  \d s \Big \|_{L^2}^2
+  \tau \sum_{j = 1}^{n} \| \bar \Phi_{h, \tau}^j v \|_{L^2}^2,
\end{align*}
where first term on the right-hand side of the above inequality can be estimated by using \eqref{Ah:sg-sta-sum-l2}, while the second and third terms can be estimated by using \eqref{lem:FEM-errs-int} and \eqref{lem:D-Fully-err-sum}, respectively. 
}

Similarly, we have 
\begin{align}\label{E(s)P_X-E_hjP_{X_h}-2}
\Big\|\sum_{j = 1}^{n} \int_{t_{j-1}}^{t_j} \nabla[E(s)P_X - \bar E_{h, \tau}^j P_{X_h}]  \, v \,  \d s \Big \|_{L^2}^2 
\lesssim\, &   \Big\|\sum_{j = 1}^{n} \int_{t_{j-1}}^{t_j} \nabla[E_h(s) - E_h(t_j)] P_{X_h} \, v \,  \d s \Big \|_{L^2}^2 
\notag \\
& + \Big\|\int_{0}^{t_n} \nabla \Phi_h(s)\, v \,  \d s \Big \|_{L^2}^2 
+  \Big\|\tau \sum_{j = 1}^{n} \nabla\bar \Phi_{h, \tau}^j v \Big \|_{L^2}^2 \notag\\
\lesssim\, &(\tau + h^2) \|v\|_{L^2}^2. 
\end{align} 
The first term on the right-hand side of \eqref{E(s)P_X-E_hjP_{X_h}-2} was estimated in \cite[p. 236, with $\rho=1$ and $P_hx$ replaced by $A_h^{\frac12}P_hx$]{kruse2014optimal}, i.e., 
\begin{align}
\Big\|\sum_{j = 1}^n\int_{t_{j-1}}^{t_j}[E_h(s) -  E_h(t_j)](I+A_h)^{\frac{1}{2}} P_{X_h}v \d s\Big\|_{L^2}^2  
\lesssim \tau\| P_{X_h} v \|_{L^2}^2  \lesssim \tau\| v \|_{L^2}^2 . \notag
\end{align}
The second and third terms on the right-hand side of \eqref{E(s)P_X-E_hjP_{X_h}-2} have been estimated in \eqref{lem:FEM-errs-int} and  \eqref{lem:D-Fully-err-sum}, respectively. This proves \eqref{lem:fem-ferrs-H1}. 
\hfill\endproof

\section{Error analysis for the stochastic problem} \label{sec:spde-fully-fem}

In this section, we present error estimates for the fully discrete method \eqref{fully-sFEM-weak} for the stochastic Stokes equations. 
Some estimates of the noise term in the Hilbert--Schmidt norm are presented in Subsection \ref{Section:HS-estimates}, and the error estimates are presented in Section \ref{Section:Proof}.




\subsection{Estimates in the Hilbert--Schmidt norm}\label{Section:HS-estimates}

\begin{lemma}\label{lem:B-noise}
Under Assumptions \ref{ass-W}--\ref{ass-u0}, the operator $B(v): L^2(D)^d\rightarrow L^2(D)^d$ satisfies
\begin{align}\label{lem:B-noise-holder}
\|E(t)P_X[B(u) - B(v)]\|_{ \L_2^0}^2 \lesssim t^{-\frac 12} \,  \|u - v\|_{L^2}^2  \quad \forall\, u, v \in L^2(D)^d,
\end{align}
and
\begin{align}\label{lem:B-noise-sta2}
\|B(u) \|_{ \L_2^0}^2 &\lesssim 1 + \|u\|_{H^{\beta}}^2 &&\forall\, \,u \in H^{\beta}(D)^d, \,\,\, \mbox{{$\beta\in(\frac{d}{2},2)$}},&\\
\label{lem:B-noise-sta}
\|(I+A)^{\frac 12}P_X B(u)\|_{ \L_2^0}^2 &\lesssim 1 +  \|u\|_{H^\beta}^2 &&\forall\,\, u \in H^\beta(D)^d, \,\,\, \mbox{{$\beta\in(\frac{d}{2},2)$}}.&
\end{align}
\end{lemma}
\begin{proof}
{
By using \eqref{IA-eqi-norms-dual} with $s = \frac12$, \eqref{A:sg-sta1} with $\gamma = \frac14$, \eqref{P_X-bound}
and \eqref{ass-con-nosie-1},  we have }
\begin{align}\label{B-noise-holder-est}
&\|E(t)P_X(B(u) - B(v))\|_{ \L_2^0}^2 \notag\\
&=  \sum_\ell \mu_\ell \|E(t)P_X[(B(u) - B(v))\phi_\ell]\|_{L^2}^2 \notag \\
& \lesssim \sum_\ell \mu_\ell \|(I+A)^{\frac 14}E(t) (I+A)^{-\frac 14}P_X[(B(u) - B(v))\phi_\ell]\|_{L^2}^2 
\notag \\
& {\lesssim  t^{-\frac 12} \sum_\ell \mu_\ell  \, \|(B(u) - B(v))\phi_\ell\|^2_{H^{-\frac12}}}
\notag\\
& {\lesssim  t^{-\frac 12}  \|(B(u) - B(v))\|_{{\L_2^0(\mathbb{L}^2, \, \mathbb{H}^{-1/2})}}}
\notag\\ 
& \lesssim t^{-\frac 12}  \, \|u - v\|_{L^2}^2, 
\end{align}
{and by using \eqref{ass-con-nosie-2}, we have}
\begin{align}\label{L20-sta-L2}
\|B(u)\|_{ \L_2^0}^2
& \le \|B(u)\|_{{\L_2^0(\mathbb{L}^2, \, \mathbb{H}^1)}}^2
\lesssim 1+\|u\|_{H^\beta}^2,
\end{align}
which imply \eqref{lem:B-noise-holder}--\eqref{lem:B-noise-sta2}. Similarly, 
\begin{align}\label{L20-sta-H1}
\|(I+A)^{\frac 12}P_X B(u)\|_{ \L_2^0}^2
&= \sum_\ell \mu_\ell \|(I+A)^{\frac 12}P_X B(u)\phi_\ell\|_{L^2}^2 \\
& \le \|B(u)\|_{{\L_2^0(\mathbb{L}^2, \, \mathbb{H}^1)}}^2
\lesssim 1+\|u\|_{H^\beta}^2, \notag
\end{align}
This proves \eqref{lem:B-noise-sta}.
\end{proof}

\begin{remark}\upshape 
As a result of the estimates in Lemma \ref{lem:B-noise}, the regularity results in {Proposition~\ref{thm-u-stability}} imply that  
\begin{align} \label{L20:Sta-B}
&\sup_{t \in [0,T]}  \E\|B(u(t))\|_{ \L_2^0}^2 + \sup_{t \in [0,T]}  \E\|(I+A)^{\frac {1}{2}}P_X B(u(t))\|_{ \L_2^0}^2
\lesssim 1. 
\end{align}
\end{remark}

The following stability estimates for the numerical solution can be proved by using Lemma \ref{lem:B-noise} and will be used in the error analysis. The proof is omitted here, and the details can be found in Appendix C.  
\begin{lemma}\label{Lemma:uh-energy}
Under Assumptions \ref{ass-W}--\ref{ass-u0}, the numerical solution determined by the fully discrete method \eqref{fully-sFEM} satisfies the following energy inequality: 
\begin{align}\label{sFEM-uhn-sta}
\max_{1 \le n \le N}  \E \|u_h^n\|_{H^\frac 12}^2 
+ \sum_{n = 1}^N \E \|u_h^n - u_h^{n-1}\|_{L^2}^2 
+ \tau\sum_{n = 1}^N \E \|u_h^n\|_{H^1}^2
\lesssim 1. 
\end{align}   
\end{lemma}

\subsection{Error estimates for the velocity}
\label{Section:Proof}


By iterating \eqref{fully-sFEM} with respect to $n$, the full discrete method can be rewritten as
\begin{equation}\label{fully-sFEM-sum}
u_h^n
=  \bar  E_{h,\tau}^n P_{X_h}u_0 +  \tau\sum_{i = 0}^{n-1} \bar E_{h, \tau}^{n-i} P_{X_h}f(t_{i + 1})
+ \sum_{i = 0}^{n-1} \bar E_{h, \tau}^{n-i} P_{X_h} [B(u_h^i) \Delta W_{i+1}]. 
\end{equation} 
Then, after subtracting \eqref{fully-sFEM-sum} from \eqref{SPDE-mild}, we obtain the following error equation:
\begin{align}\label{error-expr}
u(t_n) - u_h^n = \,&(E(t_n)P_X - \bar E_{h, \tau}^n P_{X_h}) u_0 \notag \\
&+  \sum_{i = 0}^{n-1} \int_{t_i}^{t_{i+1}} [E(t_n -s)P_X f(s) - \bar E_{h, \tau}^{n-i} P_{X_h}f(t_{i+1})]\d s\notag\\
&+  \sum_{i = 0}^{n-1} \int_{t_i}^{t_{i+1}}[E(t_n -s)P_X B(u(s)) - \bar E_{h, \tau}^{n-i} P_{X_h}B(u_h^i) ] \d W(s)
\notag\\
=&\!:T_1 + T_2 +T_3, 
\end{align}
which implies that 
\begin{align}\label{sfem-ferrs:u-inq}
\E \|u(t_n) - u_h^n\|_{L^2}^2
\lesssim \sum_{j = 1}^3  \E \|T_j\|_{L^2}^2.
\end{align}
The three terms are estimated below separately below. 

Since $T_1=\Phi_h(t_n) u^0+\bar \Phi_{h, \tau}^n u^0$, by applying \eqref{lem:FEM-errs-In-u0} and \eqref{lem:D-Fully-err-h01} with $\rho = 1$  and $v =   u_0 \in \DD(A)$, we obtain 
\begin{align}\label{sfem-ferrs:u-T1} 
\E \|T_1\|_{L^2}^2
\lesssim \E \|\Phi_h(t_n) u^0\|_{L^2}^2 + \E \|\bar \Phi_{h, \tau}^n u^0\|_{L^2}^2 
\lesssim (\tau + h^4)\E \|u^0\|_{H^2}^2. 
\end{align}

By using the triangle inequality, $T_2$ can be further decomposed into the following two parts:  
\begin{align}\label{sfem-ferrs:u-T2}
\E \|T_2\|_{L^2}^2
\lesssim\, & \E \Big\|\sum_{i = 0}^{n-1} \int_{t_i}^{t_{i+1}} E(t_n -s)P_X [f(s) - f(t_{i+1})]\d s \Big \|_{L^2}^2\\
&+ \E \Big\|\sum_{i = 0}^{n-1} \int_{t_i}^{t_{i+1}} [E(t_n -s)P_X - \bar E_{h, \tau}^{n-i} P_{X_h}] f(t_{i+1})\d s \Big \|_{L^2}^2\notag\\
=&\!: T_{21} + T_{22}.\notag
\end{align}
By using the H\"older continuity of $f(t)$ in \eqref{ass-con-f-lip}, we have 
\begin{align}\label{sfem-ferrs:u-T21}
T_{21}  
\lesssim \sum_{i = 0}^{n-1} \int_{t_i}^{t_{i+1}}\E \| f(s) - f(t_{i+1}) \|_{L^2}^2\d s 
\lesssim \sum_{i = 0}^{n-1} \int_{t_i}^{t_{i+1}}(t_{i+1}-s) \d s 
\lesssim \tau. 
\end{align}
Through a change of variables $\sigma = t_n - s$ and $j = n - i $, we obtain by using the triangle inequality 
\begin{align}\label{sfem-ferrs:u-T22}
T_{22} =\,&\E \Big\|\sum_{j = 1}^{n} \int_{t_{j-1}}^{t_j} [E(\sigma)P_X - \bar E_{h, \tau}^j P_{X_h}] f(t_{n-j+1})\d \sigma \Big \|_{L^2}^2 \notag\\
\lesssim\, & \E \Big\|\sum_{j = 1}^{n} \int_{t_{j-1}}^{t_j} [E(\sigma) - E(t_j)] P_X[f(t_{n-j+1}) - f(t_{n})]\d \sigma \Big \|_{L^2}^2 \notag\\
&+\E \Big\|\sum_{j = 1}^{n} \int_{t_{j-1}}^{t_j} [E(t_j)P_X - \bar E_{h, \tau}^j P_{X_h}] [f(t_{n-j+1}) - f(t_{n})]\d \sigma \Big \|_{L^2}^2 \notag\\
&+\E \Big\|\sum_{j = 1}^{n} \int_{t_{j-1}}^{t_j} [E(\sigma)P_X - \bar E_{h, \tau}^j P_{X_h}]  f(t_{n})\d \sigma \Big \|_{L^2}^2 \notag\\
=&\!:T_{22}^a  + T_{22}^b + T_{22}^c. 
\end{align} 
The term $T_{22}^a$ can be estimated by using \eqref{A:sg-sta1} with $\gamma = \frac12$, \eqref{A:sg-sta2} with $\mu = \frac12$, and the H\"older continuity of $f$ in \eqref{ass-con-f-lip}, i.e., 
\begin{align}\label{sfem-ferrs:u-T22a}
T_{22}^a
\lesssim&\, \sum_{j = 1}^{n} \int_{t_{j-1}}^{t_j} \| (I+A)^{\frac12} E(\sigma) (I+A)^{-\frac12} [I- E(t_j - \sigma)] P_X[f(t_{n-j+1}) - f(t_{n})]  \|_{L^2}^2 \d \sigma \notag\\
\lesssim&\, \sum_{j = 1}^{n} \int_{t_{j-1}}^{t_j}\sigma^{-1} (t_j - \sigma)
\|f(t_{n-j+1}) - f(t_{n}) \|_{L^2}^2\d \sigma\\
\lesssim&\, \tau\sum_{j = 1}^{n} \int_{t_{j-1}}^{t_j}\sigma^{-1} 
(t_{n} - t_{n-j+1} ) \d \sigma 
\lesssim  \tau\sum_{j = 1}^{n} \int_{t_{j-1}}^{t_j}\sigma^{-1} \sigma \d \sigma 
\lesssim  \tau .\notag
\end{align}
{where the conditions $t_j - \sigma \le \tau$ and $t_n - t_{n-j+1} = t_{j - 1} \leq \sigma$ are satisfied for $\sigma \in [t_{j-1}, t_j]$, ensuring the integrability in the last line.}

The term $T_{22}^b$ can be estimated by using \eqref{lem:FEM-errs-TH-l2}, \eqref{lem:D-Fully-err-l2} and the Assumption \ref{ass-con-f-lip} shows
\begin{align}\label{sfem-ferrs:u-T22b}
T_{22}^b
\lesssim &\, \E\Big(\sum_{j = 1}^{n} \int_{t_{j-1}}^{t_j}  \|[\Phi_h(t_j) +\bar \Phi_{h, \tau}^j] [f(t_{n-j+1}) - f(t_{n})]  \|_{L^2}\d \sigma \Big)^2\\
\lesssim &\, \E\Big(\sum_{j = 1}^{n} \int_{t_{j-1}}^{t_j} (t_j^{-1}h^2 + t_j^{-\frac12}\tau^{\frac12}) \|f(t_{n-j+1}) - f(t_{n})\|_{L^2}\d \sigma\Big)^2\notag\\
\lesssim &\, (\tau + h^4)\Big(\sum_{j = 1}^{n} \int_{t_{j-1}}^{t_j} t_j^{-1}(t_{n} - t_{n-j+1} )^{\frac12} \d \sigma\Big)^2\notag\\
\lesssim &\, {(\tau + h^4)\Big(\sum_{j = 1}^{n} \int_{t_{j-1}}^{t_j} \sigma^{-1} \sigma^{\frac12} \d \sigma\Big)^2}
\lesssim  \tau + h^4.\notag
\end{align}
The term $T_{22}^c$ can be estimated by applying \eqref{lem:fem-ferrs-L2} directly, i.e., 
\begin{align}\label{sfem-ferrs:u-T22c}
T_{22}^c \lesssim (\tau + h^4) \|f(t_n)\|_{L^2}^2 \lesssim \tau + h^4 .
\end{align}
Substituting \eqref{sfem-ferrs:u-T21}--\eqref{sfem-ferrs:u-T22c} into \eqref{sfem-ferrs:u-T2} yields 
\begin{align}\label{sfem-ferrs:u-T2-sum}
\E \|T_2\|_{L^2}^2 \lesssim \tau + h^4 .
\end{align}

It remains to estimate the term $T_3$ in \eqref{error-expr}.
By using \eqref{Ito_isometry} and a change of variables $\sigma=t_n - s$ and $j = n - i$, we can decompose $T_3$ into three parts as follows:  
\begin{align}\label{sfem-ferrs:u-T3}
\E \|T_3\|_{L^2}^2 
=\,& \E \sum_{i = 0}^{n-1} \int_{t_i}^{t_{i+1}}\|E(t_n -s)P_X B(u(s)) - \bar E_{h, \tau}^{n-i} P_{X_h}B(u_h^i) \|_{ \L_2^0}^2 \d s
\notag \\
= \, &\E \sum_{j = 1}^n \int_{t_{j - 1}}^{t_j}\|[E(\sigma)P_X B(u(t_n - \sigma)) - \bar E_{h, \tau}^j P_{X_h}B(u_h^{n - j})  \|_{ \L_2^0}^2 \d \sigma\notag\\
\lesssim\, &\E \sum_{j = 1}^n \int_{t_{j - 1}}^{t_j}\|E(\sigma)P_X [B(u(t_n - \sigma)) - B(u(t_{n-j}))]\|_{ \L_2^0}^2  \d \sigma \notag\\
&+ \E \sum_{j = 1}^n \int_{t_{j - 1}}^{t_j}\|[E(\sigma)P_X - \bar E_{h, \tau}^j P_{X_h}] B(u(t_{n-j})) \|_{ \L_2^0}^2 \d \sigma\notag \\
&+\tau \E \sum_{j = 1}^n \|\bar E_{h, \tau}^jP_{X_h}[B(u(t_{n-j})) - B(u_h^{n-j})]\|_{ \L_2^0}^2\notag\\
=&\!: T_{31} + T_{32} + T_{33}.
\end{align}
The three terms $T_{31}$, $T_{32}$ and $T_{33}$ are estimated separately below. 

The term $T_{31}$ can be estimated by using \eqref{lem:B-noise-holder} and H\"older continuity \eqref{SPDE-u-holder}, which imply that 
\begin{align}\label{sfem-ferrs:u-T31}
T_{31}
\le\, & \sum_{j = 1}^n \int_{t_{j - 1}}^{t_j}\sigma^{-\frac12}\E\|u(t_n - \sigma)- u(t_{n-j})\|_{L^2}^2\d \sigma\\
\le\, &C\big(1 + \E\|u^0\|^2_{H^2}\big)\E \sum_{j = 1}^n \int_{t_{j - 1}}^{t_j}\sigma^{-\frac12}(t_j - \sigma)\d \sigma \notag\\
\le\, &C \tau \big(1 + \E\|u^0\|^2_{ H^2}\big). \notag
\end{align}
The term $T_{32}$ can be estimated by decomposing it into three parts, i.e.,
\begin{align}\label{sfem-ferrs:u-T32}
T_{32} 
\lesssim &\,   \E \sum_{j = 1}^n \int_{t_{j - 1}}^{t_j}\|[E(\sigma) - E(t_j)] P_X[B(u(t_{n-j})) - B(u(t_{n-1}))] \|_{ \L_2^0}^2\d \sigma \\
&\, + \E \sum_{j = 1}^n \int_{t_{j - 1}}^{t_j}\| [E(t_j)P_X - \bar E_{h, \tau}^j P_{X_h}] [B(u(t_{n-j})) - B(u(t_{n-1}))]\|_{ \L_2^0}^2 \d \sigma \notag\\
&\,  + \E \sum_{j = 1}^n \int_{t_{j - 1}}^{t_j}\| [E(\sigma)P_X - \bar E_{h, \tau}^j P_{X_h}]  B(u(t_{n-1})) \|_{ \L_2^0}^2 \d \sigma\notag\\
= &\!:T_{32}^a  + T_{32}^b + T_{32}^c. \notag
\end{align} 
By using  \eqref{ass-con-nosie-1}, \eqref{A:sg-sta1} with $\gamma = \frac34$, \eqref{A:sg-sta2} with $\mu = \frac12$
and H\"older continuity \eqref{SPDE-u-holder}, we have the following bound 
for $T_{32}^a$ 
\begin{align}\label{sfem-ferrs:u-T32a}
T_{32}^a
\lesssim &\, \sum_{j = 1}^{n} \int_{t_{j-1}}^{t_j}\E \| E(\sigma) [I- E(t_j - \sigma)] P_X[B(u(t_{n-j})) - B(u(t_{n-1}))]  \|_{ \L_2^0}^2 \d \sigma \notag\\
\lesssim &\, \sum_{j = 1}^{n} \int_{t_{j-1}}^{t_j}\E \sum_{\ell}\mu_\ell\| {(I+A)^{\frac34}}E(\sigma)  {(I+A)^{-\frac34}} [I- E(t_j - \sigma)] P_X[B(u(t_{n-j})) - B(u(t_{n-1}))] \phi_\ell \|_{L^2}^2\d \sigma \notag\\
\lesssim &\, {\sum_{j = 1}^{n} \int_{t_{j-1}}^{t_j} \sigma^{-\frac32} \E \sum_{\ell}\mu_\ell\|  (I+A)^{-\frac12} [I- E(t_j - \sigma)] (I+A)^{-\frac14} P_X[B(u(t_{n-j})) - B(u(t_{n-1}))] \phi_\ell \|_{L^2}^2\d \sigma }\notag\\
\lesssim &\, \sum_{j = 1}^{n} \int_{t_{j-1}}^{t_j}{\sigma^{-\frac32}} (t_j - \sigma) \E \sum_{\ell}\mu_\ell \|B(u(t_{n-j})) - B(u(t_{n-1}))\phi_\ell \|_{{H^{-\frac12}}}^2\d \sigma \notag\\
\lesssim &\, \sum_{j = 1}^{n} \int_{t_{j-1}}^{t_j}{\sigma^{-\frac32}} (t_j - \sigma)
\E \|u(t_{n-j}) - u(t_{n-1})\|_{L^2}^2\d \sigma 
\notag\\
\lesssim &\, \tau \big(1 + \E\|u^0\|^2_{ H^2}\big)\sum_{j = 1}^{n} \int_{t_{j-1}}^{t_j}{\sigma^{-\frac32}}
(t_{n-1} - t_{n-j} )\d \sigma \notag\\
\lesssim &\, {\tau \big(1 + \E\|u^0\|^2_{ H^2}\big)\sum_{j = 1}^{n} \int_{t_{j-1}}^{t_j} \sigma^{-\frac32}
\sigma \d \sigma} \notag\\
\lesssim &\, \tau \big(1 + \E\|u^0\|^2_{ H^2}\big).
\end{align}
where \eqref{ass-con-nosie-1} is used in the derivation of the fourth to last inequality, inequalities $t_j - \sigma \le \tau$ and \eqref{SPDE-u-holder} are used in deriving the third to last inequality, and estimate $t_{n-1} - t_{n-j} = t_{j - 1} \leq \sigma$ for $\sigma \in [t_{j-1}, t_j]$ is used in deriving the second to last inequality.

By noting $E(t_j)P_X - \bar E_{h, \tau}^j P_{X_h}=\Phi_h(t_j)+\bar \Phi_{h, \tau}^j$,
$T_{32}^b$ is bounded by 
\begin{align}\label{sfem-ferrs:u-T32b}
T_{32}^b
= &\, \sum_{j = 1}^{n} \int_{t_{j-1}}^{t_j} \E \|[\Phi_h(t_j)+\bar \Phi_{h, \tau}^j][B(u(t_{n-j})) - B(u(t_{n-1}))]  \|_{ \L_2^0}^2 \d \sigma\\
\lesssim &\, \sum_{j = 1}^{n} \int_{t_{j-1}}^{t_j} t_j^{-1}h^4 \,\E \|(I + A)^{\frac{1}{2}}P_X[B(u(t_{n-j})) - B(u(t_{n-1}))]\|_{ \L_2^0}^2 \d \sigma 
\notag\\
&+ \sum_{j = 1}^{n} \int_{t_{j-1}}^{t_j} t_j^{-1}\tau\,\E \|B(u(t_{n-j})) - B(u(t_{n-1}))\|_{ \L_2^0}^2  \d \sigma
\notag\\
\lesssim &\, (\tau+h^4)\sum_{j = 1}^{n} \int_{t_{j-1}}^{t_j}  t_j^{-1} \E \|u(t_{n-j}) - u(t_{n-1})\|_{H^\beta}^2 \d \sigma 
\notag\\
\lesssim &\,(\tau+h^4)\big(1 + \E\|u^0\|^2_{ H^2}\big)\sum_{j = 1}^{n} \int_{t_{j-1}}^{t_j} t_j^{-1}(t_{n-1} - t_{n-j} )^{2-\beta} \d \sigma
\notag\\
\lesssim &\,(\tau + h^4)\big(1 + \E\|u^0\|^2_{ H^2}\big)\int_{0}^{t_n}  \sigma^{1-\beta} \d \sigma\notag\\
\lesssim &\, (\tau + h^4)\big(1 + \E\|u^0\|^2_{ H^2}\big).\notag
\end{align}
{where \eqref{lem:FEM-errs-TH} and \eqref{lem:D-Fully-err-l2} are used in the derivation of the fifth to last inequality,  \eqref{ass-con-nosie-2} is used in the fourth to last inequality, \eqref{SPDE-u-holder-Hbeta} is used in the third to last inequality, and $t_{n-1} - t_{n-j} = t_{j - 1} \leq \sigma$ for $\sigma \in [t_{j-1}, t_j]$ is used in the second to last inequality.} 
Directly applying \eqref{lem:fem-ferrs-L2-int} and \eqref{L20:Sta-B} shows that 
%
\begin{align}\label{sfem-ferrs:u-T32c}
T_{32}^c 
\lesssim
(\tau + h^4)\sup_{t \in [0, T]}\E \|(I+A)^{ \frac{1}{2}} P_X B(u(t)) \|_{ \L_2^0}^2 
\lesssim \tau +  h^4 .
\end{align}
Moreover,  by \eqref{E_ht:sg-sta} with $\gamma = \frac12$ {and \eqref{ass-con-nosie-1}}, we further have  
\begin{align}\label{sfem-ferrs:u-T33}
\E \|T_{33}\|_{L^2}^2
\lesssim&\, \tau  \sum_{j = 1}^n \E\|(I+A_h)^{\frac14}\bar E_{h, \tau}^j (I+A_h)^{-\frac14}P_{X_h}[B(u(t_{n-j})) - B(u_h^{n-j})]\|_{ \L_2^0}^2 \notag \\
\lesssim&\, \tau \sum_{j = 1}^n  t_j^{-\frac12}\E\sum_\ell \mu_\ell  \, \|P_{X_h}[(B(u(t_{n-j})) - B(u_h^{n-j}) )\phi_{\ell}]\|^2_{H^{-\frac12}}
\notag\\
\lesssim &\, \tau \sum_{j = 1}^n  t_j^{-\frac12}\E  \|B(u(t_{n-j})) - B(u_h^{n-j}) \|_{{\L_2^0(\mathbb{L}^2, \mathbb{H}^{-1/2})}}^2
\notag\\
\lesssim&\, {\tau \sum_{j = 1}^n  t_j^{-\frac12}\E\|u(t_{n-j}) - u_h^{n-j}\|_{L^2}^2 }\notag\\
\lesssim&\, \tau \sum_{i = 0}^{n-1}  (t_n - t_i)^{-\frac12}\E\|u(t_i) - u_h^i\|_{L^2}^2. 
\end{align}
In view of \eqref{sfem-ferrs:u-T3}--\eqref{sfem-ferrs:u-T33}, we have
\begin{align}\label{sfem-ferrs:u-T3-sum}
\E \|T_{3}\|_{L^2}^2 \lesssim \tau + h^4  + \tau \sum_{i = 0}^{n-1}  (t_n - t_i)^{-\frac12}\E\|u(t_i) - u_h^i\|_{L^2}^2.
\end{align}
Substituting \eqref{sfem-ferrs:u-T1}, \eqref{sfem-ferrs:u-T2-sum} and \eqref{sfem-ferrs:u-T3-sum} into \eqref{sfem-ferrs:u-inq} yields 
\begin{align}
\E \|u(t_n) - u_h^n\|_{L^2}^2 \lesssim \tau + h^4 + \tau \sum_{i = 0}^{n-1}  (t_n - t_i)^{-\frac12}\E\|u(t_i) - u_h^i\|_{L^2}^2.
\end{align}
By using the {discrete version of generalized Gronwall's inequality in \cite[Lemma A.4]{kruse2014strong}}, we obtain the desired error estimate in \eqref{sfem-ferrs:u} for the velocity.

\subsection{Error estimates for the pressure.}  
From the numerical scheme in \eqref{fully-sFEM-weak} we can derive that 
\begin{align}\label{fully-sFEM-weak-sum}
\tau\sum_{n = 1}^m(p_h^n,\nabla\cdot v_h) 
= &\, (u_h^m,v_h) - (u_h^0,v_h) 
+2\tau\sum_{n = 1}^m\big(\D(u_h^n),\D(v_h)\big)
\notag\\
&\, -\tau\sum_{n = 1}^m (f(t_n) ,v_h) - \tau\sum_{n = 1}^m(B(u_h^{n-1})\Delta W_n ,v_h) \qquad v_h \in V_h. 
\end{align}
Subtracting \eqref{fully-sFEM-weak-sum} from \eqref{spde-weak} yields 
\begin{align*} 
\Big(\int_0^{t_m}p(t)\,\d t - \tau\sum_{n = 1}^m  p_h^n,\nabla\cdot v_h\Big) 
=\, &\big([u(t_m) - u_h^m] - [u^0 - u_h^0],v_h\big)\\
&+  2\Big(\sum_{n = 1}^m \int_{t_{n-1}}^{t_n} \D( u(t) -u_h^n)\,\d t, \D( v_h)\Big) \notag \\
&- \Big(\sum_{n = 1}^m \int_{t_{n-1}}^{t_n} [f(t) - f(t_n)]\,\d t,v_h\Big)
\notag\\
&- \Big(\sum_{n = 1}^m \int_{t_{n-1}}^{t_n} [B(u(t)) - B(u_h^{n-1})]\,\d W(t),v_h\Big)
\notag\\
=&\!: J_1(v_h) + J_2(v_h) + J_3(v_h) + J_4(v_h) \qquad \forall\, v_h \in V_h , 
\end{align*}
where
\begin{align*} 
|J_1(v_h)|
\lesssim &\, \big(  \|u(t_m) - u_h^m\|_{L^2} +\|u^0 - P_{X_h}u^0\|_{L^2} \big) 
\|v_h\|_{L^2} 
=: J_1^*\|v_h\|_{L^2} , \\
|J_2(v_h)|
\lesssim &\,
\bigg\|\sum_{n = 1}^m  \int_{t_{n-1}}^{t_n}  \D\big( u(t) -u_h^n\big) \,\d t \bigg\|_{L^2} \|\nabla v_h\|_{L^2} 
=: J_2^* \|\nabla v_h\|_{L^2} \\
|J_3(v_h)|
\lesssim &\, 
{\sum_{n = 1}^m  \int_{t_{n-1}}^{t_n}}
\|
f(t) - f(t_n)\|_{L^2} \,\d t \|v_h\|_{L^2} 
=: J_3^* \|v_h\|_{L^2} \\
|J_4(v_h)|
\lesssim &\, 
\Big\|\sum_{n = 1}^m  \int_{t_{n-1}}^{t_n} [B(u(t)) - B(u_h^{n-1})]\,\d W(t) \Big\|_{H^{-1}} \|v_h\|_{H^1} 
=: J_4^* \|v_h\|_{H^1} , 
\end{align*} 
which imply that (by using the inf-sup condition \eqref{inf-sup} of the finite element space) 
\begin{align}\label{Error-J1234}
\E\bigg\|\int_0^{t_m}p(t)\,\d t - \tau\sum_{n = 1}^m  p_h^n\bigg\|_{L^2}^2
\lesssim \E |J_1^*|^2+ \E |J_{2}^*|^2+ \E |J_{3}^*|^2+ \E |J_{4}^*|^2 . 
\end{align}

By using \eqref{sfem-ferrs:u} and \eqref{H2-app-PXh1}, we have 
\begin{align}\label{J1-star}
\E |J_1^*|^2
\lesssim &\, ( \tau + h^4  )  .
\end{align}
and
\begin{align}\label{J2-1234-star}
J_2^* 
\le &\, \bigg\| \sum_{n = 1}^m  \int_{t_{n-1}}^{t_n}  \D\big( u(t) -u(t_n)\big) \,\d t  \bigg\|_{L^2}
+ \bigg\| \sum_{n = 1}^m  \int_{t_{n-1}}^{t_n}  \D\big( u(t_n) - u_h^n\big) \,\d t  \bigg\|_{L^2}    \\
\le &\, \bigg\| \sum_{n = 1}^m  \int_{t_{n-1}}^{t_n}  \D\big( u(t) -u(t_n)\big) \,\d t  \bigg\|_{L^2} \quad\mbox{(using \eqref{SPDE-mild} with $t = t_n$ and \eqref{fully-sFEM-sum})} \notag\\
&\,+ \bigg\| \sum_{n = 1}^m \int_{t_{n-1}}^{t_n}  \D\big( [E(t_n)P_X - \bar E_{h, \tau}^nP_{X_h}]u^0\big) \,\d t \bigg\|_{L^2}  \notag \\
&\,+ \bigg\| \sum_{n = 1}^m  \int_{t_{n-1}}^{t_n} \sum_{i = 0}^{n-1}  \int_{t_{i}}^{t_{i+1}}  \D\big( E(t_n - s)P_{X}f(s) - \bar E_{h, \tau}^{n-i} P_{X_h}f(t_{i+1})\big)\,\d s \,\d t \bigg\|_{L^2}  \notag\\
&\,+ \bigg\| \sum_{n = 1}^m  \int_{t_{n-1}}^{t_n} \sum_{i = 0}^{n-1}  \D\Big(\int_{t_{i}}^{t_{i+1}}E(t_n - s)P_{X} B(u(s)) - \bar E_{h, \tau}^{n-i} P_{X_h} B(u_h^i)\,\d W(s)\Big) \,\d t \bigg\|_{L^2} \notag\\
=&\!: J_{21}^*+ J_{22}^* +  J_{23}^* + J_{24}^*, \notag
\end{align}
where 
\begin{align}\label{sFEM-err-p:I21}
\E |J_{21}^*|^2 
\lesssim &\, \sum_{n = 1}^m  \int_{t_{n-1}}^{t_n}
\E \|u(t) -u(t_n)\|_{H^1}^2 \,\d t  \notag\\
\lesssim &\, \sum_{n = 1}^m  \int_{t_{n-1}}^{t_n}  \big(1+ \E\|u^0\|_{H^2}^2\big)(t - t_n) \, \d t 
\quad\mbox{(here \eqref{SPDE-u-holder} is used)} \notag\\
\lesssim &\,\tau , \notag \\
\E |J_{22}^*|^2 
\lesssim &\, \E \Big\|\sum_{n = 1}^m \int_{t_{n-1}}^{t_n} \nabla [E(t)P_X - \bar E_{h, \tau}^nP_{X_h}]u^0 \,\d t \Big\|_{L^2}^2 
+ \E \Big\|\sum_{n = 1}^m \int_{t_{n-1}}^{t_n} \nabla [E(t_n)- E(t) ]P_X u^0 \,\d t \Big\|_{L^2}^2 \notag\\
\lesssim &\, \tau + h^2  \qquad\mbox{(here \eqref{lem:fem-ferrs-H1} is used)} \notag\\
&\, + {\E \sum_{n = 1}^m \int_{t_{n-1}}^{t_n} \| (I+A)^{-\frac12}[E(t_n - t) - I] E(t) (I+A) u^0  \|_{L^2}^2\,\d t} \notag\\
\lesssim &\, \tau + h^2 
+ \E \int_{0}^{t_m} \tau \|E(t)  (I+A) u^0 \|_{L^2}^2\,\d t \quad\mbox{(here \eqref{A:sg-sta2} is used)}\notag\\
\lesssim &\, \tau + h^2 + \E \int_{0}^{t_m} \tau \|  u^0 \|_{{H^2}}^2\,\d t  \notag\\
\lesssim &\, \tau + h^2 , \notag \\ 
\E |J_{23}^*|^2 
\lesssim &\, \E \Big\|\sum_{n = 1}^m  \int_{t_{n-1}}^{t_n} \sum_{i = 0}^{n-1}  \int_{t_{i}}^{t_{i+1}} \nabla E(t_n - s)P_{X}[f(s) -f(t_{i+1})]\,\d s \,\d t\Big\|_{L^2}^2\\
&+ \E \Big\|\sum_{n = 1}^m  \int_{t_{n-1}}^{t_n} \sum_{i = 0}^{n-1}  \int_{t_{i}}^{t_{i+1}} \nabla[E(t_n - s)P_{X} - \bar E_{h, \tau}^{n-i} P_{X_h}]f(t_{i+1})\,\d s \,\d t\Big\|_{L^2}^2\notag\\
=&\!: \E |J_{231}^*|^2 + \E |J_{232}^*|^2 .
\notag
\end{align}
The two terms $\E |J_{231}^*|^2 $ and $ \E |J_{232}^*|^2$ can be estimated as follows:  
\begin{align}
\E |J_{231}^*|^2 
\lesssim &\,  \Big(\sum_{n = 1}^m  \int_{t_{n-1}}^{t_n}  \sum_{i = 0}^{n-1}  \int_{t_{i}}^{t_{i+1}}   \|(I+A)^{\frac12}E(t_n - s) P_{X}[f(s) -f(t_{i+1})]\|_{L^2}\,\d s \,\d t\Big)^2 \\
\lesssim &\,  \Big( \sum_{n = 1}^m  \int_{t_{n-1}}^{t_n}  \sum_{i = 0}^{n-1}  \int_{t_{i}}^{t_{i+1}} (t_n - s)^{-\frac12}\E \|f(t_{i+1})-f(s) \|_{L^2}\,\d s \,\d t\Big)^2  \notag\\
\lesssim &\,  \Big(\int_{0}^{t_m} \int_{0}^{t_n}  (t_n - s)^{-\frac12} \tau^{\frac12} \d s  \,\d t\Big)^2 \notag\\
\lesssim &\, \tau 
\notag
\end{align}
and, by using a change of variables $\sigma= t_n - s $ and $j=n - i $, 
\begin{align}
\E |J_{232}^*|^2 
= \,&  \E \Big\|\tau\sum_{i = 0}^{m-1}   \sum_{n = i+1}^m  \int_{t_i}^{t_{i+1}} \nabla[E(t_n - s)P_{X} - \bar E_{h, \tau}^{n-i} P_{X_h}]f(t_{i+1})\,\d s \Big\|_{L^2}^2\\
= \,&  \E \Big\|\tau\sum_{i = 0}^{m-1}   \sum_{j = 1}^{m-i}  \int_{t_{j-1}}^{t_j} \nabla[E(\sigma)P_{X} - \bar E_{h, \tau}^{j} P_{X_h}]f(t_{i+1})\,\d \sigma \Big\|_{L^2}^2\notag\\
\lesssim &\, \tau\sum_{i = 0}^{m-1}  \E \Big\|\sum_{j = 1}^{m-i}  \int_{t_{j-1}}^{t_j} \nabla[E(\sigma)P_{X} - \bar E_{h, \tau}^{j} P_{X_h}]f(t_{i+1})\,\d \sigma \Big\|_{L^2}^2\notag\\
\lesssim &\, \tau\sum_{i = 0}^{m-1}(\tau + h^2)\E\|f(t_{i+1})\|_{L^2}^2 \quad\mbox{(here \eqref{lem:fem-ferrs-H1} is used)} \notag\\
\lesssim &\, \tau + h^2 .\notag
\end{align}

It remains to estimate the term $J_{24}^*$ in \eqref{J2-1234-star}. 
Since the integrand in the expression of $J_{24}^*$ is independent of $t$ for $t\in(t_{n-1},t_n]$, it follows that the integration with respect to $t$ is actually a summation times $\tau$.  
By interchanging the order of integration in the expression of $J_{24}^*$ and using It\^o's isometry in \eqref{Ito_isometry}, we obtain 
\begin{align}
\E |J_{24}^*|^2 
=\,&\E \bigg\| \D\Big( \sum_{n = 1}^m  \int_{t_{n-1}}^{t_n} \sum_{i = 0}^{n-1} \int_{t_{i}}^{t_{i+1}}E(t_n - s)P_{X} B(u(s)) - \bar E_{h, \tau}^{n-i} P_{X_h} B(u_h^i)\,\d W(s) \,\d t \Big)\bigg\|_{L^2}^2 \notag\\
=\, &\E \bigg\|\D\Big( \sum_{i = 0}^{m-1} \int_{t_i}^{t_{i+1}}\sum_{n = i+1}^m \int_{t_{n-1}}^{t_n} [E(t_n - s)P_{X} B(u(s)) - \bar E_{h, \tau}^{n-i} P_{X_h} B(u_h^i)]\, \d t \,\d W(s) \Big) \bigg\|_{L^2}^2\notag \\
\lesssim\,&\E\sum_{i = 0}^{m-1} \int_{t_i}^{t_{i+1}}\Big\| \sum_{n = i+1}^m \int_{t_{n-1}}^{t_n} \nabla[E(t_n - s)P_{X} B(u(s)) - \bar E_{h, \tau}^{n-i} P_{X_h} B(u_h^i)] \, \d t\Big\|_{ \L_2^0}^2\,\d s. \notag 
\end{align}
%
By using a change of variables $\sigma = t - t_i$ and $j = n - i$, we split $\E |J_{24}^*|^2$ into three parts as follows, by using the triangle inequality, 
\begin{align}
\E |J_{24}^*|^2 
=\,&\E\sum_{i = 0}^{m-1} \int_{t_i}^{t_{i+1}}\Big\| \sum_{j = 1}^{m-i} \int_{t_{j-1}}^{t_j}  \nabla[E(t_{i+j} - s)P_{X} B(u(s)) - \bar E_{h, \tau}^{j} P_{X_h} B(u_h^i)]\, \d \sigma\Big\|_{ \L_2^0}^2\,\d s \notag\\
\lesssim\, &\E\sum_{i = 0}^{m-1} \int_{t_i}^{t_{i+1}}\Big\| \sum_{j = 1}^{m-i} \int_{t_{j-1}}^{t_j}  \nabla[E(t_{i+j} - s) - E(\sigma)]P_{X} B(u(s)) \, \d \sigma\Big\|_{ \L_2^0}^2\,\d s \notag\\
&+\E\sum_{i = 0}^{m-1} \int_{t_i}^{t_{i+1}}\Big\| \sum_{j = 1}^{m-i} \int_{t_{j-1}}^{t_j}  \nabla E(\sigma)P_{X} [B(u(s)) - B(u_h^i)]\, \d \sigma\Big\|_{ \L_2^0}^2\,\d s \notag\\
&+\tau \E\sum_{i = 0}^{m-1} \Big\| \sum_{j = 1}^{m-i} \int_{t_{j-1}}^{t_j}  \nabla[E(\sigma)P_{X}  - \bar E_{h, \tau}^{j} P_{X_h}] B(u_h^i)\, \d \sigma\Big\|_{ \L_2^0}^2 \notag\\
=&\!:\E |J_{241}^*| + \E |J_{242}^*| + \E |J_{243}^*|.\notag
\end{align}
In order to estimate the term $\E |J_{241}^*|$,  we first note that for $\xi_j,\sigma\in (t_{j-1},t_j)$ we have $|\xi_j - \sigma| < \tau$ and therefore, by using the triangle inequality, 
\begin{align}
&\sum_{ j = 1}^{m} \int_{t_{j-1}}^{t_j} \| [E(\xi_j)-E(\sigma)] \widetilde{v(s)} \|_{ \L_2^0}^2\, \d \sigma  \notag\\
=\, & \int_{0}^{2\tau} \| [E(\xi_j)-E(\sigma)] \widetilde{v(s)} \|_{ \L_2^0}^2\, \d \sigma 
+ \sum_{ j = 3}^{m} \int_{t_{j-1}}^{t_j}\big\|\big[[E(\xi_j)-E(\sigma -\tau)] - [E(\sigma)-E(\sigma -\tau)]\big] \widetilde{v(s)}\big \|_{ \L_2^0}^2\, \d \sigma \notag\\
\lesssim\, &  \tau   \| \widetilde{v(s)}\|_{ \L_2^0}^2 + \sum_{ j = 3}^{m} \int_{t_{j-1}}^{t_j}\|(I+A)^{-1}[E(\xi_j + \tau - \sigma) - I]  (I+A)E(\sigma-\tau) \widetilde{v(s)}\|_{ \L_2^0}^2 \, \d \sigma \notag\\
& + \sum_{ j = 3}^{m} \int_{t_{j-1}}^{t_j}\|(I+A)^{-1}[E(\tau) - I]  (I+A)E(\sigma-\tau) \widetilde{v(s)}\|_{ \L_2^0}^2 \, \d \sigma 
\notag\\
\lesssim\, &\tau   \| \widetilde{v(s)}\|_{ \L_2^0}^2 
+ \tau^2  \int_{2\tau}^{t_m}  \| (I+A)E(\sigma-\tau)\widetilde{v(s)}\|_{ \L_2^0}^2\, \d \sigma   \quad\mbox{(here \eqref{A:sg-sta2} with $\mu = 1$ is used)} \notag\\
\lesssim\, &\tau   \| \widetilde{v(s)}\|_{ \L_2^0}^2 
+ \tau^2 \| \widetilde{v(s)}\|_{ \L_2^0}^2 \int_{2\tau}^{t_m}  (\sigma - \tau )^{-2}\, \d \sigma  \quad\mbox{(here \eqref{A:sg-sta1} with $\gamma =1$ is used)} \notag\\
\lesssim\, & \tau   \| \widetilde{v(s)}\|_{ \L_2^0}^2. \notag
\end{align}
Then, the term $\E |J_{241}^*|$ can be estimated by applying the above estimate with $\xi = t_{i+j} - s$ and $\widetilde{v(s)} = (I+A)^{\frac{1}{2}} P_{X} B(u(s))$, i.e., 
\begin{align}
\E |J_{241}^{*}|
\lesssim \tau \sum_{i = 0}^{m-1} \int_{t_i}^{t_{i+1}}   \E\|(I+A)^{\frac{1}{2} }P_{X} B(u(s))\|_{ \L_2^0}^2 \,\d s \lesssim  \tau.  \qquad\mbox{(here \eqref{L20:Sta-B} is used)} \notag
\end{align}
The term $\E |J_{242}^{*}|$ can be estimated by using a variable transform $t_{m-i} - \sigma = t$, i.e.,
\begin{align}
\E |J_{242}^{*}|
\lesssim\, &\E\sum_{i = 0}^{m-1} \int_{t_i}^{t_{i+1}}\Big\| (I+A)^{\frac34} \int_{0}^{t_{m - i}}  E(t_{m - i} - t) (I+A)^{-\frac 14}P_{X} [B(u(s)) - B(u_h^i)]\, \d t\Big\|_{ \L_2^0}^2\,\d s\notag\\
\lesssim\, & \sum_{i = 0}^{m-1} \int_{t_i}^{t_{i+1}}  \bigg(\int_{0}^{t_{m - i}}  (t_{m - i} - t)^{-\frac34} \d t\bigg)^2 \E\| (I+A)^{-\frac 14}P_{X} [B(u(s)) - B(u_h^i)]\|_{ \L_2^0}^2\,\d s \notag\\
&\hspace{205pt} 
\quad\mbox{(here \eqref{A:sg-sta4} with $\rho = \frac34$ is used)}  \notag\\ 
\lesssim\, & \sum_{i = 0}^{m-1} \int_{t_i}^{t_{i+1}}\E\| (I+A)^{-\frac 14}P_{X} [B(u(s)) - B(u_h^i)]\|_{ \L_2^0}^2\,\d s 
\notag\\
\lesssim\, & \sum_{i = 0}^{m-1} \int_{t_i}^{t_{i+1}}\E\sum_{\ell} \mu_{\ell}\|[B(u(s)) - B(u_h^i) ]\phi_{\ell}\|_{{H^{-\frac12}}}^2\,\d s 
\quad\mbox{(here \eqref{P_X-bound} with $s = \frac12$ is used)}  \notag\\ 
\lesssim\, & \sum_{i = 0}^{m-1} \int_{t_i}^{t_{i+1}}\E\| u(s) - u_h^i\|_{L^2}^2\,\d s 
\quad\mbox{(here \eqref{ass-con-nosie-1} used)} \notag \\
\lesssim\, & \sum_{i = 0}^{m-1} \int_{t_i}^{t_{i+1}}\big( \E \| u(s) - u(t_i)\|_{L^2}^2+\E\| u(t_i) - u_h^i\|_{L^2}^2\,\big)\, \d s \notag\\
\lesssim\, & \tau + \max_{1\le n\le N}\E\|u(t_{n}) - u_h^n\|_{L^2}^2
\quad\mbox{(here \eqref{SPDE-u-holder} is used)}  \notag\\
\lesssim\,& \tau + h^2. \qquad\mbox{(here \eqref{sfem-ferrs:u} is used, which is already proved)} 
\end{align}
The term $\E |J_{243}^{*}|$ can be estimated by directly using \eqref{lem:fem-ferrs-H1}, we obtain
\begin{align}\label{sFEM-err-p:I24c}
\E |J_{243}^{*}|
\lesssim\, &  \tau \sum_{i = 0}^{m-1}(\tau + h^2)\E\|B(u_h^i)\|_{ \L_2^0}^2 \\
\lesssim\, & (\tau + h^2)  \sum_{i = 0}^{m-1}\tau \big(1+\E\|u_h^i\|_{H^{\frac12}}^2 \big) 
\qquad\mbox{(here \eqref{L20-sta-L2} is used)}\notag\\
\lesssim\, & \tau + h^2.  \qquad\mbox{(here \eqref{sFEM-uhn-sta} is used)}\notag
\end{align}
The estimates for $J_{21}^*$, $J_{22}^*$, $J_{23}^*$ , $J_{24}^*$ above imply that 
\begin{align}\label{sFEM-err-p:I2}
\E | J_2^* |^2\lesssim \tau + h^2 .    
\end{align}

The term $J_3$ can be estimated directly by using the H\"older continuity of $f$ in time, i.e., 
\begin{align} 
\E | J_3^* |^2 
\lesssim&\, \sum_{n = 1}^m  \int_{t_{n-1}}^{t_n} \E\|f(t) - f(t_n)\|_{L^2}^2\,\d t 
\lesssim \sum_{n = 1}^m  \int_{t_{n-1}}^{t_n} \tau\,\d t 
\lesssim \tau . 
\end{align} 

The term $J_4$ can be estimated by
\begin{align}\label{sFEM-err-p:I4}
\E | J_4^* |^2 
\lesssim &\,  
\E \sum_{n = 1}^m  \int_{t_{n-1}}^{t_n} \big\|[B(u(t)) - B(u_h^{n-1})] \big\|_{{\L_2^0(\mathbb{L}^2, \, \mathbb{H}^{-1})}}^2
\d t \notag\\
\lesssim &\,    \sum_{n = 1}^m  \int_{t_{n-1}}^{t_n}\E\|u(t) - u_h^{n-1}\|_{L^2}^2 
\quad\mbox{( here \eqref{ass-con-nosie-1} is used)} \notag\\
\lesssim &\,    \sum_{n = 1}^m  \int_{t_{n-1}}^{t_n}
\big( \E\|u(t) - u(t_{n-1})\|_{L^2}^2 + \E\|u(t_{n-1}) - u_h^{n-1}\|_{L^2}^2 \big) \notag\\
\lesssim &\,   \tau + h^2 \qquad\mbox{(here \eqref{SPDE-u-holder} and \eqref{sfem-ferrs:u} are used)} .
\end{align} 
By substituting \eqref{J1-star} and \eqref{sFEM-err-p:I2}--\eqref{sFEM-err-p:I4} into \eqref{Error-J1234}, we obtain 
\begin{align*}
\E \Big\|\int_0^{t_m}p(t)\,\d t - \tau\sum_{n = 1}^m  p_h^n\Big\|_{L^2}^2  \lesssim \tau   + h^2 . 
\end{align*}
This proves desired error estimate in \eqref{sfem-ferrs:p} for the pressure.
\hfill\endproof

\section{Numerical experiments}\label{sec:num}
In this section, we present numerical tests to support the theoretical analysis in Theorem \ref{THM:sfem-ferrs} by illustrating the convergence of the fully discrete finite element solutions.  For a stable discretization in space, we use the prototypical MINI element; cf. \cite{arnold1984stable, girault2003quasi} for details.  All the computations are performed using the software package NGSolve (\url{https://ngsolve.org/}). 

We solve the stochastic Stokes equations \eqref{spde} in the two-dimensional square $D = [0, 1] \times [0, 1]$ under the stress boundary condition by the proposed scheme \eqref{fully-sFEM-weak} up to time $T$, with initial value $u_0 = (0, 0)$ and source term $f = (1, 1)^\top$. The noise term $B(u)\d W$ is determined by 
{
\begin{align}
\label{test-B}
B(u) &= \frac12 \left(\begin{array}{cc}\sqrt{u^2_1 +1} &  \sqrt{u^2_1 +1} \\[4pt]  \sqrt{u^2_2 + 1} & \sqrt{u^2_2 + 1} \end{array}\right),\\[5pt] 
\label{test-W}
W(t, \mathbf{x})
&=  \sum_{\ell_1 = 1}^{\infty}  \sum_{\ell_2 = 1}^{\infty}  \sqrt{\mu_{\ell_1 \ell_2}}
\left(\begin{array}{cc} 
\phi_{\ell_1 \ell_2}(\mathbf{x})  \\[4pt] 
\phi_{\ell_1 \ell_2}(\mathbf{x}) \end{array}\right) 
w_{\ell_1 \ell_2}(t) \qquad \forall \,\, t \in [0, T], 
\end{align}
for $\mathbf{x} = (x_1, x_2) \in D$, where 
$$\{w_{\ell_1 \ell_2}(t):\,  \ell_1,\ell_2 = 0, 1,2,\dots\}$$ is a set of independent $\R$-valued Wiener processes, 
\begin{align}
\label{test-phi-cos}
\{ \phi_{\ell_1 \ell_2}(\mathbf{x})  = \cos(\ell_1 \pi x_1) \cos(\ell_2 \pi x_2): \,  \ell_1,\ell_2 = 0, 1,2,\dots\} 
\end{align}
is an orthonormal basis of $L^2(D)$, and 
\begin{align}
\label{test-mu-cos}
\mu_{\ell_1 \ell_2} & = 
\left \{
\begin{aligned}
&0 \quad &&\mbox{for} \,\, \,(\ell_1, \ell_2) = (0, 0),\\
& (\ell_1^2 + \ell_2^2)^{-(r + \varepsilon)}\quad &&\mbox{for} \,\, \, (\ell_1,\ell_2)\in \mathbb{Z}^2 /\{(0, 0)\},  
\end{aligned}
\right .
\end{align}
with $\varepsilon = 0.1$ and $r \in (0, 2]$ determining the regularity of the noise. 
%

The noise term $B(u)\d W$ determined by \eqref{test-B}--\eqref{test-W} can be written as 
\begin{align}\label{B-dW}
B(u)\d W(t)
&= \sum_{\ell_1 = 1}^{\infty}  \sum_{\ell_2 = 1}^{\infty}  \sqrt{\mu_{\ell_1 \ell_2}} \left(\begin{array}{cc}\sqrt{u^2_1 +1} \\ \sqrt{u^2_2 +1}  \end{array}\right) \phi_{\ell_1 \ell_2}(\mathbf{x}) \, w_{\ell_1 \ell_2}(t),
\end{align}
which is non-solenoidal and was used to measure the effectiveness of numerical methods for the stochastic Stokes/NS equation (with $\phi_{\ell_1 \ell_2}(\mathbf{x}) = 2 \sin(\ell_1 \pi x_1) \sin(\ell_2 \pi x_2)$ therein); see \cite{fengProhl2021optimally,fengVo2022analysis}. 
Based on  the  coefficients $\mu_{\ell_1 \ell_2}$ given in \eqref{test-mu-cos}, 
the regularity of the series $W(t)$ presented in  \eqref{test-W} can be characterized as follows:
\begin{align}\label{test-eigen-cos}
\big \|(-\Delta)^{\frac{r-1}{2}}\big\|_{\L_2^0}= 
\big \|(-\Delta)^{\frac{r-1}{2}} Q^{\frac12}\big\|_{\L_2(\mathbb{L}^2, \mathbb{L}^2)} =
\sum_{\ell_1 = 0}^\infty\sum_{\ell_2 = 0}^\infty  \lambda_{\ell_1 \ell_2}^{r-1} \mu_{\ell_1 \ell_2} \lesssim 1 \quad \, \mbox{for} \quad r \in (0, 2],
\end{align}
where $\lambda_{\ell_1, \ell_2} = \pi^2(\ell_1^2 + \ell_2^2)$ represents eigenvalues of  $-\Delta$. In the case $r=2$, the series given in \eqref{test-W} constitutes a $Q$-Wiener process which satisfies Assumption \ref{ass-W}, and the noise in \eqref{B-dW} fulfills the conditions 
\eqref{ass-con-nosie-1}--\eqref{ass-con-nosie-2} in Assumption \ref{ass-B}. In particular, the Wiener processes in \eqref{test-W} is in $L^2(\Omega, H^1(D))$ but not in $L^2(\Omega, H^{1+\varepsilon}(D))$. 

We consider three cases in the numerical experiments:
\begin{itemize}
\item[Case] I : $r = 2$. In this case, the Assumption \ref{ass-W} is satisfied.
\item[Case] II : $r = 1$. The noise is trace-class, but Assumption \ref{ass-W} is not satisfied. 
\item[Case] III : $r = 0.5$. The noise is not trace-class.
\end{itemize}
\begin{figure}[htp]
\centerline{
\includegraphics[width=2.8in]{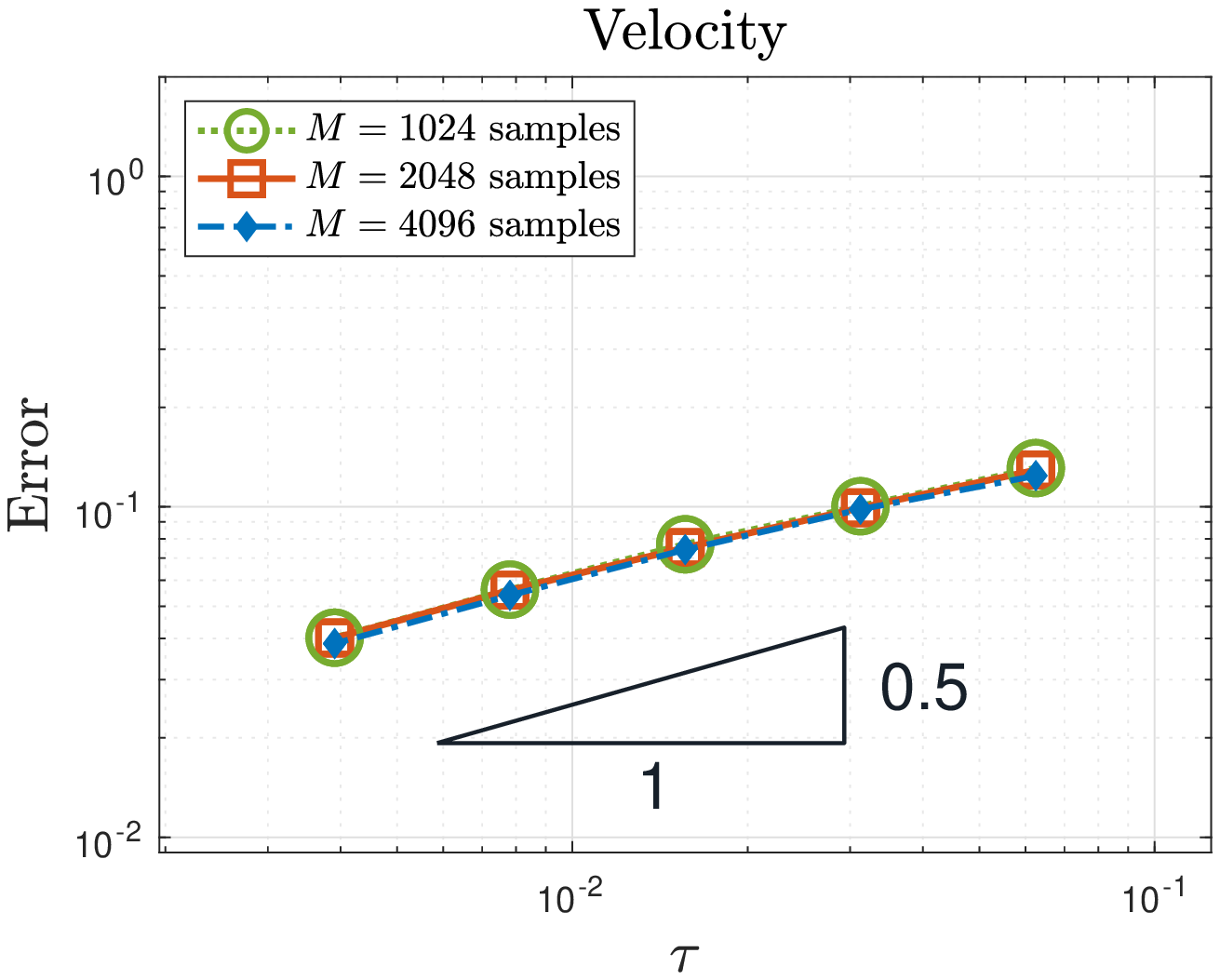}
\includegraphics[width=2.8in]{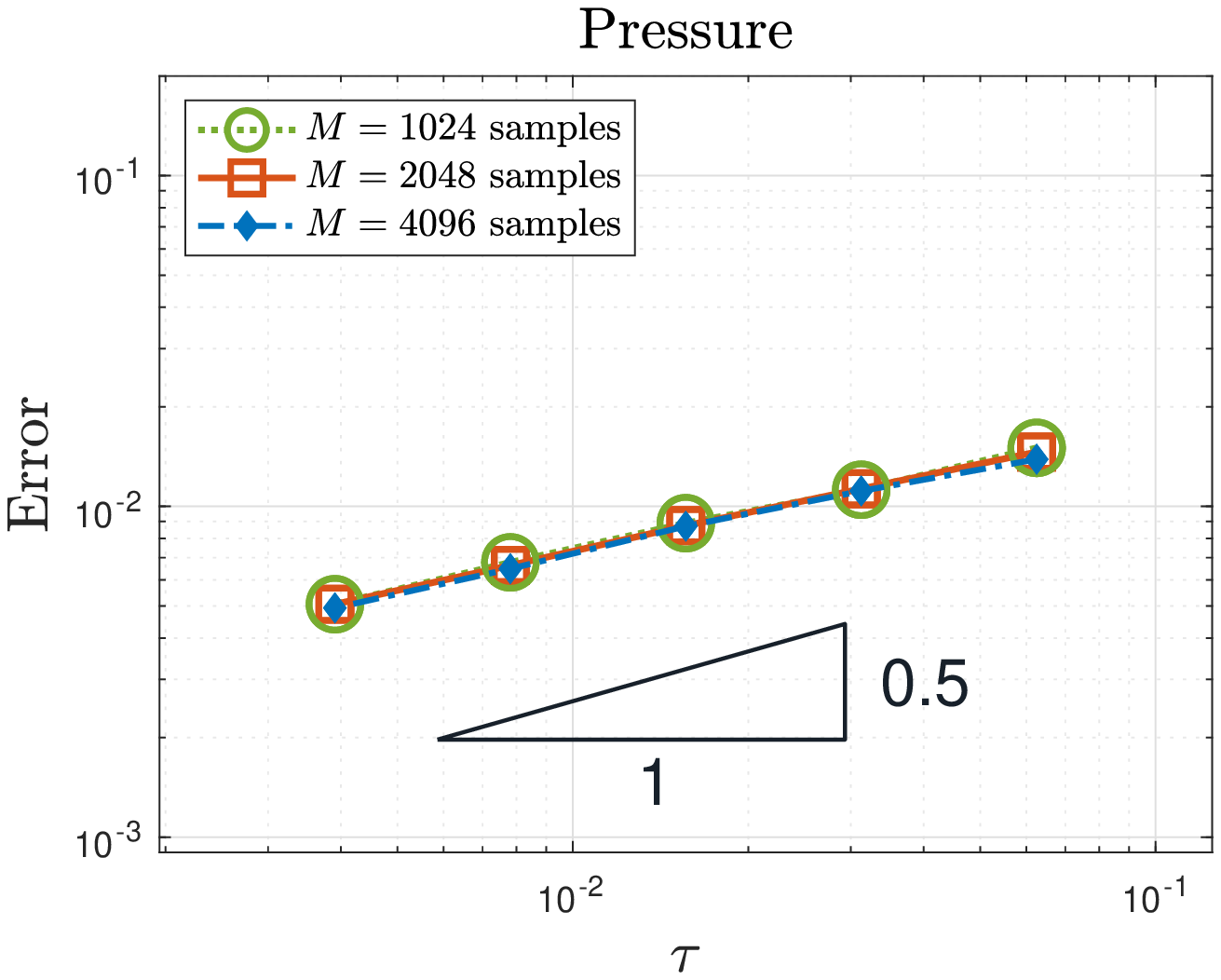}
}
\vspace{-10pt}
\caption{{Time discretization errors at $T = 1$ for Case I with $h = 2^{-6}$.}}
\label{fig_errorRate_Time_cos}
\hspace{8pt}
\centerline{
\includegraphics[width=2.8in]{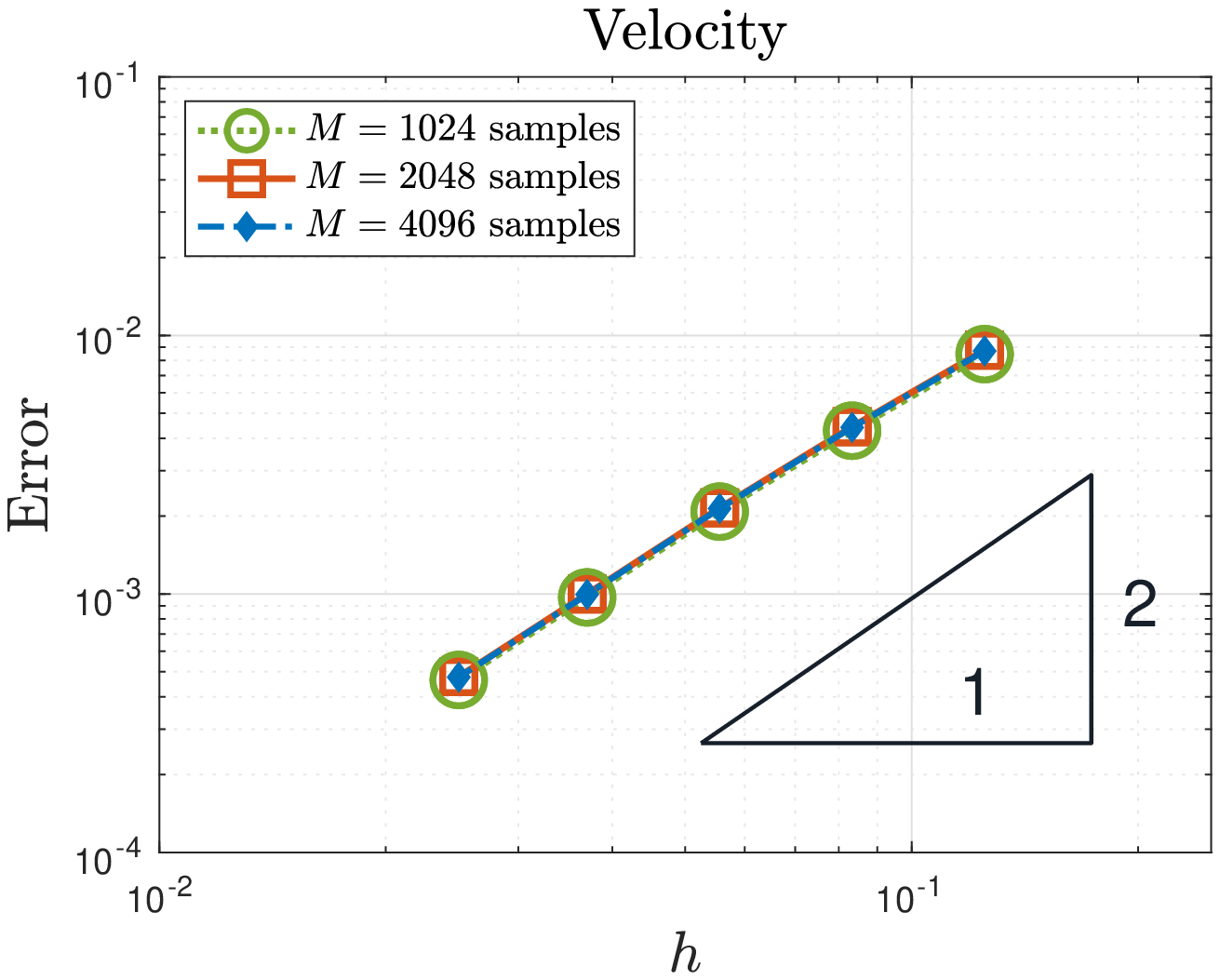}
\includegraphics[width=2.8in]{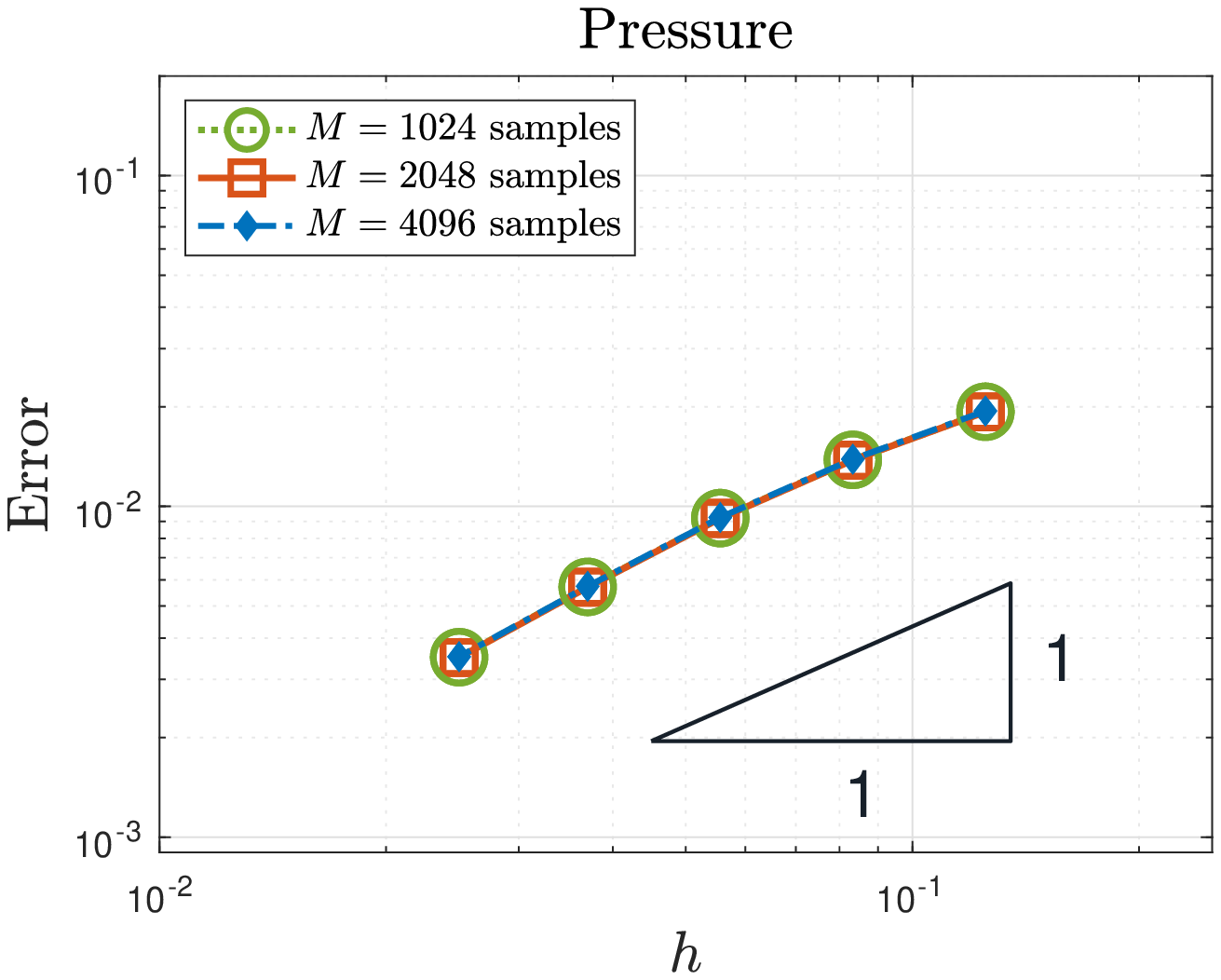}
}
\vspace{-20pt}
\caption{{Spatial discretization errors at $T = 1$ for Case I with $\tau = 2^{-8}$.}}
\label{fig_errorRate_Space_cos}
\end{figure}  

The errors of the numerical solutions in Case I are presented in Figures \ref{fig_errorRate_Time_cos} and \ref{fig_errorRate_Space_cos}. 
The expectations of the errors are computed as averages over $M$ samples, where $M=1024, 2048, 4096$, respectively.
The numerical results in Figures \ref{fig_errorRate_Time_cos} and \ref{fig_errorRate_Space_cos} indicate that the numerical solutions have half-order convergence in time for both velocity and pressure, second-order convergence in space for the velocity, and first-order convergence in space for the pressure. Therefore, the convergence orders observed in the numerical experiments are consistent with the theoretical results proved in Theorem \ref{THM:sfem-ferrs}.

\begin{figure}[htp]
\centerline{
\includegraphics[width=2.95in]{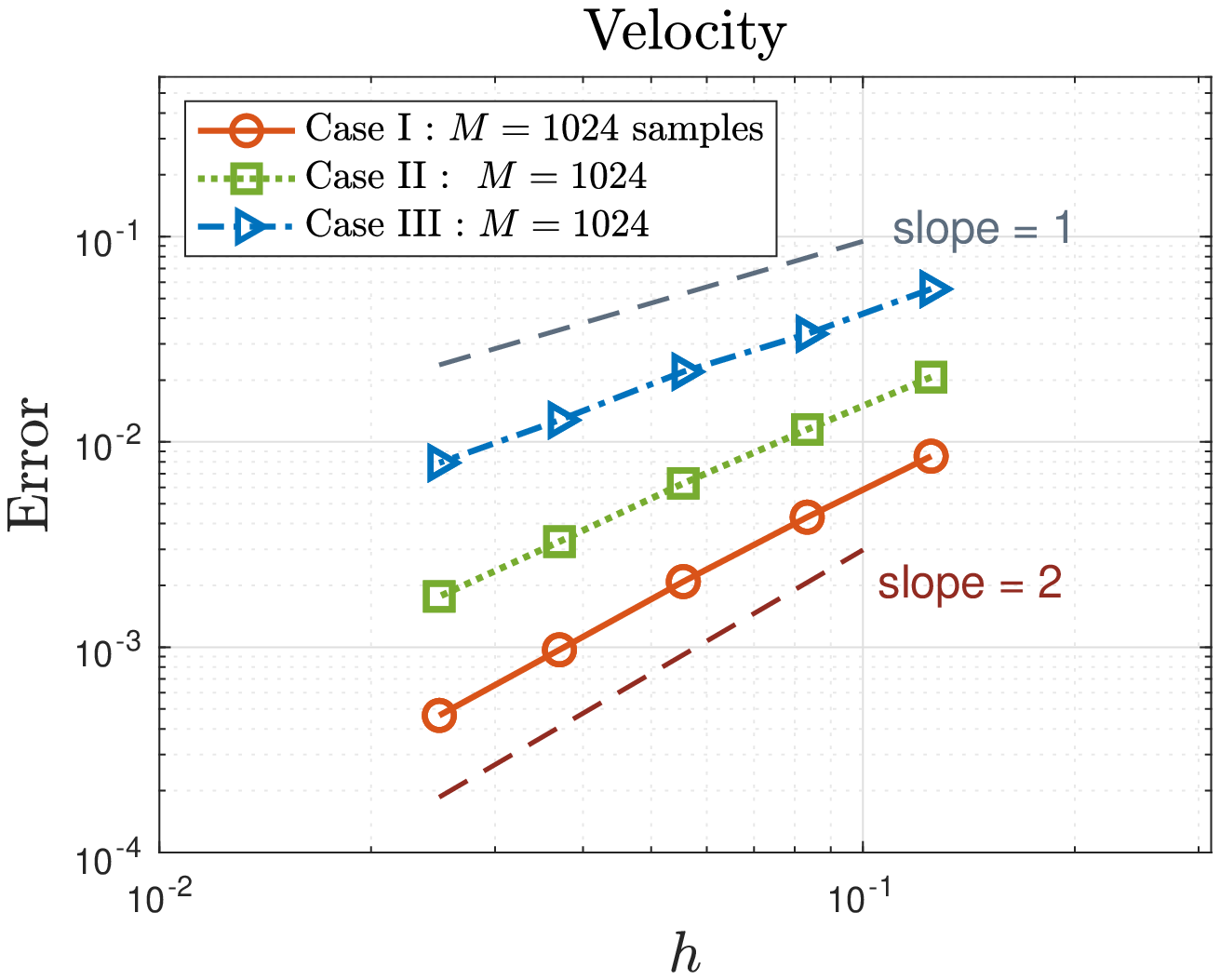}
\includegraphics[width=2.95in]{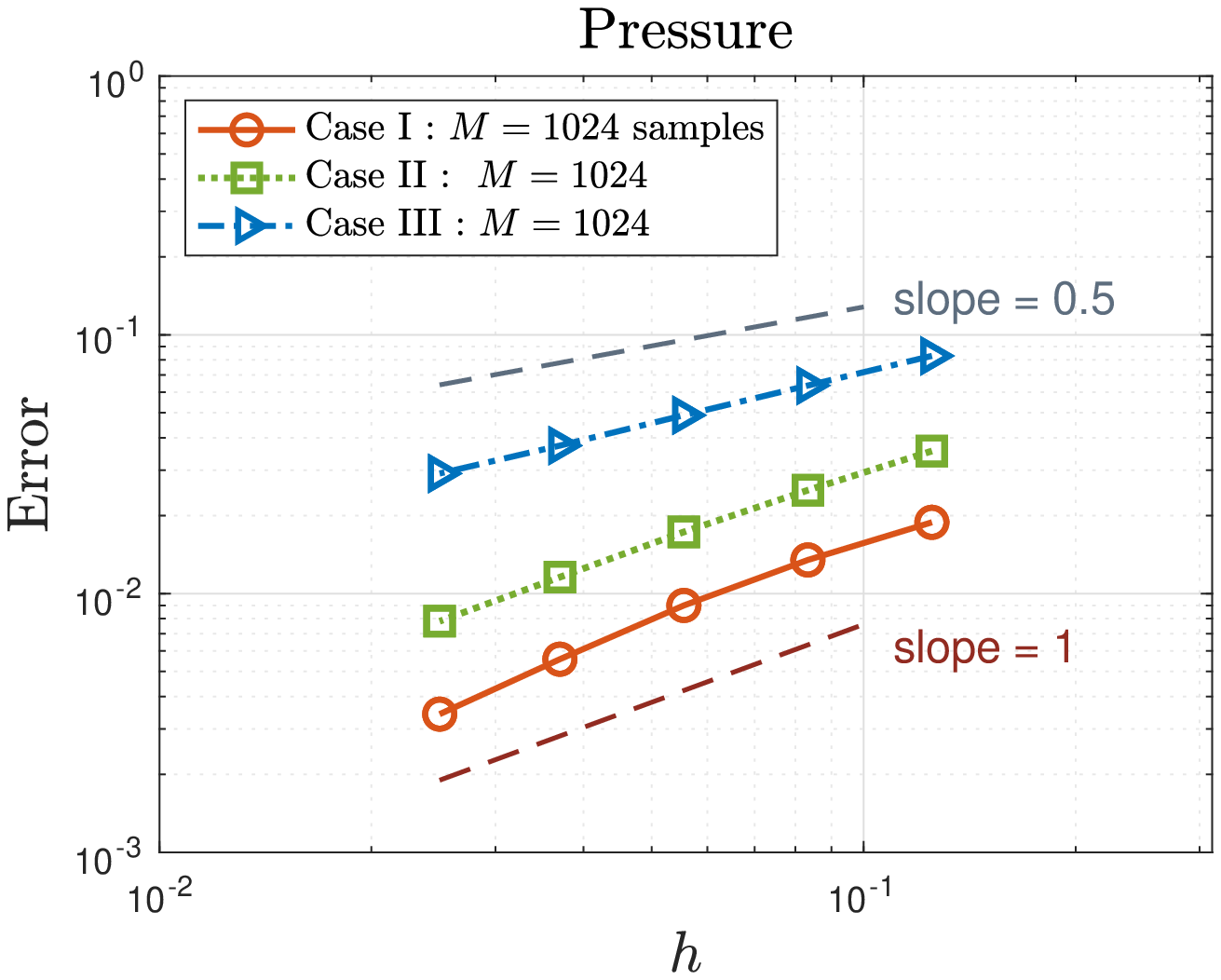}
}
\vspace{-10pt}
\caption{{  Spatial discretization errors at $T = 1$ for Cases I, II and III.}}
\label{fig_errorRate_Space_cos_r}
\end{figure}
The spatial discretization errors for noises in Cases I, II, and III are presented in Figure \ref{fig_errorRate_Space_cos_r}, where the expectations are approximated by computing averages over $M = 1024$ samples.   Notably,  Figures \ref{fig_errorRate_Time_cos} and \ref{fig_errorRate_Space_cos}  demonstrate that the number of samples,  $M = 1024$,  is already sufficiently large to capture the influence of the noise on the convergence rate. The numerical results in Figure \ref{fig_errorRate_Space_cos_r} indicate that order reduction may occur if  Assumption \ref{ass-W} is not satisfied.
}

\section{Conclusion}

We have proved higher-order strong convergence of fully discrete mixed FEMs for the stochastic Stokes equations under the stress boundary condition driven by a stochastic noise satisfying condition \eqref{ass-BdW-weak}. The error estimates of $O(\tau^\frac12 + h^2)$ and $O(\tau^\frac12 + h)$ are proved for the velocity and pressure approximations, respectively, for a semi-implicit mixed finite element method. The analysis is based on new estimates of the semi-discrete and fully discrete semigroups associated to the Stokes operator, and the $H^1$-stability of the orthogonal projection onto the discrete divergence-free finite element subspace, as shown in Section \ref{sec:est-dis-semigroup}. The improved convergence orders are consistent with the numerical experiments. 

For the simplicity of illustration, we have focused on the stochastic Stokes equations in this article. However, the methodology introduced in this article may also be extended to the stochastic NS equations to obtain higher-oder convergence in space.

\renewcommand{\theequation}{A.\arabic{equation}}
\renewcommand{\thelemma}{A.\arabic{lemma}}
\setcounter{equation}{0}

\section*{Appendix A: Well-posedness and regularity of the mild solution}$\,$

In this Appendix we prove Proposition \ref{thm-u-stability}, including existence and uniqueness of a mild solution to the stochastic Stokes equations Assumptions \ref{ass-W}--\ref{ass-u0}, and the regularity of the mild solution. 
\medskip

{\it Existence and uniqueness:} 
Let 
$Y = C([0, T]; L^2(\Omega;X)) \cap L^2(\Omega, \{\F_t\}_{t \ge 0}; L^2(0, T;X))$, 
i.e., the predictable subspace of $C([0, T]; L^2(\Omega;X))$, where $X =\{ v\in L^2(D)^d: \nabla\cdot v=0\}$. For any $v \in Y$ we denote by $Mv$ the function defined by
\begin{align}
Mv(t)
=E(t)u^0 + \int_{0}^{t} E(t-s) P_X f(s)\d s
+ \int_{0}^{t}E(t-s) P_X B(v(s))\, \d W(s) \,\,\, \mbox{for} \,\, t\in [0, T],
\end{align}
Under Assumptions \eqref{ass-W}--\eqref{ass-u0}, the following results hold, which are simple modifications of \eqref{lem:B-noise-holder}:
\begin{align}\label{E-PX-B-sta1}
\|E(t)P_XB(v)\|_{ \L_2^0}^2 &\lesssim t^{-\frac12} (1 +  \E \|v\|_{L^2}^2) &&t >0,&\\
\label{E-PX-B-sta2}
\|E(t)P_X[B(u) - B(v)]\|_{ \L_2^0}^2 &\lesssim t^{-\frac12}\, \E \|u - v\|_{L^2}^2  &&t >0.&
\end{align}
Then $Mv$ is predictable and 
\begin{align}
\E \|Mv(t)\|_{L^2}^2 
&\lesssim \E \|u^0\|_{L^2}^2 +  \E\int_{0}^{t} \|f(s)\|_{L^2}^2 \d s
+ \E\int_{0}^{t} \| E(t-s) P_X B(v(s))\|_{ \L_2^0}^2 \, \d s \notag\\
&\lesssim (1+ \E \|u^0\|_{L^2}^2 ) 
+ \int_{0}^{t} (t - s)^{-\frac {1}{2}} \E \| v(s)\|_{L^2}^2  \, \d s 
\notag\\
&\lesssim 1 + \|v\|_Y,
\notag
\end{align}
which implies that $Mv \in L^{\infty}(0, T; L^2(\Omega;X))$. 
By considering $\E\|Mv(t_2)-Mv(t_1)\|_{L^2}^2$ one can also prove the continuity of $Mv$ in time. 
As a result,  
$Mv \in C([0, T]; L^2(\Omega;X))$ and therefore the map $M : Y \to Y$ is well defined.

Clearly, a mild solution of \eqref{spde_abstact} is equivalent to a fixed point of $M$.

To prove the existence and uniqueness of a fixed point of $M$, we define $Y_\lambda$ to be
$Y$ equipped with an equivalent norm
\begin{align}
\|w\|_{Y_\lambda} := \max_{t\in [0, T]} e^{-\lambda t} \big(\E \|w\|_{L^2}^2\big)^{\frac12}  \qquad \mbox{for} \,\, w \in Y_\lambda.
\end{align}
If $v_1, v_2 \in Y$  then
\begin{align}
Mv_1(t) - Mv_2(t) = \int_0^t E(t - s) P_X[B(v_1(s)) - B(v_2(s))] \,\d W(s) \qquad \mbox{for} \,\, t \in [0,T] . \notag
\end{align} 
By applying \eqref{E-PX-B-sta2} we obtain
\begin{align}
e^{-2\lambda t}\, \E\|Mv_1(t) - Mv_2(t)\|_{L^2}^2 
&\lesssim \int_0^t e^{-2\lambda t} (t - s)^{-\frac {1}{2}} \E \| v_1(s) - v_2(s)\|_{L^2}^2 \, \d s   \notag\\
&= \int_0^t e^{-2\lambda (t-s)} (t - s)^{-\frac {1}{2}}  e^{-2\lambda s}\E \|v_1(s) - v_2(s)\|_{L^2}^2  \,\d s   \notag\\
&\lesssim \int_0^t e^{-2\lambda \sigma} \sigma^{-\frac {1}{2}} \,\d \sigma\|v_1 - v_2\|_{Y_\lambda}, \notag\\
&\lesssim \lambda^{-\frac12} \|v_1 - v_2\|_{Y_\lambda}, \notag
\end{align}
which implies that
\begin{align}
\|Mv_1 - Mv_2\|_{Y_{\lambda}} 
\lesssim \lambda^{-\frac14}\|v_1 - v_2\|_{Y_\lambda}.
\end{align}
As a result, by choosing a sufficiently large $\lambda$, the map $M:Y_\lambda\rightarrow Y_\lambda$ is a contraction. By the Banach fixed point theorem, $M$ has a unique fixed point in  $Y_\lambda$ (which
consists of the same elements as $Y$). 
This proves the existence of a unique mild solution satisfying \eqref{SPDE-mild}.


\medskip
{\it Regularity:} 
We first consider the case $H^{\beta}$ {with parameter $\beta \in (\frac{d}{2}, 2)$}.
It\^o's isometry \eqref{Ito_isometry} together with \eqref{IA-eqi-norms} and \eqref{A:sg-sta1} yields
\begin{align}\label{H-beta-u}
\E\|u(t)\|_{H^\beta}^2 
\lesssim\, &\E\|E(t)(I+A)^{\frac {\beta}{2}}u^0\|_{L^2}^2 +  \E\Big(\int_{0}^{t} \| (I+A)^{\frac {\beta}{2}}E(t-s) P_X f(s)\|_{L^2}\d s \Big)^2\notag\\
&+ \E \Big\|\int_{0}^{t} (I+A)^{\frac {\beta}{2}} E(t-s) P_X B(u(s))\d W(s)\Big\|_{L^2}^2\notag\\
\lesssim\, &\E \|u^0\|_{H^\beta}^2 
+ \E\Big(\int_{0}^{t} (t - s)^{-\frac {\beta}{2}} \|f(s)\|_{L^2} \d s\Big)^2 \notag\\
&+ \E\int_{0}^{t} \| (I+A)^{\frac {\beta -  1}{2}}E(t-s)  (I+A)^{ \frac {1}{2}}P_X B(u(s))\|_{ \L_2^0}^2 \d s \notag\\
\lesssim\, & (1+ \E \|u^0\|_{H^\beta}^2 ) 
+ \int_{0}^{t} (t - s)^{-(\beta- 1)} \E \| (I+A)^{\frac 12}P_X B(u(s))\|_{ \L_2^0}^2  \d s 
\notag\\
\lesssim\, & (1+ \E \|u^0\|_{H^\beta}^2 ) + \int_{0}^{t} (t - s)^{-(\beta- 1)} \E \|u(s)\|^2_{ H^\beta}   \d s ,
\end{align}
where \eqref{lem:B-noise-sta} is used in the last inequality. By using the {generalized Gronwall's inequality in \cite[Lemma 6.3]{elliott1992error} (or \cite[Lemma A.2]{kruse2014strong}}), we arrive at
\begin{align}\label{SPDE-u-sta-Hbeta}
\sup_{t \in [0,T]}  \E\|u(t)\|_{H^\beta}^2  \lesssim 1+ \E \|u^0\|^2_{ H^\beta}  \qquad {\mbox{for all}\quad  \frac{d}{2} < \beta < 2}.
\end{align}
For the regularity in time, it follows from \eqref{lem:B-noise-sta2}, \eqref{lem:B-noise-sta} and \eqref{SPDE-u-sta-Hbeta} that
\begin{align} 
&\sup_{t \in [0,T]}  \E\|B(u(t))\|_{ \L_2^0}^2 + \sup_{t \in [0,T]}  \E\|(I+A)^{\frac {1}{2}}P_X B(u(t))\|_{ \L_2^0}^2 \lesssim 1+ \E \|u^0\|^2_{H^2}, \notag
\end{align}
which together with \eqref{A:sg-sta1}--\eqref{A:sg-sta2} shows
\begin{align}
\E \|u(t) - u(s)\|_{H^1}^2 \lesssim \, & \E\|(I+A)^{-\frac{1}{2}}(E(t-s) - I)(I+A) u^0\|_{L^2}^2 \\
&+ \E\Big(\int_{s}^{t} \| (I+A)^{\frac{1}{2}}E(t-\sigma) P_X f(\sigma)\|_{L^2}\d \sigma \Big)^2 \notag\\
&+ \E\int_{s}^{t}\|(I+A)^{\frac{1}{2}}E(t-\sigma) P_X B(u(\sigma))\|_{ \L_2^0}^2\d \sigma\notag\\
\lesssim\, & (t-s)(1 + \E\|u^0\|_{H^2}^2 ) 
+ \int_{s}^{t} \E\|(I+A)^{\frac{1}{2}} P_X B(u(\sigma))\|_{ \L_2^0}^2\d \sigma 
\notag \\
\lesssim\, & (1+ \E \|u^0\|_{H^2}^2 )(t-s). \notag
\end{align}
and
\begin{align}
\E \|u(t) - u(s)\|_{H^\beta}^2 \lesssim\, & \E\| (I+A)^{-\frac{2- \beta}{2}}(E(t-s) - I)(I+A) u^0\|_{L^2}^2 \\
&+ \E\Big(\int_{s}^{t} \| (I+A)^{\frac{\beta}{2}}E(t-\sigma) P_X f(\sigma)\|_{L^2}\d \sigma \Big)^2\notag\\
&+ \E\int_{s}^{t}\|(I+A)^{\frac{\beta-1}{2}}E(t-\sigma) (I+A)^{\frac{1}{2}}P_X B(u(\sigma))\|_{ \L_2^0}^2\d \sigma \notag\\
\lesssim\, & (t-s)^{2 -\beta }(1 + \E\|u^0\|_{H^2}^2 ) 
+ \int_{s}^{t} (t - \sigma)^{1-\beta}\E\| (I+A)^{\frac{1}{2}} P_X B(u(\sigma))\|_{ \L_2^0}^2\d \sigma 
\notag \\
\lesssim\, & (1+ \E \|u^0\|_{H^2}^2 )(t-s)^{2 - \beta}
\qquad {\mbox{for all}\quad  \frac{d}{2} < \beta < 2}. \notag
\end{align}
This proves \eqref{SPDE-u-holder} and \eqref{SPDE-u-holder-Hbeta}.

With the help of the above estimates, the $H^2$ stability follows from 
\begin{align}
\E\|u(t)\|_{H^2}^2
\lesssim\, &\E\|u^0\|_{H^2}^2 
+ \E \Big\|\int_{0}^{t}  (I+A)E(t-s) P_X f(s)\d s  \Big\|_{L^2}^2 
\\
&+ \E \Big\|\int_{0}^{t} (I+A)^{\frac12} E(t-s)  (I+A)^{\frac12} P_X B(u(s))\,\d W(s)\Big\|_{L^2}^2\notag\\
=:\, &\E\|u^0\|_{H^2}^2  + \E|S_1|^2 + \E|S_2|^2,\notag
\end{align}
where we  use the triangle inequality to get
\begin{align}
\E|S_1|^2 \lesssim \,&\E \Big( \int_{0}^{t} \| (I+A)E(t-s) P_X [f(s) - f(t)]\|_{L^2}\d s  \Big)^2 
+ \E \Big\|(I+A)\int_{0}^{t}  E(t-s) P_X f(t)\d s  \Big\|_{L^2}^2 
\notag\\
\lesssim\, & \E \Big( \int_{0}^{t} (t-s)^{-1} \|f(s) - f(t)\|_{L^2} \d s \Big)^2
+ \E \| f(t)\|_{L^2}^2
\quad \mbox{(here \eqref{A:sg-sta1} and \eqref{A:sg-sta4} are used)} 
\notag \\
\lesssim\, &1,  \qquad  \mbox{(here \eqref{ass-con-f-sta}--\eqref{ass-con-f-lip} are used)}  \notag \\[5pt]
\E|S_2|^2 \lesssim\, &\E \int_{0}^{t} \|(I+A)^{\frac12} E(t-s)  (I+A)^{\frac12} P_X  [B(u(s)) - B(u(t))]\|_{ \L_2^0}^2 \,\d s\notag\\
&+ \E \int_{0}^{t} \|(I+A)^{\frac12} E(t-s)  (I+A)^{\frac12} P_X  B(u(t))\|_{ \L_2^0}^2 \,\d s\notag\\
\lesssim\, &\int_{0}^{t} (t-s)^{-1}\E\|(I+A)^{\frac12} P_X  [B(u(t))-B(u(s))]\|_{ \L_2^0}^2 \, \d s
\quad \mbox{(here \eqref{A:sg-sta1} is  used)}  \notag\\
&+ \E \|  (I+A)^{\frac12} P_X B(u(t))\|_{ \L_2^0}^2
\quad \mbox{(here \eqref{A:sg-sta3} with $\rho = 1$ is used)} 
\notag \\
\lesssim\, & \int_{0}^{t} (t-s)^{-1}\E\|u(t)-u(s)\|_{H^\beta}^2 \d s  + 1+ \E \|u^0\|^2_{H^2}
\quad  \mbox{(here \eqref{L20-sta-H1} is used)} 
\notag\\
\lesssim\, &1+ \E \|u^0\|_{H^2}^2   + \int_{0}^{t} (t - s)^{1-\beta} \d s 
\quad  \mbox{(here \eqref{SPDE-u-holder-Hbeta} for {$\beta \in (\frac{d}{2}, 2)$ is used)}} 
\notag\\
\lesssim\, &1+ \E \|u^0\|_{H^2}^2.\notag
\end{align}
This proves \eqref{SPDE-u-sta}.
The proof of Proposition \ref{thm-u-stability} is completed.
\hfill\endproof

\begin{remark}\upshape 
We have shown the regularity of the mild solution more formally (like a priori estimates for PDEs) than rigorously. Rigorously speaking, we need to firstly ensure that $u\in L^\infty(0,T;L^2(\Omega;H^\beta(D)^d))$ before using the last inequality in \eqref{H-beta-u}. Here we briefly discuss how the proof of Proposition \ref{thm-u-stability} can be made  rigorous by considering the following regularized problem (in the semigroup formulation): 
Find $u_\varepsilon\in C([0, T]; L^2(\Omega;X)) \cap L^2(\Omega, \{\F_t\}_{t \ge 0}; L^2(0, T;X))$ such that 
\begin{align}\label{mollified-problem}
u_\varepsilon(t)
=E(t)u^0 + \int_{0}^{t} E(t-s) P_X f(s)\d s
+ \int_{0}^{t}E(t-s) P_X B(\phi_\varepsilon *\Lambda u_\varepsilon(s))\, \d W(s) \,\,\, \mbox{for} \,\, t\in [0, T],
\end{align}
where $\Lambda: L^1(\Omega)^d\rightarrow L^1(\R^d)^d$ is Stein's extension operator (see \cite[p.\ 181, Theorem 5]{Stein1970}) which is bounded from $H^{s}(\Omega)^d$ to $H^{s}(\R^d)^d$ for all $s\ge 0$, and $\phi_\varepsilon $ is a standard mollifier which has the following estimates:  
\begin{align*}
\|\phi_\varepsilon *\Lambda u_\varepsilon\|_{H^s(\R^d)} \le C\|u_\varepsilon\|_{H^s(\Omega)} 
\quad\mbox{and}\quad 
\|\phi_\varepsilon *\Lambda u_\varepsilon\|_{H^s(\R^d)} \le C\varepsilon^{-s}\|u_\varepsilon\|_{L^2(\Omega)} 
\quad\,\,\,\forall\,s\ge 0. 
\end{align*} 

For the problem in \eqref{mollified-problem}, one can prove the existence and uniqueness of a weak solution $u_\varepsilon$ by using the same fixed-point argument in Appendix A, and then prove the higher regularity of the solution by using \eqref{H-beta-u} and the generalized Gronwall's inequality in \cite[Lemma 6.3]{elliott1992error}. However, for the mollified problem in \eqref{mollified-problem}, inequality \eqref{H-beta-u} is rigorous because we already know that $\phi_\varepsilon *\Lambda u_\varepsilon \in L^\infty(0,T;L^2(\Omega;H^\beta(D)^d))$ for $u_\varepsilon\in L^\infty(0,T;L^2(\Omega;L^2(D)^d))$. The latter has been proved by using the fixed-point argument. Therefore, \eqref{H-beta-u} can be slightly changed to 
\begin{align*}
\E\|u_\varepsilon(t)\|_{H^\beta}^2 
\lesssim\, & (1+ \E \|u^0\|_{H^\beta}^2 ) + \int_{0}^{t} (t - s)^{-(\beta- 1)} \E \|\phi_\varepsilon *\Lambda  u_\varepsilon(s)\|^2_{ H^\beta}   \d s \\
\lesssim\, & (1+ \E \|u^0\|_{H^\beta}^2 ) + \varepsilon^{-2\beta} \int_{0}^{t} (t - s)^{-(\beta- 1)} \E \|u_\varepsilon(s)\|^2_{L^2}   \d s 
\lesssim  1+ \varepsilon^{-2\beta} 
\end{align*}
which implies $u_\varepsilon \in C([0,T];L^2(\Omega;H^\beta(D)^d))$ qualitatively. Then, once we already know that $u_\varepsilon \in C([0,T];L^2(\Omega;H^\beta(D)^d))$, we can apply \eqref{H-beta-u} again with the following quantitative estimate: 
\begin{align*}
\E\|u_\varepsilon(t)\|_{H^\beta}^2 
\lesssim\, & (1+ \E \|u^0\|_{H^\beta}^2 ) + \int_{0}^{t} (t - s)^{-(\beta- 1)} \E \|\phi_\varepsilon *\Lambda  u_\varepsilon(s)\|^2_{ H^\beta}   \d s \\
\lesssim\, & (1+ \E \|u^0\|_{H^\beta}^2 ) + \int_{0}^{t} (t - s)^{-(\beta- 1)} \E \|u_\varepsilon(s)\|^2_{H^\beta}   \d s .
\end{align*}
This leads to \eqref{SPDE-u-sta-Hbeta} by applying the deterministic generalized Gronwall's inequality in \cite[Lemma 6.3]{elliott1992error}, and the constant in the inequality would be independent of $\varepsilon$. In this way, all the estimates in Proposition \ref{thm-u-stability} can be proved with constants independent of $\varepsilon$, i.e.,
\begin{align}\label{u_epsilon_H^beta}
\begin{aligned}
\sup_{t \in [0,T]}  \E \|u_\varepsilon(t)\|_{H^2}^2 &\lesssim \big(1+ \E\|u^0\|_{H^2}^2\big) , \\
\E \|u_\varepsilon(t) - u_\varepsilon(s)\|_{H^1}^2 &\lesssim \big(1+ \E\|u^0\|_{H^2}^2\big)(t - s)
&&\forall\, 0\le s\le t\le T \\[5pt]
\E \|u_\varepsilon(t) - u_\varepsilon(s)\|_{H^\beta}^2 &\lesssim  \big(1+ \E\|u^0\|_{H^2}^2\big)(t - s)^{2 - \beta} 
&&\forall\, 0\le s\le t\le T,\,\,\, \forall\, \mbox{$\beta\in(\frac{d}{2},2)$} .
\end{aligned}
\end{align}

Finally, by comparing \eqref{SPDE-mild} with \eqref{mollified-problem}, the following error estimate can be shown easily: 
\begin{align}\label{mollified-converg}
\E\| u(t) - u_\varepsilon(t) \|_{L^2}^2
\lesssim 
\int_{0}^{t} \E\| u(s) - u_\varepsilon(s) \|_{L^2}^2\, \d s
+ \int_{0}^{t} \E\| u(s) - \phi_\varepsilon *\Lambda u(s) \|_{L^2}^2 \, \d s .
\end{align}
Since $\lim\limits_{\varepsilon\rightarrow 0}\| u(s) - \phi_\varepsilon *\Lambda u(s) \|_{L^2}^2 = 0$ for almost all $(\omega,t)\in \Omega\times(0,T)$, and $\| u(s) - \phi_\varepsilon *\Lambda u(s) \|_{L^2}^2\lesssim \| u(s) \|_{L^2}^2$ with $\| u(s) \|_{L^2}^2$ integrable on $\Omega\times(0,T)$, by the Lebesgue dominated convergence theorem we have 
$$
\lim_{\varepsilon\rightarrow 0 } \int_{0}^{t} \E\| u(s) - \phi_\varepsilon *\Lambda u(s) \|_{L^2}^2 \, \d s = 0. 
$$
This, together with \eqref{mollified-converg} and the standard deterministic Gronwall's inequality, implies that $u_\varepsilon$ converges to $u$ in $C([0, T]; L^2(\Omega;X))$ as $\varepsilon\rightarrow 0$. 
Since $u_\varepsilon$ satisfies the estimates in \eqref{u_epsilon_H^beta} with right-hand sides independent of $\varepsilon$, by passing to the limit $\varepsilon\rightarrow 0$, the limit function $u$ also satisfies these estimates. This proves Proposition \ref{thm-u-stability} rigorously. 
\end{remark}

\section*{Appendix B: Proof of several technical lemmas}$\,$
\renewcommand{\theequation}{B.\arabic{equation}}
\renewcommand{\thelemma}{B.\arabic{lemma}}
\setcounter{equation}{0}

\medskip
\noindent{\bf Proof of Lemma~\ref{lem-IA-sta}.} 
By Korn's inequality \cite[Theorem 2.4]{john2002analysis}, the following equivalence relation holds: 
\begin{align}\label{az-bilinear-H1}
\|v\|_{L^2}^2 +  \|\D(v)\|_{L^2}^2 \sim \|v\|_{H^1}^2 \qquad \forall\, v \in  H^1(D)^d \cap X .
\end{align}
This implies the coercitivity of the operator $I+A$ and the existence of an inverse operator $(I+A)^{-1}:X\rightarrow X\cap H^1(D)^d$ 
(see the Lax--Milgram Lemma in \cite[Theorem 6.2-1]{ciarlet2013linear}). 
Since $X\cap H^1(D)^d$ is compactly embedded into $X$, it follows that $(I+A)^{-1}$ is compact. 
For such a symmetric positive operator $I+A$ with compact inverse, its fractional powers can be defined by means of the spectral decomposition \cite[Appendix B.2]{kruse2014strong}. 

The $H^2$ elliptic regularity of the Stokes equations implies that $(I+A)^{-1}:X\rightarrow X\cap H^2(D)^d$, which implies $\|u\|_{H^2}\lesssim \|(I+A)u\|_{L^2}$ for $u\in D(A)$. 
Since $\|(I+A)u\|_{L^2}\lesssim \|u\|_{H^2}$ for $u\in D(A)$, it follows that the statement in \eqref{IA-eqi-norms} holds for $s = 2$. The intermediate case for $s \in (0, 2)$ follows by the real interpolation between the two endpoint cases $s=0$ and $s=2$. 

\eqref{az-bilinear-H1} implies $\|(I+A)^{\frac{1}{2}}v\|_{L^2} \sim \|v\|_{H^1}$ for $v\in D(A)$, i.e., the two norms 
$\|\cdot\|_{D(A^{\frac12})} $ and $ \|\cdot\|_{H^1}$ are equivalent on $D(A)$. Since $D(A)$ is dense in $H^1(D)^d \cap X$, it follows that 
$$
D(A^{\frac12}) 
= \mbox{closure of $D(A)$ under the norm $\|\cdot\|_{D(A^{\frac12})}\sim  \|\cdot\|_{H^1}$} = H^1(D)^d \cap X .
$$
Therefore, for $s\in(0,1)$ we have 
\begin{align*}
D(A^{\frac{s}{2}}) 
&= \mbox{closure of $D(A^{\frac12}) $ under the norm $\|\cdot\|_{D(A^{\frac{s}{2}})}$} \\
&= \mbox{closure of $H^1(D)^d \cap X$ under the norm $\|\cdot\|_{H^s}$} 
= H^s(D)^d \cap X .
\end{align*}
%
The statement in \eqref{IA-eqi-norms-dual} follows from the duality argument. 
\hfill\endproof
\bigskip

\noindent{\bf Proof of Lemma~\ref{lem-A:sg-sta}.} 
Let $\widetilde E(t):=e^{-t(I+A)}$ be the semigroup generated by $I+A$, which has a compact inverse operator. {Then \cite[Lemma B.9]{kruse2014strong} is applicable} and gives the estimates in \eqref{A:sg-sta1}--\eqref{A:sg-sta4} with $E(t)$ replaced by $\widetilde E(t)$ therein. Since $E(t)=e^t \widetilde E(t)$ and $t\in(0,T]$, it follows that \eqref{A:sg-sta1}--\eqref{A:sg-sta4} also hold for $E(t)$. 
%
\hfill\endproof

\bigskip
\noindent{\bf Proof of Lemma~\ref{lem:FEM-errs-w}.} 
Subtracting \eqref{pde-wh} from \eqref{pde-w} yields
\begin{equation}
\label{w-err-eqn}
\left \{
\begin{aligned} 
(z(w - w_h), v_h) + 2\big(\D(w - w_h), \D( v_h)\big) - (p - p_h,  \nabla \cdot v_h)&= 0 &&  \forall\, v_h \in V_h ,\\
(\nabla\cdot (w-w_h), q_h)&=0&& \forall \, q_h \in Q_h.
\end{aligned}
\right .
\end{equation}
Let $e_h = P_{X_h}w - w_h \in X_h$. Choosing $(v_h, q_h) = (e_h, P_{Q_h} p - p_h) \in X_h \times Q_h$ in \eqref{w-err-eqn}, where  $P_{Q_h}$ denotes the $L^2$-orthogonal projection onto $Q_h$, we have
\begin{align} 
z\|e_h\|_{L^2}^2 + 2\|\D( e_h)\|_{L^2}^2 = -2\big(\D(w - P_{X_h}w), \D(e_h)\big) + (p - P_{Q_h} p, \nabla \cdot e_h) .
\end{align}
By considering the imaginary and the real parts of the above equality, separately, we obtain
\begin{equation*}
\left \{
\begin{aligned}  
|{\rm Im} (z)| \|e_h\|_{L^2}^2  \lesssim \varepsilon ^{-1}\|w - P_{X_h}w\|_{H^1}^2 + \varepsilon \|\D( e_h)\|_{L^2}^2 + \|p - P_{Q_h} p\|_{L^2},\\
{\rm Re} (z)\|e_h\|_{L^2}^2 + 2\|\D( e_h)\|_{L^2}^2 \lesssim  \varepsilon^{-1}\|w - P_{X_h}w\|_{H^1}^2 + \varepsilon \|\D( e_h)\|_{L^2}^2 + \|p - P_{Q_h} p\|_{L^2} , 
\end{aligned}
\right .
\end{equation*}
where $\varepsilon>0$ can be arbitrarily small. 
For $z \in \Sigma_\phi:=\{1+z'\in\C: |{\rm arg}(z')|<\phi\}$, with some $\phi\in(0,\pi)$, we have $|z|\gtrsim 1$ 
and $|{\rm Re}(z)| \lesssim |{\rm Im}(z)|$. 
By using \eqref{Fortin-Pih-conv}, \eqref{app-pro-q},  \eqref{H2-app-PXh2} and Korn's inequality $\| u\|_{H^1}^2 \lesssim \|u\|_{L^2}^2+\|\D(u)\|_{L^2}^2$, choosing a sufficiently small $\varepsilon$ in the above estimates, we can derive the following result:    
\begin{align} \label{w-err-L2H1-z}
|z|^{\frac12} \|P_{X_h}w - w_h\|_{L^2} + \|\nabla (P_{X_h}w - w_h)\|_{L^2}
&\lesssim \|w - P_{X_h}w\|_{H^1} + \|p - P_{Q_h} p\|_{L^2} \notag \\
&\lesssim h(\|w\|_{H^2} +  \|p\|_{H^1}), 
\end{align}
which, together with triangle inequality  and \eqref{H2-app-PXh2}, implies that 
\begin{align}\label{w-err-H1}
\|\nabla (w - w_h)\|_{L^2} \lesssim h(\|w\|_{H^2} +  \|p\|_{H^1}).
\end{align}

In order to obtain an estimate for $\| w - w_h\|_{L^2} $, we consider the following duality argument. Let $\phi \in \DD(A)$ and $\phi_h \in X_h$ be the solution of the linear Stokes equations
\begin{equation}\label{phi-dual}
\left \{
\begin{aligned}  
z\phi - \nabla \cdot \T(\phi, \eta) &= w - w_h \qquad &&{\rm in}\, \, D,\\
\nabla \cdot \phi &= 0 \qquad &&{\rm in}\, \, D,\\
\T(\phi, \eta)\n &= 0 \qquad &&{\rm on}\, \, \partial D,
\end{aligned}
\right .
\end{equation}
and its finite element weak formulation 
\begin{equation}
\label{pde-phih}
\left \{
\begin{aligned} 
(z\phi_h, v_h) +  2\big(\D( \phi_h), \D( v_h)\big) -(\eta_h,  \nabla \cdot v_h)&= (w -w_h, v_h) &&  \forall \, v_h \in V_h ,\\
(\nabla\cdot \phi_h, q_h)&=0&& \forall \, q_h \in Q_h .\\
\end{aligned}
\right.
\end{equation}
Testing \eqref{phi-dual} by $\phi \in X$ and considering the imaginary and real parts of the results, we can obtain the following result similarly as \eqref{w-err-L2H1-z}:  
\begin{align*}
|z|\|\phi\|_{L^2}^2 + 2\|\nabla \phi\|_{L^2}^2  \le C |z|^{-1}\|w - w_h\|_{L^2}^2 
+ \frac12 |z|\|\phi\|_{L^2}^2 ,
\end{align*}
which implies
\begin{align}
|z|\|\phi\|_{L^2} \lesssim \|w - w_h\|_{L^2}.
\notag
\end{align}
{Through the above estimate and the $H^2$ estimate of the Stokes equations in \eqref{stokes-sta-H2}, we obtain}
\begin{align}
\label{phi-sta-H2}
\|\phi\|_{H^2} + \|\eta\|_{H^1} \lesssim \|w - w_h\|_{L^2} + (1+|z|)\|\phi\|_{L^2}
\lesssim \|w - w_h\|_{L^2} , 
\end{align}
where the last inequality uses property $|z|\gtrsim 1$ for $z \in \Sigma_\phi$ and the estimate of $|z|\|\phi\|_{L^2}$ above. 
Similar to \eqref{w-err-L2H1-z}--\eqref{w-err-H1}, we have
\begin{align} \label{phi-err-H1}
|z|^{\frac12}\|P_{X_h}\phi - \phi_h\|_{L^2} +\|\nabla (\phi - \phi_h)\|_{L^2}
\lesssim h(\|\phi\|_{H^2} +  \|\eta\|_{H^1}) 
\lesssim h\|w - w_h\|_{L^2},
\end{align}
where the last inequality follows from \eqref{phi-sta-H2}. 

Since $(\nabla \cdot v_h, P_{Q_h} \eta) = 0$ for $v_h \in X_h$, it follows that 
\begin{align*}
- ( \eta, \nabla \cdot v_h) = 
- ( \eta - P_{Q_h}  \eta, \nabla \cdot v_h)  \quad \forall \, v_h \in X_h.
\end{align*}
Since $\phi_h \in X_h$, we can derive the following equation from \eqref{w-err-eqn}: 
\begin{align*} 
(z\phi_h, w -w_h) +  2\big(\D( \phi_h), \D(w -w_h)\big) - (\nabla \cdot \phi_h,  p -P_{Q_h} p)&= 0.
\end{align*}
Then, testing the first relation of \eqref{phi-dual} by $w - w_h$ and using the two equations above, 
by utlizing \eqref{app-pro-q} and \eqref{H2-app-PXh2}--\eqref{H2-app-PXh1} we obtain 
\begin{align*}
\|w -w_h\|_{L^2}^2 =\,& z(\phi-\phi_h, w -w_h) +  2\big(\D(\phi-\phi_h), \D (w -w_h)\big)\\
& - (\nabla \cdot(\phi-\phi_h), p - P_{Q_h} p)  - (\eta - P_{Q_h} \eta, \nabla \cdot (w -w_h))\\
\lesssim\,& |z|^{\frac 12}\|\phi-\phi_h\|_{L^2} |z|^{\frac 12}\|w -w_h\|_{L^2} 
+ \|\nabla(\phi-\phi_h)\|_{L^2} \|\nabla (w -w_h)\|_{L^2}\\
&
+ h^2\|\phi\|_{H^2} \|p\|_{H^1} 
+ h\|\eta\|_{H^1} \|\nabla ( w -w_h)\|_{L^2}\\
\lesssim\,& (h^2z+h|z|^{\frac 12})\|w -w_h\|_{L^2}^2
+ h\|w - w_h\|_{L^2}\|\nabla (w -w_h)\|_{L^2} \\
&+   h^2 (\|w\|_{H^2} +  \|p\|_{H^1})\|w - w_h\|_{L^2} 
\qquad\quad \mbox{(here \eqref{phi-sta-H2} and \eqref{phi-err-H1} are used)}\\
\lesssim\,& (h^2+h^3|z|^{\frac12}) \|w - w_h\|_{L^2}(\|w\|_{H^2} +  \|p\|_{H^1}) . 
\qquad \mbox{(here \eqref{w-err-L2H1-z} are used)} \notag
\end{align*}
When $|z|\le h^{-2}$, the inequality above reduces to 
\begin{align}\label{w-err-L2-z}
\|w - w_h\|_{L^2} \lesssim h^2(\|w\|_{H^2} +  \|p\|_{H^1}).
\end{align}
When $|z|\ge h^{-2}$, \eqref{w-err-L2H1-z} immediately implies \eqref{w-err-L2-z}. 

Combining the estimates \eqref{w-err-L2H1-z} and \eqref{w-err-L2-z}, we obtain \eqref{pde-w-err}.
The proof of \eqref{reg-w-err} is similar as that for \eqref{phi-sta-H2}. 
\hfill\endproof

\section*{Appendix C: Proof of Lemma \ref{Lemma:uh-energy}}$\,$
\renewcommand{\theequation}{C.\arabic{equation}}
\renewcommand{\thelemma}{C.\arabic{lemma}}
\setcounter{equation}{0}

\begin{proof}
By iterating with respect to the time levels, the full discrete method in \eqref{fully-sFEM} can be rewritten as
\begin{equation}
\label{fully-sFEM-sum2}
u_h^n
=  \bar  E_{h,\tau}^n P_{X_h}u_0 +  \tau\sum_{i = 0}^{n-1} \bar E_{h, \tau}^{n-i} P_{X_h}f(t_{i + 1})
+ \sum_{i = 0}^{n-1} \bar E_{h, \tau}^{n-i} P_{X_h} [B(u_h^i) \Delta W_{i+1}]. 
\end{equation} 
The stability of $\bar E_{h,\tau}^n$, $1 \le n \le N$ together with \eqref{Ito_isometry} and \eqref{E_ht:sg-sta} yields
\begin{align}
\E\|u_h^n\|_{H^\frac 12}^2 
\lesssim&\, \E\|\bar E_{h,\tau}^n (I+A_h)^{\frac 14} P_{X_h}u^0\|_{L^2}^2 \notag\\
&\,+  \E\Big(\tau \sum_{i = 0}^{n-1}\|(I+A_h)^{\frac 14}\bar E_{h,\tau}^{n - i} P_{X_h} f(t_{i+1})\|_{L^2} \Big)^2 \notag \\
&\,+ \E \Big\|\sum_{i = 0}^{n-1} (I+A_h)^{\frac 14} \bar E_{h,\tau}^{n - i} P_{X_h} [B(u_h^i)\Delta W_{i+1}]\Big\|_{L^2}^2\notag\\
\lesssim&\, \E \|u^0\|_{H^{\frac12}}^2 
+  \E\Big(\tau \sum_{i = 0}^{n-1} t_n^{-\frac14}\|\bar E_{h,\tau}^{n - i} P_{X_h} f(t_{i+1})\|_{L^2} \Big)^2 \notag \\
&\,+ \tau\sum_{i = 0}^{n-1} \E\| (I+A_h)^{\frac 14} \bar E_{h,\tau}^{n - i} P_{X_h} B(u_h^i)\|_{ \L_2^0}^2\notag\\
\lesssim&\,  \big(1+ \E \|u^0\|_{H^{\frac12}}^2 \big) 
+ \tau\sum_{i = 0}^{n-1} t_{n-i}^{-\frac12}\E\|B(u_h^i)\|_{ \L_2^0}^2
\notag\\
\lesssim&\,  \big(1+ \E \|u^0\|_{H^{\frac12}}^2 \big) 
+ \tau\sum_{i = 0}^{n-1} (t_n - t_i)^{-\frac12}(1+\E\|u_h^i\|_{H^\frac12}^2).
\quad \mbox{(here \eqref{ass-con-AQ} is used)} 
\end{align}
By using the discrete version of generalized Gronwall's inequality in \cite[Lemma A.4]{kruse2014strong} (and mathematical induction on $n$), we obtain
\begin{align} \label{SPDE-uh-sta-L2}
\max_{1 \le n \le N}  \E \|u_h^n\|_{H^\frac 12}^2  \lesssim  1+ \E \|u^0\|_{H^{\frac12}}^2 .
\end{align}
This proves the desired estimate for the first term in \eqref{sFEM-uhn-sta}.

Setting $v_h= u_h^n \in X_h$ in \eqref{fully-sFEM-weak} and applying the identity $2(a-b)a  = |a|^2-|b|^2+|a-b|^2 $ for $a, b \in \R$ yield 
\begin{align}\label{uhn-sta-err}
\frac12 \|u_h^n&\|_{L^2}^2 + \frac12\|u_h^n - u_h^{n-1}\|_{L^2}^2
+2\tau\|\D( u_h^n)\|_{L^2}^2\\
\le\, &  \frac12 \|u_h^{n-1}\|_{L^2}^2 + \frac12 \tau \|f(t_n)\|_{L^2}^2
+ \frac12\tau \|u_h^n\|_{L^2}^2
+  (B(u_h^{n-1})\Delta W_n, u_h^n),  \notag
\end{align}
where the last term can be estimated by applying the property of a martingale $$\E(B(u_h^{n-1})\Delta W_n, u_h^{n-1}) = 0$$
and using \eqref{lem:B-noise-sta2} and \eqref{SPDE-uh-sta-L2} as follows: 
\begin{align}
\E(B(u_h^{n-1})\Delta W_n, u_h^n) 
=\,&\E(B(u_h^{n-1})\Delta W_n, u_h^n - u_h^{n-1}) \\
\le\, & C\E\|B(u_h^{n-1})\Delta W_n\|_{L^2}^2 + \frac14 \E\|u_h^n - u_h^{n-1}\|_{L^2}^2\notag\\
\le\, & C\tau \E\|B(u_h^{n-1})\|_{ \L_2^0}^2 + \frac14 \E\|u_h^n - u_h^{n-1}\|_{L^2}^2\notag\\
\le\, &   C\tau \big(1 + \E\|u_h^{n-1}\|_{H^{\frac12}}^2 \big) + \frac14 \E\|u_h^n - u_h^{n-1}\|_{L^2}^2\notag\\
\le\, &  C\tau \big(1 + \E\|u_h^0\|_{H^1}^2 \big)+ \frac14 \E\|u_h^n - u_h^{n-1}\|_{L^2}^2\notag.
\end{align}
By using Korn's inequality (cf. \cite[Theorem 2.4]{john2002analysis}) 
$\|u\|_{L^2}^2 + \|\D(u)\|_{L^2}^2\ge \theta(\|u\|_{L^2}^2 + \|\nabla u\|_{L^2}^2)$, where $\theta > 0$ is some constant, we obtain 
\begin{align}
\E\|u_h^N\|_{L^2}^2 + \sum_{n = 1}^{N}\E\|u_h^n - u_h^{n-1}\|_{L^2}^2
+\tau\sum_{n = 1}^{N}\E\|\nabla u_h^n\|_{L^2}^2
\lesssim  1 + \E\|u_h^0\|_{H^1}^2
+ \tau \sum_{n = 1}^N\E\|u_h^n\|_{L^2}^2 . 
\end{align}
Then applying the discrete Gronwall's inequality yields the desired estimate for the last two terms in \eqref{sFEM-uhn-sta}.
This proves Lemma \ref{Lemma:uh-energy}. 
\end{proof}

\bibliographystyle{abbrv}
\bibliography{NS-reference}

\end{document}